\documentclass{article}
\usepackage{graphicx} % Required for inserting images

\usepackage{amsfonts}
\usepackage{amstext}

\usepackage{stmaryrd}
\usepackage{amsthm}
\usepackage{amsmath}

\usepackage{mathrsfs}

\usepackage{graphics}

\usepackage{mathtools}
\usepackage{amssymb}
\usepackage{commath}
\usepackage{bigints}
\usepackage{romanbar}
\usepackage{enumerate}
\usepackage{etaremune}
\usepackage{multirow}
\usepackage{graphicx}

\usepackage{tikz}

\usepackage{hyperref}
\usepackage{fullpage}

\usepackage{algorithm}
\usepackage{algpseudocode}

\usepackage{comment}

\usepackage[shortlabels]{enumitem}

\usepackage{color}
\usepackage{todonotes}

\DeclareMathOperator*{\argmin}{arg\,min}

\newcommand{\vmeas}{ {\boldsymbol{w}} }
\newcommand{\vx}{\mathbf{x}}

\newcommand{\vz}{\mathbf{z}}
\newcommand{\vu}{\mathbf{u}}

\newcommand{\vv}{\mathbf{v}}
\newcommand{\vw}{\mathbf{w}}
\newcommand{\vn}{\mathbf{n}}
\newcommand{\vf}{\mathbf{f}}
\newcommand{\vg}{\mathbf{g}}

\newcommand{\valpha}{\boldsymbol{\alpha}}
\newcommand{\vlambda}{\boldsymbol{\lambda}}
\newcommand{\vvarphi}{\boldsymbol{\varphi}}
\newcommand{\vpsi}{\boldsymbol{\psi}}
\newcommand{\veta}{\boldsymbol{\eta}}
\newcommand{\vpi}{\boldsymbol{\pi}}

\newcommand{\veps}{\boldsymbol{\varepsilon}}
\newcommand{\vphi}{\boldsymbol{\phi}}

\newcommand{\vzeta}{\boldsymbol{\zeta}}

\newcommand{\vdiv}{\textbf{div}}

\newcommand{\cS}{\mathcal{S}}
\newcommand{\cXs}{\mathcal{X}^s}
\newcommand{\cXsz}{\mathcal{X}_0^s}
\newcommand{\cKs}{\mathcal{K}^s}
\newcommand{\Bd}{\Gamma}

\newcommand{\mean}{\text{mean}}

\theoremstyle{plain}% default
\newtheorem{theorem}{Theorem}[section]

\newtheorem{remark}{Remark}[section]

\title{A near-optimal recovery algorithm for the Stokes equations with incomplete information on the boundary conditions
\thanks{This research was supported in part by the NSF Grant DMS-2409807 (AB) and the NSERC Discovery Grant RGPIN-2021-04311 (DG).}}

\author{Andrea Bonito$^1$ and Diane Guignard$^2$ \\
\small
$^1$Texas A\&M University, Department of Mathematics, College Station, TX 77843, USA (bonito@tamu.edu) \\
\small
$^2$Department of Mathematics and Statistics, University of Ottawa, Canada (dguignar@uottawa.ca)}

\date{\today}

\begin{document}

\maketitle

\begin{abstract}
We address the problem of numerically approximating the velocity and pressure governed by the Stokes system when the boundary conditions are only partially known and thus do not uniquely determine the velocity-pressure couple. We propose an algorithm that takes advantage of available linear measurements of the velocity and pressure to construct a numerical approximation. This approximation is guaranteed to be near-optimal in the sense that it approximates the velocity-pressure couple that minimizes, in the energy norm, the distance to all other solutions satisfying the measurements and the Stokes system.
\end{abstract}

\section{Introduction} \label{sec:intro}

We consider the Stokes system 
\begin{equation}\label{e:Stokes}
-\vdiv(\nabla \vu + \nabla \vu ^T) + \nabla p = \vf \qquad \textrm{and}\qquad \textrm{div}(\vu) = 0 \qquad \textrm{in }\Omega
\end{equation}
relating the velocity $\vu:\Omega \rightarrow \mathbb R^d$ with the pressure $p:\Omega \rightarrow \mathbb R$ of a fluid in a bounded domain $\Omega \subset \mathbb R^{d}$, $d=2,3$, with Lipschitz boundary $\partial \Omega$. 
Here, $\vf:\Omega \rightarrow \mathbb R^{d}$ is a given force applied to the fluid. Typically, the Stokes system is supplemented by boundary conditions on $\partial \Omega$ that either set the velocity, the force $(\nabla \vu + \nabla \vu ^T){\vn} -p{\boldsymbol \vn}$, where $\vn$ stands for the outward pointing normal to $\Omega$, or a combination of both. We refer for instance to \cite{temam2024navier,galdi2011introduction,sohr2012navier,bernardi2024mathematics} for details.

Instead, this work is concerned with situations where the type of boundary condition or its value is (partially) unknown.
Missing boundary conditions in complex systems are the result of a lack of understanding of the physics \cite{raj,richards2019appropriate}, the inaccessibility of the actual value to impose at the boundary \cite{grinberg2008outflow}, or the modification/truncations of the physical domain to design efficient numerical approximations \cite{formaggia2002numerical,richards2011appropriate}.
We refer in particular to \cite{xu2021explore,brunton2020machine,duraisamy2019turbulence} for examples where boundary conditions are missing in fluid dynamics, for which the Stokes system considered here is a prototype.

For the moment, consider the case where the velocity is known on $\partial \Omega$, say
\begin{equation}\label{e:dirichlet}
\vu = \vg \qquad \textrm{on } \partial \Omega
\end{equation}
for some $\vg: \partial \Omega \rightarrow \mathbb R^d$ satisfying $\int_{\partial \Omega} \vg \cdot \vn=0$.
Standard results, see Theorem~I.5.1 and Remarks~I.5.1 in \cite{GR86}, guarantee that for $\vg \in H^{\frac{1}{2}}(\partial \Omega)^{d}$ and $\vf \in H^{-1}(\Omega)^d$, there exists a unique solution $(\vu(\vg),p(\vg))\in X$with $\int_{\Omega}p(\vg)=0$, where
$$
X:=\left\{ (\vv,q) \in H^1(\Omega)^d \times L^2(\Omega) \ : \ \int_{\partial \Omega} \vv \cdot \vn = 0\right\}
$$
is equipped with the graph norm; see Section~\ref{ss:Prelim} for the definition of these spaces.
The velocity-pressure couple is also unique when restricting the mean value of the pressure $|\Omega|^{-1}\int_\Omega p = P$ for some given $P\in \mathbb R$.
In that case, the data $(\vg,P)$ and the solution $(\vu(\vg,P),p(\vg,P))\in X$ satisfy
\begin{equation}\label{e:a-priori}
\|\vu(\vg,P),p(\vg,P) \|_X \leq C \left(\| \vf \|_{H^{-1}(\Omega)} + \| \vg \|_{H^{\frac{1}{2}}(\partial \Omega)}+|P|\right)
\end{equation}
for some constant $C$ only depending on $\Omega$; refer to Remarks~I.5.1 and~I.5.3 in \cite{GR86}.
    
In contrast, the velocity and pressure that satisfy the Stokes system \eqref{e:Stokes} cannot be uniquely determined without a boundary condition. 
To alleviate some of this uncertainty, we propose a recovery algorithm that relies on a model class assumption, or prior, to construct the ``best'' velocity-pressure as defined below. It consists of restricting our search for velocities-pressures such that for some $s>1/2$ and some $C_\textrm{data}>0$
\begin{equation}\label{e:model}
(\vu,p) \in K:= \left\{ (\vv,q) \in X \ : \ (\vv,q) \quad \textrm{satisfies }\eqref{e:Stokes}, \ \| \vv \|_{H^s(\partial \Omega)} +\left||\Omega|^{-1}\int_\Omega q\right|\leq C_\textrm{data} \right\}.
\end{equation}
Such velocities possess slightly more smoothness than the standard boundary regularity $H^{\frac{1}{2}}(\partial \Omega)$. However, we observe that, in any case, higher regularity is necessary for any algorithm to construct a meaningful approximation of the exact velocity. We also emphasize that it is not assumed that the target solution of \eqref{e:Stokes} satisfies any specific type of boundary condition and that although the pressures in $K$ must satisfy $|\Omega|^{-1} \int_\Omega p \leq C_\textrm{data}$, the mean value of the pressure is not prescribed and therefore becomes part of the recovery process.

In addition, we assume that measurements $\omega_i:=\lambda_i(\vu,p)\in \mathbb R$, $i=1,\ldots,m$, of some $(\vu,p)$ satisfying \eqref{e:Stokes} and~\eqref{e:model} are available, where the $\lambda_i$ are fixed measurement functionals on $X$. 
We then gather all the information about the function to be recovered in
\[
K^s_{\vmeas} := \{ (\vv,q) \in K \ :  \ \lambda_i(\vu,q)=\omega_i, i=1,...,m\}. 
\]
The set $K^s_{\vmeas}$ is generally not reduced to one element (the solution to recover).
In fact, any element in $K^s_{\vmeas}$ is a possible recovery for $(\vu,p)$.
We say that $(\vu^*,p^*)$ is the best recovery in $K^s_\vmeas$ if
\begin{equation}\label{e:opt}
\sup_{(\vv,q) \in K^s_{\vmeas}} \| (\vv,q)-(\vu^*,p^*)\|_X \leq \inf_{(\vu,p)\in X} \sup_{(\vv,q) \in K^s_{\vmeas}} \| (\vv,q) - (\vu,p) \|_X=:R(K^s_\vmeas)_X=:R(K^s_\vmeas).
\end{equation}
Geometrically, $(\vu^*,p^*)$ is the center of the smallest ball in $X$ (with radius $R(K^s_\vmeas)$) containing $K^s_{\vmeas}$. It is called the Chebyshev center of $K^s_{\vmeas}$ and $R(K^s_\vmeas)$ is the Chebyshev radius \cite{bojanov1994optimal, micchelli1985lectures, micchelli1976optimal, traub1973theory}. The latter serves as a benchmark for recovery algorithms. 

When $\vf=\mathbf{0}$, the Chebyshev center $(\vu^*,p^*)$ in the recovery setting considered here is characterized as element in $K^s_{\vmeas}$ whose norm is minimal \cite{micchelli1985lectures, micchelli1976optimal}. In particular, $K^s_\vmeas = K \cap \textrm{span}(\vphi_1,...,\vphi_m)$, where $\vphi_j$ are appropriate Riesz representers of the measurement functionals $\lambda_j$. They are the solutions to the Stokes system \eqref{e:Stokes} with $\vf=\mathbf{0}$ and satisfy  a fractional partial differential equation on $\partial \Omega$; refer to Section~\ref{sec:algo} for more details. Hence, the Chebyshev center satisfies
\begin{equation}\label{e:rieszrep}
(\vu^*,p^*)= \sum_{j=1}^m \alpha_j^* \vphi_j,
\end{equation}
for some $\valpha^*=(\alpha_j^*)\in\mathbb{R}^m$ determined so that $(\vu^*,p^*)$ satisfies the measurements.
The general case where $\vf \not = \mathbf{0}$ is obtained by shifting expression \eqref{e:rieszrep} by the solution to \eqref{e:Stokes} supplemented by vanishing Dirichlet condition. 

The minimal norm property is exploited in \cite{batlle2025error} to design a kernel method algorithm for the approximation
of an optimal recovery of general, possibly nonlinear, partial differential equations.
In \cite{BBCDDP2023} and later in \cite{BG2024,BDS2025}, the representation~\eqref{e:rieszrep} is advocated to design near-optimal algorithms for the Poisson problem.
In these works, fully practical and near-optimal algorithms based on the approximation of the Riesz representers are proposed.
The challenges that arise in the Stokes system, compared to the Poisson problem, stem from the added complexity and efficiency requirements to determine the Riesz representer $\vphi_j$, the inherent nontrivial compatibility conditions that admissible boundary data must satisfy, and the need to recover the pressure average within the recovery process. 

Inspired by \cite{BBCDDP2023,BG2024}, we consider an algorithm that constructs Riesz representers that are within a tolerance $\varepsilon>0$ in $H^1(\Omega)^d\times L^2(\Omega)$ of the exact Riesz representer. We show that using these approximate Riesz representers in \eqref{e:rieszrep} yields $(\vu^\varepsilon,p^\varepsilon)$ such that
\begin{equation}\label{e:u_near_opt}
\sup_{(\vv,q) \in K^s_{\vmeas}} \| (\vv,q)-(\vu^\varepsilon,p^\varepsilon)\|_X \leq C(\varepsilon) R(K^s_\vmeas),
\end{equation}
where $C(\varepsilon) \to 1^+$ as $\varepsilon \to 0^+$. 
In that case, we say that the algorithm $A:(\varepsilon,\vmeas) \mapsto (\vu^\varepsilon,p^\varepsilon)$ is near-optimal.

We further address the question of near‑optimal performance for quantities of interest $Q:X \rightarrow \mathbb R$ such as drag or lift. In particular, we show that the algorithm $A_Q:(\varepsilon,\vmeas) \mapsto Q(\vu^\varepsilon,p^\varepsilon)\in \mathbb R$ is near‑optimal for the quantity of interest, that is, 
\begin{equation}\label{e:Q_near_opt_intro}
\sup_{(\vv,q) \in K^s_\vmeas} |Q(\vv,q)-Q(\vu^\varepsilon,p^\varepsilon)| \leq C(\varepsilon) \inf_{l \in \mathbb R} \sup_{(\vv,q) \in K^s_\vmeas} |Q(\vv,q)-l|,
\end{equation}
where again $C(\varepsilon)\to 1^+$ as $\varepsilon \to 0$. 

A practical realization of the algorithm is then obtained by approximating the Riesz representers using a finite element method together with a numerical method for fractional differential operator based on the Balakrishnan formula \cite{lunardi2018interpolation} and a sinc quadrature \cite{lund1992sinc}. The proposed numerical approximation of each $\vphi_j$ consists of three successive steps. The first step of the algorithm is the determination of Lagrange multipliers enforcing that $(\vw_j,r_j):=\vphi_j$ satisfies \eqref{e:Stokes} with $\vf=\mathbf{0}$ in $\Omega$. This consists in solving a Stokes system with homogeneous Dirichlet boundary condition. Then, an approximation of $\vw_j|_{\partial\Omega}$ is obtained using a Lagrange multiplier to enforce the compatibility condition in $X$ and solving numerically a fractional partial differential equation on $\partial\Omega$. In the last step, the approximation of $\vphi_j=(\vw_j,r_j)$ inside $\Omega$ is determined solving again a Stokes system in $\Omega$ with prescribed Dirichlet boundary condition $\vw_j|_{\partial\Omega}$.
 
For completeness, we mention that \cite{BJL2025} considers a unique continuation problem for the Stokes system with missing boundary conditions and given velocity measurements in a subdomain, see also \cite{BLLL2025}.
We also remark that the problem considered here relates to an optimal boundary control problem with Stokes constraint. In this context, a Tikhonov regularization is typically added to the objective function in order to stabilize the process and the question of near-optimality is typically not addressed. In \cite{binev2024optimal} it is shown that to guarantee quasi-optimality (\eqref{e:Q_near_opt_intro} holds for a constant not necessarily satisfying $C(\varepsilon)\to 1^+$ as $\varepsilon \to 0^+$), the added stabilization should be of the form $c(\varepsilon) \| \vu \|_{H^{s}(\partial \Omega)}$ with $c(\varepsilon)\to 0^+$ as $\epsilon\to 0^+$. In contrast, the algorithms developed in this work exhibit the stronger near-optimal property. 

In Section~\ref{sec:ORP}, we describe in more detail the concept of an optimal algorithm in the present context and present in Section~\ref{sec:algo} the analysis of conceptual near-optimal algorithms, i.e., producing $(\vu^\varepsilon,p^\varepsilon)$ satisfying \eqref{e:u_near_opt} and \eqref{e:Q_near_opt_intro}. 
A finite element method for the approximation of the Riesz representers is described in Section~\ref{sec:Riesz} followed by numerical illustrations of the performance of the practical algorithms in Section~\ref{sec:numeric}.
But before we start with some preliminaries and notation. 

\subsection{Preliminaries and notations}\label{ss:Prelim}

Let $\Omega\subset\mathbb{R}^d$, $d\in\{2,3\}$, be a bounded domain with Lipschitz boundary $\Gamma:=\partial\Omega$ and let $\mathcal R=\Omega$ or $\mathcal R=\Gamma$.

We denote by  $L^q(\mathcal R)$,  $1\leq q \leq \infty$, the standard Lebesgue spaces of measurable functions. When $q=2$ the space of square integrable functions $L^2(\mathcal R)$ is a Hilbert space with the inner product and associated norm
$$
\langle v,w\rangle_{L^2(\mathcal R)}:=\int_{\mathcal R} vw \quad \text{and} \quad \| v\|_{L^2(\mathcal R)}:= \sqrt{\langle v,v\rangle_{L^2(\mathcal R)}}, \qquad v,w \in L^2(\mathcal R).
$$
For a function $q\in L^1(\Omega)$, we use $\bar q$ to denote its mean (average) defined by
\begin{equation} \label{eq:avrg}
    \bar q:=\mean(q):=\frac{1}{|\Omega|}\int_{\Omega}q.
\end{equation}
Moreover, we introduce the space $L_0^2(\Omega):=\{q\in L^2(\Omega): \,\, \mean(q)=0\}$ of square integrable functions with mean value zero. 

We equip the Sobolev space $H^1(\mathcal R)$ with the standard norm $\|\cdot\|_{H^1(\mathcal R)}$ induced by the inner product
\begin{equation} \label{def:H1_IP}
  \langle v,w\rangle_{H^1(\mathcal R)}:=\int_{\mathcal R}vw+\int_{\mathcal R}\nabla_{\mathcal R} v\cdot\nabla_{\mathcal R} w, \quad v,w\in H^1(\mathcal R), 
\end{equation}
where $\nabla_\Omega:=\nabla$ and $\nabla_{\Gamma}$ is the tangential gradient to $\Gamma$.
We denote by $H^1_0(\Omega)$ the space of $H^1(\Omega)$ functions whose trace vanishes on $\Gamma$, and by $H^{-1}(\Omega)$ the dual of $H_0^1(\Omega)$. The $H^{-1}(\Omega)-H^1_0(\Omega)$ duality product is denoted $\langle \cdot,\cdot\rangle$.

We also need the fractional (interpolation) Sobolev space $H^s(\Gamma)$, $0<s<1$, which are Hilbert spaces with the inner product and associated norm defined for $v,w\in H^s(\Gamma)$ by
\begin{equation} \label{def:Hs}
    \langle v,w\rangle_{H^s(\Gamma)}:= \langle (I-\Delta_{\Gamma})^{s/2}v,(I-\Delta_{\Gamma})^{s/2}w\rangle_{L^2(\Gamma)} \quad \text{and} \quad \|v\|_{H^s(\Gamma)}:=\sqrt{\langle v,v\rangle_{H^s(\Gamma)}}.
    \end{equation}
Here $(I-\Delta_{\Gamma})^{s/2}$ are the fractional powers of the operator $I-\Delta_{\Gamma}$ \cite{lunardi2018interpolation} with $I$ the identity and $\Delta_{\Gamma}$ the Laplace--Beltrami operator. This special choice of inner products, which are equivalent to the standard Slobodeckij inner products, is preferred to facilitate their approximations needed by the proposed recovery algorithm; see Section~2.2 in \cite{BG2024} for more details. We recall that there is $r^*>1$ depending on the regularity of $\Omega$ such that 
\begin{equation}\label{e:r_star}
\| v|_{\Gamma}\|_{H^{\frac{1}{2}+r}(\Gamma)} \leq C\| v \|_{H^{1+r}(\Omega)}, \qquad 1\leq r\leq r^*,
\end{equation}
for some constant $C$ only depending on $\Omega$ and where $v|_{\Gamma}$ denotes the (extension to $H^{1}(\Omega)$ of the) trace operator. When confusion is not possible, we will simply write $v$ for the trace of $v$. 

The inner products and norms of scalar-valued functions are extended to vector-valued functions by a component-wise application. More precisely, if $H$ is a Hilbert space with inner product $\langle\cdot,\cdot\rangle_H$ and associated norm $\|\cdot\|_H$, then for $\vv=(v_1,...,v_d),\vw=(w_1,...w_d)\in H^d$ we define
\begin{equation} \label{def:Multivar}
\langle\vv,\vw\rangle_H:=\sum_{i=1}^d\langle v_i,w_i\rangle_H \quad \text{and} \quad \|\vv\|_H:=\sqrt{\sum_{i=1}^d\|v_i\|_H^2}.
\end{equation}
In particular, we recall that we defined the functional space for the velocity-pressure couple by
\begin{equation} \label{def:SpaceX}
    X = \left\{ (\vv,q) \in H^1(\Omega)^d\times L^2(\Omega) \ : \
\int_{\Gamma} \vv \cdot \vn = 0 \right\}
\end{equation}
with $\|(\vv,q)\|_X:=\sqrt{\|\vv\|_{H^1(\Omega)}^2+\|q\|_{L^2(\Omega)}^2}$.

Similarly, an operator acting on a scalar-valued function will be extended to vector-valued functions by component-wise application. 

For $x\in\mathbb{R}$, $\lfloor x \rfloor$ and $\lceil x \rceil$ denote the greatest integer less than or equal to $x$ and the smallest integer greater than or equal to $x$, respectively.

We end this section with some notational conventions. Given a functional $F:X\rightarrow\mathbb{R}$, we will often write $F(\vv,q):=F\big((\vv,q)\big)$ for $(\vv,q)\in X$ and similarly $\| \vv,q\|_X:= \| (\vv,q)\|_X$. 

\section{Optimal recovery problem} \label{sec:ORP}

We assume that the boundary conditions on $\Gamma$ are not known. This could reflect a lack of knowledge of the boundary data, but also of the actual physics, namely, the type of boundary conditions.
To simplify the presentation, we suppose for now that no information is available on boundary conditions and postpone the discussion of partially known boundary data to Section~\ref{sec:partial}. Not knowing the boundary conditions means \eqref{e:Stokes} does not uniquely determine $\vu$ and $p$.
This lack of knowledge is alleviated by the following assumptions on the velocity-pressure $(\vu^\textrm{ex},p^\textrm{ex})$ to recover. We recall that $r^*$ is defined in \eqref{e:r_star}.

\begin{enumerate}[label={(H\arabic*)}]
    \item \label{h:reg} \textbf{Regularity}: The boundary trace of the velocity satisfies $\vu^\textrm{ex}|_\Gamma \in H^s(\Bd)^d$ for some $\frac 1 2 < s \leq r^*$;
    $$
    \| \vu^{\textrm{ex}}\|_{H^s(\Bd)}+ | \mean(p^{\textrm{ex}}) | \leq C_\textrm{data}.
    $$
    \item \textbf{Measurements}: We have access to measurements of $(\vu^\textrm{ex},p^\textrm{ex})$, namely we know the measurement vector
$$\vmeas:=\vlambda(\vu^\textrm{ex},p^\textrm{ex}):=\Big(\lambda_1(\vu^\textrm{ex},p^\textrm{ex}),\lambda_2(\vu^\textrm{ex},p^\textrm{ex}),\ldots,\lambda_m(\vu^\textrm{ex},p^\textrm{ex})\Big)\in\mathbb{R}^m,$$
where $\lambda_j$, $j=1,...,m$, are linearly independent functionals on $X$.
\end{enumerate}

We note that the regularity and boundedness assumptions \ref{h:reg} are on the trace of the velocity and the pressure average. This indicates that we chose to recover these quantities from the measurements. An alternative approach would be, for example, to recover the boundary force $(\nabla \vu^\textrm{ex} + (\nabla \vu^\textrm{ex}) ^T)\vn -p^\textrm{ex}\vn$ in $H^{-\frac{1}{2}}(\Gamma)^{d}$ (the dual of $H^{\frac{1}{2}}(\Gamma)^{d}$) along with the average velocity, but this approach is less favorable to the design of a practical recovery algorithm because it would involve negative norms on $\Gamma$.

We also anticipate that the recovery algorithm discussed in Section~\ref{sec:algo} explicitly uses the value of $s$, which is a user-prescribed parameter, but not that of $C_\textrm{data}$. This explains why only the dependence on $s$ is indicated for the set gathering all the a‑priori knowledge:
$$
K^s_{\vmeas}:=\{(\vu,p)\in K^s \ : \,\, \vlambda(\vu,p)=\vmeas\},
$$
where
$$
K^s:=\{(\vu,p)\in X: \, (\vu,p)\text{ satisfies } \eqref{e:Stokes} \textrm{ and } \| \vu\|_{H^s(\Bd)} +|\mean(p)|\leq C_{\textrm{data}}\}
$$
is the so-called model class and is compact in $X$.

The \emph{optimal recovery problem} consists of determining the best representative $(\vu^*,p^*)\in X$ for all elements of $K_{\vmeas}^s$, i.e.,
\begin{equation} \label{eqn:up_star}
    (\vu^*,p^*)\in \argmin_{(\vu,p)\in X} \sup_{(\vv,q)\in K_{\vmeas}^s}\|(\vv,q)-(\vu,p)\|_X.
\end{equation}

\subsection{Decomposition of functions in $X$ and associated recovery problem} \label{subsec:splitting}

We define the Stokes operator $\mathcal S: X \rightarrow H^{-1}(\Omega)^d \times L^2(\Omega)$ for $(\vv,q)\in X$ by  
\begin{equation*}
\cS(\vv,q):=  \left(
\begin{array}{c}
   % -\Delta\vv+\nabla q \\
   -\vdiv(2\veps(\vv))+\nabla q \\
   {\rm div}(\vv)
\end{array}
\right),  \quad  \veps(\vv):= \frac 1 2 (\nabla \vv + (\nabla \vv)^T).
\end{equation*}
Given $\vf \in H^{-1}(\Omega)^d$, we consider velocity-pressure couples $(\vu,p)\in X$ satisfying
\begin{equation} \label{def:pb1}
		\cS(\vu,p)=\begin{pmatrix}
		    \vf \\
            0
		\end{pmatrix}
\end{equation}
in the weak sense, that is 
$$
\langle 2\veps( \vu),  \veps(\vv)\rangle_{L^2(\Omega)} - \langle p,\textrm{div}(\vv)\rangle_{L^2(\Omega)} - \langle q,\textrm{div}(\vu)\rangle_{L^2(\Omega)} = \langle \vf, \vv \rangle, \qquad \forall\, (\vv,q) \in H^1_0(\Omega)^d \times L^2_0(\Omega).
$$
We recall that Theorem~5.1 in \cite{GR86}, see also Remarks~5.1 and~5.3 in \cite{GR86}, guarantees the existence and uniqueness of $(\vu_\vf,p_\vf)\in  X \cap \left(H^1_0(\Omega)^d\times L^2_0(\Omega)\right)$ satisfying~\eqref{def:pb1}.
Therefore, it will be convenient to decompose any couple $(\vu,p) \in X$ satisfying \eqref{def:pb1} into two parts and introduce
\begin{equation}\label{def:splitting}
(\vu_{\cS},p_{\cS}) := (\vu,p) - (\vu_{\vf},p_{\vf}).
\end{equation}
Since $\vf$ is known, $(\vu_\vf,p_\vf)$ is uniquely determined and the challenge is to recover $(\vu_{\cS},p_{\cS})\in X$. Notice that the latter satisfies $\mathcal S(\vu_{\cS},p_{\cS})= \bf{0}$ and that if $(\vu,p) \in K^s$, we also have $\|\vu_{\cS}\|_{H^s(\Gamma)} + |\mean(p_{\cS})| \leq C_\textrm{data}$.
In other words,  $(\vu,p)\in K^s$ if and only if $(\vu,p)=(\vu_\vf,p_\vf)+(\vu_\cS,p_\cS)$ for some $(\vu_\cS,p_\cS) \in \cKs$, where
\begin{equation} \label{def:cKs}
    \cKs:=\{(\vv,q)\in \cXs: \, \|\vv \|_{H^s(\Gamma)}+ | \mean(q)| \le C_\textrm{data} \}
\end{equation}
with
\begin{equation} \label{def:space_Xs}
  \cXs:=\left\{(\vv,q)\in X: \,\, \cS(\vv,q)=\mathbf{0} \quad \textrm{and} \quad \vv|_\Gamma \in H^s(\Gamma)^d\right\}\subset X.  
\end{equation}

This indicates that the recovery of $(\vu,p)\in K^s$ is equivalent to the recovery of $(\vu_\cS,p_\cS)\in \cKs$ provided measurements of the latter are accessible. We take advantage of the linearity of the measurements to define
$\vmeas_{\cS}:=\vmeas-\vlambda(\vu_{\vf},p_{\vf})$ so that
$\vmeas_{\cS}=\vlambda(\vu_{\cS},p_{\cS})$ 
are the desired measurements of $(\vu_{\cS},p_{\cS})$. This leads to the definition of 
$$
\mathcal{K}^s_{\vmeas_{\cS}}:=\{ (\vu,p)\in \cKs \ : \ \vlambda(\vu,p) = \vmeas_{\cS}\}.
$$

We are now in a position to restate the optimal recovery problem in terms of $\cKs_{\vmeas_{\cS}}$: given $s>1/2$ and $\vmeas\in\mathbb{R}^m$, find 
\begin{equation} \label{eqn:up_star_S}
    (\vu_{\cS}^*,p_{\cS}^*)\in \argmin_{(\vu,p)\in X}\sup_{(\vv,q)\in \cKs_{\vmeas_{\cS}}}\|(\vv,q)-(\vu,p)\|_X
\end{equation}
and set $(\vu^*,p^*):=(\vu_{\vf},p_{\vf})+(\vu_{\cS}^*,p_{\cS}^*)$ to obtain \eqref{eqn:up_star}.

We end this section noting that thanks to \eqref{e:a-priori} with $\vf=\bf{0}$, the functional space $\cXs$ is a Hilbert space when equipped with the inner product and the associated norm
\begin{equation} \label{eqn:IP_Xs}
\langle(\vv,q),(\vu,p) \rangle_{\cXs} := \langle \vv,\vu\rangle_{H^s(\Gamma)}+\mean(q)\mean(p), \qquad \|\vv,q\|_{\cXs}:=\sqrt{\langle(\vv,q),(\vv,q) \rangle_{\cXs}}.
\end{equation}
Moreover, from the a-priori estimate \eqref{e:a-priori} and the continuous embedding 
$H^{s}(\Gamma)\subset H^{\frac{1}{2}}(\Gamma)$, we infer the existence of some constant $C_s>0$ depending only on $\Omega$ such that $\|\vv,q\|_X \le C_s\|\vv,q\|_{\cXs}$ for $(\vv,q)\in\cXs$ and in particular $\|\vv,q\|_X \le C_s C_\textrm{data}$ for $(\vv,q)\in\cKs$.

\subsection{Partial knowledge of the boundary conditions} \label{sec:partial}

In several applications, the boundary conditions are known on part of the boundary $\Gamma$.
We briefly remark that the strategy developed in this work easily extends to that case.

Consider the decomposition $\Gamma=\Gamma_k\cup\Gamma_u$ of the boundary into two disjoint non-empty portions, where $\Gamma_k$ is the part where the boundary conditions are known while no information about boundary data is available on $\Gamma_u$. To simplify the discussion, we assume that the velocity is known on $\Gamma_k$, i.e., $\vu=\vg_k$ for some known function $\vg_k:\Gamma_k\rightarrow\mathbb{R}$, but the method can be extended to accommodate other types of known boundary conditions.

We are seeking a velocity-pressure couple $(\vu,p)$ such that $\mathcal S(\vu,p)= (\vf,0)^T$ and $\vu = \vg_k$ on $\Gamma_k$.
Taking advantage of the splitting strategy \eqref{def:splitting}, we define $(\vu_{\vf},p_{\vf})\in X$ as satisfying $\mathcal S(\vu_{\vf},p_{\vf})= ({\bf \vf}, 0)^T$
together with the boundary conditions
$$ 
\vu = \vg_k  \qquad \textrm{on }\Gamma_k \qquad \textrm{and}\qquad  2\veps(\vu)\vn -p\vn = {\bf 0} \quad \textrm{on }\Gamma_u.
$$
Note that since we impose (homogeneous) Neumann boundary conditions on $\Gamma_u$, the pressure is uniquely determined in $Q=L^2(\Omega)$. The recovery problem then reduces to finding $(\vu_{\cS},p_{\cS})$ in $\cXs$ satisfying the modified measurement vector $\vmeas_{\cS}:=\vmeas-\vlambda(\vu_{\vf},p_{\vf})$, where now
$$
\cXs:=\left\{(\vv,q)\in X: \,\, \cS(\vv,q)=\mathbf{0}, \,\, \vv|_{\Gamma_k}=\mathbf{0}, \,\, \vv|_{\Gamma_u}\in \left[H^s(\Gamma_u)\cap H_{00}^{\frac{1}{2}}(\Gamma_u)\right]^d \right\}
$$
and $H_{00}^{\frac{1}{2}}(\Gamma_u)$ denotes the space of $H^{\frac{1}{2}}(\Gamma_u)$ functions whose extension by zero to $\Gamma$ belongs to $H^{\frac{1}{2}}(\Gamma)$. Note that the latter reduces to $H^{\frac{1}{2}}(\Gamma_u)$ when $\Gamma_u$ is a closed hyper-surface, which occurs when the closures of $\Gamma_k$ and $\Gamma_u$ are disjoint. 

Unless specified otherwise, in the rest of the paper, we assume that the boundary condition is unknown on the entire boundary $\Gamma$, i.e., $\Gamma_k=\emptyset$.

\section{Near optimal algorithm} \label{sec:algo}

In the setting considered in this work, the optimization problem \eqref{eqn:up_star_S} has a unique solution $(\vu_\cS^*,p_\cS^*)$, which is the center of the smallest ball $B$ in $X$ containing $\cKs_{\vmeas_{\cS}}$. It is called the Chebyshev center of $\cKs_{\vmeas_{\cS}}$ and the best recovery performance is the radius $R(\cKs_{\vmeas_{\cS}})$ of $B$, see \eqref{e:opt} which serves as benchmark for recovery algorithms. Moreover, in the present setting, $(\vu_\cS^*,p_\cS^*) \in \mathcal{K}_{\vmeas_{\cS}}^s$ is the element with minimal norm in $\cXs$:
\begin{equation}\label{e:minimal_norm}
    (\vu_\cS^*,p_\cS^*)= \argmin_{(\vv,q)\in \cKs_{\vmeas_{\cS}}} \| \vv,q\|_{\cXs}
\end{equation}
and has thus the following favorable representation for its numerical approximation
\begin{equation} \label{eqn:opt_recovery}
    (\vu_{\cS}^*,p_{\cS}^*) = \sum_{j=1}^m\alpha_j^*\vphi_j.
\end{equation}
Here, for each $j=1,2,\ldots,m$, $\vphi_j$ is the Riesz representer of $\lambda_j$ in the Hilbert space $\cXs\subset X$, i.e.,  $\vphi_j\in \cXs$ solves
\begin{equation} \label{eqn:Riesz}
    \langle \vphi_j,\veta \rangle_{\cXs} = \lambda_j(\veta) \qquad \forall\,\veta\in \cXs.
\end{equation}
Moreover, $\valpha^*=(\alpha_j^*)\in\mathbb{R}^m$ is the solution to
\begin{equation} \label{eqn:LS_Gram}
    G\valpha = \vmeas_{\cS},
\end{equation}
where $G=(g_{ij})\in\mathbb{R}^{m\times m}$ is the Gram matrix with entries $g_{ij} := \lambda_i(\vphi_j) = \langle \vphi_i,\vphi_j \rangle_{\cXs}$, $i,j=1,2,\ldots,m$.

The idea, introduced in \cite{BBCDDP2023} and further considered in \cite{BG2024,BDS2025}, is to approximate the Riesz representers $\{\vphi_j\}_{j=1}^m$ using a finite element method. The proposed strategy is summarised in Algorithm~\ref{alg:OR}.

\begin{algorithm}
\small
\caption{Recovery Algorithm} \label{alg:OR}
\begin{algorithmic}[1]
\Require Tolerances $\varepsilon_1>0$, $\varepsilon_2>0$ % $\vf$
\Ensure Approximation $(\widehat\vu,\widehat p)$ of $(\vu^*,p^*)$
% $\diamond$
\State Compute an approximation $(\widehat \vu_{\vf},\widehat p_{\vf})$ of the solution $(\vu_{\vf},p_{\vf})\in H_0^1(\Omega)^d\times L_0^2(\Omega)$ of \eqref{def:pb1} such that
\begin{equation} \label{eqn:eps1}
    \|(\vu_{\vf},p_{\vf})-(\widehat\vu_{\vf},\widehat p_{\vf})\|_X\le\varepsilon_1;
\end{equation}
\State Set $\widehat\vmeas_{\cS} := \vmeas-\vlambda(\widehat \vu_{\vf},\widehat p_{\vf})$;
\State For each $j\in\{1,2,\ldots,m\}$, compute an approximation $\widehat\vphi_j$ of the solution $\vphi_j\in\cXs$ of \eqref{eqn:Riesz} such that
\begin{equation} \label{eqn:eps2}
    \|\vphi_j-\widehat\vphi_j\|_X\le\varepsilon_2;
\end{equation}
\State Set $\widehat G=(\widehat g_{ij})\in\mathbb{R}^{m\times m}$ with $\widehat g_{ij}:= \lambda_i(\widehat \vphi_j)$ for $i,j=1,2,\ldots,m$;
\State Solve
    \begin{equation} \label{eqn:LS_Gram_approx}
        \widehat G\widehat\valpha = \widehat\vmeas_{\cS}
    \end{equation}
    and set
    \begin{equation} \label{eqn:opt_recovery_approx}
        (\widehat \vu,\widehat p) := (\widehat \vu_{\vf},\widehat p_{\vf})+(\widehat \vu_{\cS},\widehat p_{\cS}), \quad  \text{where} \quad (\widehat \vu_{\cS},\widehat p_{\cS}):=\sum_{j=1}^m\widehat\alpha_j\widehat\vphi_j.
    \end{equation} 
\end{algorithmic}
\end{algorithm}

We have
\begin{theorem}[Near optimal algorithm] \label{thm:OR}
    Let $(\widehat \vu,\widehat p)$ be the output of Algorithm~\ref{alg:OR}. If $\varepsilon_2$ is small enough so that $\widehat G$ is invertible, then
    \begin{equation}\label{e:approx_center}
    \|(\vu^*,p^*)-(\widehat \vu,\widehat p)\|_X\le \varepsilon ,
    \end{equation}
    where $\varepsilon>0$ satisfies $\varepsilon \in O(\varepsilon_1+\varepsilon_2)$ and depends only on $\varepsilon_1$, $\varepsilon_2$, $m$, $\{\lambda_i\}_{i=1}^m$, $\{\vphi_i\}_{i=1}^m$, $G$, $\widehat G$, $C_\textrm{data}$, and $C_s$.
    In particular, we have
    $$
    \sup_{(\vv,q) \in K^s_\vmeas} \|(\vv,q)-(\widehat \vu,\widehat p)\|_X\le R(K_{\vmeas}^s) + \varepsilon.
    $$
\end{theorem}

The proof of this result directly follows the argument provided in \cite[Theorem 3.1]{BBCDDP2023} for the Poisson problem, see also \cite{BDS2025}, but using the bound $\|\vu_{\cS}^*,p_{\cS}^*\|_X \le C_s C_\textrm{data}$ for the solution of \eqref{e:minimal_norm} obtained from the a-priori estimate \eqref{e:a-priori} instead of its Poisson counterpart. It is therefore omitted for the sake of brevity. In contrast, an optimal recovery result for a quantity of interest $Q \in X^\#$ - the dual of $X$ - appears to be not directly available in the literature. The next result guarantees the optimality relation \eqref{e:Q_near_opt}.

\begin{theorem}[Near optimal for quantity of interest] \label{thm:OR_Q}
    The unique optimal solution $(\vu^*,p^*) \in X$ to \eqref{eqn:up_star} satisfies $(\vu^*,p^*)\in K_{\vmeas}^s$ and the optimality relation
    \begin{equation}\label{e:Q_opt}
\sup_{(\vv,q) \in K^s_\vmeas} |Q(\vv,q)-Q(\vu^*,p^*)| = \inf_{l \in \mathbb R} \sup_{(\vv,q) \in K^s_\vmeas} |Q(\vv,q)-l|=: R(K^s_\vmeas;Q).
\end{equation}
Moreover, under the assumption of Theorem~\ref{thm:OR}, the output $(\widehat \vu,\widehat p)$ of Algorithm~\ref{alg:OR} satisfies
\begin{equation}\label{e:Q_near_opt}
\sup_{(\vv,q) \in K^s_\vmeas} |Q(\vv,q)-Q(\widehat \vu,\widehat p)| \leq R(K^s_\vmeas;Q)+\varepsilon\|Q\|_{X^{\#}}, \qquad \text{where} \quad \|Q\|_{X^\#}:=\sup_{(\vv,q)\in X}\frac{|Q(\vv,q)|}{\|\vv,q\|_X}.
    \end{equation}
\end{theorem}
\begin{proof}
To prove \eqref{e:Q_opt}, we follow \cite{FH2025} where this result is derived for a different model class assumption. 

We first derive the optimality result for $Q(\vu_\cS^*,p_\cS^*)$ in $\cKs_{\vmeas_{\cS}}$, where $\vzeta^*_\cS:=(\vu_\cS^*,p_\cS^*) \in \cKs_{\vmeas_{\cS}}$ is given by \eqref{eqn:up_star_S}. We assume that $\cKs_{\vmeas_{\cS}}$ does not reduce to a point, in which case the proof is trivial. In particular, this implies that $\|\vzeta^*_\cS\|_{\cXs}<C_\textrm{data}$.

We denote by $Q_{\mathcal N}: \mathcal N \rightarrow \mathbb R$ the restriction of $Q$ to the null space $\mathcal N:=\{ (\vv,q) \in \cXs \ : \ \vlambda(\vv,q)=0\}$ and by $Q^*_{\mathcal N}$ its adjoint.
Let $\vpsi$ be an eigenfunction associated with the largest eigenvalue of $Q^*_{\mathcal N}Q_{\mathcal N}: \mathcal N \rightarrow \mathcal N$ and normalized such that $\|\vpsi\|_{\cXs}^2 = C^2_\textrm{data}- \| \vzeta^*_\cS \|_{\cXs}^2> 0$.
 
Observe that $\vzeta^*_\cS \pm \vpsi \in \cXs$ and 
$
\vlambda(\vzeta^*_\cS \pm \vpsi) =\vlambda(\vzeta^*_\cS)\pm\vlambda(\vpsi) = \vlambda(\vzeta^*_\cS) = \vmeas_{\cS}$.
Since, according to \eqref{e:minimal_norm}, $\vzeta^*_\cS$ is the element of minimal norm in $\cXs$ satisfying the measurements, we deduce
$$
\| \vzeta^*_\cS \pm \vpsi \|^2_{\cXs} = \| \vzeta^*_\cS \|_{\cXs}^2 + \| \vpsi \|_{\cXs}^2 =C^2_\textrm{data}.
$$
This implies that $\vzeta^*_\cS\pm \vpsi  \in \cKs_{\vmeas_{\cS}}$ and for any $l \in \mathbb R$ we have
$$
\sup_{(\vv,q) \in \cKs_{\vmeas_{\cS}}} | Q(\vv,q)-l|^2 \geq \frac 1 2 |Q(\vzeta^*_\cS+ \vpsi) -l|^2 + \frac 1 2 |Q(\vzeta^*_\cS- \vpsi) -l|^2 = \frac 1 2 |Q(\vzeta^*_\cS)-l + Q(\vpsi)|^2 + \frac 1 2 |Q(\vzeta^*_\cS)-l -Q(\vpsi)|^2.
$$
Estimating further using the parallelogram law, we find that
$$
\sup_{(\vv,q) \in \cKs_{\vmeas_{\cS}}} | Q(\vv,q)-l|^2 \geq |Q(\vzeta^*_\cS)-l|^2 + |Q(\vpsi)|^2 \geq |Q(\vpsi)|^2.
$$
% Furthermore
Moreover, because $\vpsi \in \mathcal N$ corresponds to the largest eigenvalue of $Q^*_{\mathcal N}Q_{\mathcal N}$, we find that
$$
|Q(\vpsi)|^2 = \sup_{ \substack{\veta \in \mathcal N, \\ \| \veta \|_{\cXs}^2 \leq C_{\textrm{data}}^2-\| \vzeta^*_\cS\|_{\cXs}^2}} |Q(\veta)|^2 = \sup_{(\vv,q) \in \cKs_{\vmeas_{\cS}}} |Q((\vv,q)-\vzeta^*_\cS)|^2 = \sup_{(\vv,q) \in \cKs_{\vmeas_{\cS}}} |Q(\vv,q)-Q(\vzeta^*_\cS)|^2.
$$

Combining the above relations, we find that
$$
\sup_{(\vv,q) \in \cKs_{\vmeas_{\cS}}} |Q(\vv,q)-Q(\vzeta^*_\cS)|^2 \leq \sup_{(\vv,q) \in \cKs_{\vmeas_{\cS}}} |Q(\vv,q)-l|^2
$$
and thus 
\begin{equation}\label{e:opt_Q_S}
 \sup_{(\vv,q) \in \cKs_{\vmeas_{\cS}}} |Q(\vv,q)-Q(\vzeta^*_\cS)| = \inf_{l \in \mathbb R} \sup_{(\vv,q) \in \cKs_{\vmeas_{\cS}}} |Q(\vv,q)-l|,
\end{equation}
which is the desired optimality result for $Q(\vu_\cS^*,p_\cS^*)=Q(\vzeta^*_\cS)$.

The estimate \eqref{e:Q_opt} follows recalling that $(\vu,p)\in K^s_\vmeas$ if and only if $(\vu,p)=(\vu_\vf,p_\vf)+(\vu_\cS,p_\cS)$, where $(\vu_\vf,p_\vf) \in H^1_0(\Omega)^d\times L^2_0(\Omega)$ solves \eqref{def:pb1} and $(\vu_\cS,p_\cS) \in \cKs_{\vmeas_{\cS}}$. Therefore, for $(\vv,q)\in K_{\vmeas}^s$ we have
$$
Q(\vv,q)-Q(\vu^*,p^*) = Q((\vv,q)-(\vu_{\vf},p_{\vf}))-Q(\vzeta^*_\cS)
$$
so that \eqref{e:opt_Q_S} yields
\begin{equation*}
    \begin{split}
        \sup_{(\vv,q) \in K^s_{\vmeas}} |Q(\vv,q)-Q(\vu^*,p^*)| \leq \inf_{l \in \mathbb R} \sup_{(\vv,q) \in \cKs_{\vmeas_{\cS}}} |Q(\vv,q)-l| = \inf_{l \in \mathbb R} \sup_{(\vv,q) \in K^s_{\vmeas}} |Q(\vv,q)-l|. 
        \end{split}
\end{equation*}

The proof of \eqref{e:Q_near_opt} directly follows from \eqref{e:Q_opt} and the estimate for the approximation of the Chebyshev center \eqref{e:approx_center} given in Theorem~\ref{thm:OR}:
$$
\sup_{(\vv,q) \in K^s_\vmeas} |Q(\vv,q)-Q(\widehat \vu,\widehat p)| \leq R(K^s_\vmeas;Q)+ 
|Q(\vu^*,p^*)-Q(\widehat \vu,\widehat p)| \leq R(K^s_\vmeas;Q) + \varepsilon \|Q\|_{X^{\#}}.
$$
\end{proof}

\begin{remark}
It is possible to relate $R(K^s_\vmeas;Q)=R(\cKs_{\vmeas_{\cS}};Q)$ with $R(K^s_\vmeas)=R({\cKs_\vmeas})$. Indeed, it transpires from the argument provided in the proof of Theorem~\ref{thm:OR_Q} that
    $
    |Q(\vpsi)|^2 = R(\cKs_{\vmeas_{\cS}};Q)^2$, 
    where we recall that $\vpsi$ is an eigenfunction corresponding to the largest eigenvalue of $Q^*_{\mathcal N}Q_{\mathcal N}$, say $\mu$. Therefore,
$$
R(K^s_\vmeas;Q)^2 = \mu \| \vpsi \|_{\cXs}^2 = \mu(C_\textrm{data}^2-\| \vu^*_\cS,p^*_\cS \|_{\cXs}^2) = \mu \inf_{(\vu,p)\in X} \sup_{(\vv,q)\in \cKs_{\vmeas_{\cS}}}\|(\vv,q)-(\vu,p)\|_X = \mu R(K^s_\vmeas)^2.
$$
\end{remark}

\section{Computation of the Riesz representers} \label{sec:Riesz}

The key ingredient in the near-optimal recovery algorithm presented in Section~\ref{sec:algo} is the computation of the Riesz representers $\vphi_j$ of the measurement functionals $\lambda_j$, $j=1,2,\ldots,m$, when viewed as functionals on $\cXs$. In view of the definition of the inner product on $\cXs$, see \eqref{eqn:IP_Xs}, we have that for $j=1,2,\ldots,m$, $\vphi_j=(\vw_j,r_j)\in\cXs$ solves
\begin{equation} \label{eqn:Riesz_j}
  \langle \vw_j,\vv\rangle_{H^s(\Gamma)} + \mean(r_j)\mean(q)=\lambda_j(\vv,q) \quad \forall\,(\vv,q)\in\cXs.  
\end{equation}

In this section, we discuss a practical algorithm that can be used to approximate $\vphi_j$, $j=1,2,\ldots,m$. We start by describing a simple technique to handle the nonlocal mean value terms in \eqref{eqn:Riesz_j} without extra expense.  We then propose a practical finite element method to approximate the Riesz representers.

For simplicity, we describe the procedure considering a generic measurement functional $\lambda$ by omitting the index $j \in \{1,...,m\}$. 

\subsection{Mean value constraint} \label{sec:splitting_Riesz}

The mean value of the pressure requires a special treatment that we discuss here. The idea is to decompose the pressure component $r$ of $\vphi$ into a mean-free part plus a constant and observe that the latter can be explicitly determined by \eqref{eqn:Riesz_j}. Subtracting this constant from $r$ at the outset, all subsequent calculations are done for pressures with vanishing average, that is, looking for velocity-pressure couples in  
\begin{equation} \label{def:space_Xs0}
  \cXsz:=\left\{(\vv,q)\in \cXs: \,\, q\in L_0^2(\Omega)\right\} \subset \cXs.
\end{equation}
Notice that $\| \vv,q \|_{\cXs} = \| \vv \|_{H^s(\Gamma)}$ for $(\vv,q)\in \cXsz$. 

Let $\vphi=(\vw,r)\in\cXs$ be the Riesz representer of $\lambda$ in $\cXs$. Letting $(\vv,q)=(\mathbf{0},1) \in \cXs$ in \eqref{eqn:Riesz_j}, we find that the mean of $r$ is explicitly given by
$$\mean(r)=\lambda(\mathbf{0},1)=:\Lambda \in \mathbb R.$$
We then introduce $r_{0}:=r-\Lambda \in L^2_0(\Omega)$ so that 
\begin{equation} \label{eqn:split_Riesz_j}
  (\vw,r) = (\vw,r_{0}) + (\mathbf{0},\Lambda), \qquad \textrm{or} \qquad \vphi = \vphi_0+ (\mathbf{0},\Lambda),
\end{equation}
where $\vphi_{0}:=(\vw,r_{0})\in\cXsz$ solves
\begin{equation} \label{eqn:Riesz_j0}
  \langle \vw,\vv\rangle_{H^s(\Gamma)}=\lambda(\vv,q) \quad \forall\,(\vv,q)\in\cXsz;
\end{equation}
compare with \eqref{eqn:Riesz_j}.
We refer to $\vphi_{0}$ as the Riesz representer in $\cXsz$ of $\lambda$ (restricted to pressures with vanishing mean value). 

\subsection{Saddle-point formulation} \label{subsec:saddle}

The Stokes and compatibility constraits in $\cXsz$ makes the finite element discretization of \eqref{eqn:Riesz_j0} challenging. To address this difficulty, we adopt a standard approach based on Lagrange multipliers, which allows the constraints to be enforced weakly rather than built directly into the functional space.

We define the functional space
$$X_0^s = \left\{(\vv,q)\in H^1(\Omega)^d\times L_0^2(\Omega)\ :   \ \vv|_\Gamma \in H^s(\Gamma)^d \right\}$$
relaxing the Stokes and compatibility constraints in $\cXs_0$ and equip it with the norm
\begin{equation} \label{def:X0s_norm}
  \|(\vv,q)\|_{X_0^s}:=\sqrt{\|\nabla\vv\|_{L^2(\Omega)}^2+\|\vv\|_{H^s(\Bd)}^2+\|q\|_{L^2(\Omega)}^2}.  
\end{equation}
The saddle-point formulation of \eqref{eqn:Riesz_j0} using Lagrange multipliers read: find $(\vw,r_0)\in X_0^s$ and $(\vpi,\rho)\in H_0^1(\Omega)^d\times L^2(\Omega)$ such that
\begin{equation} \label{def:saddle}
	\left\{\begin{array}{rcll}
		a_s(\vw,\vv)+\mathcal{A}((\vv,q),(\vpi,\rho)) & = & \lambda(\vv,q), & \forall\, (\vv,q)\in X_0^s, \\[2ex]
        \mathcal{A}((\vw,r_0),(\vz,\tau)) & = & 0, & \forall\, (\vz,\tau)\in H_0^1(\Omega)^d\times L^2(\Omega),
	\end{array}\right.
\end{equation}
where $a_s(\vz,\vv) :=\langle \vz,\vv\rangle_{H^s(\Bd)}$ and $\mathcal{A}((\vv,q),(\vz,\tau)) := k(\vv,\vz)+b(\vv,\tau)+b(\vz,q)$ with $k(\vv,\vz):=\int_{\Omega}2\veps(\vv):\veps(\vz)$ and $b(\vv,q) := - \int_{\Omega}q\,{\rm div}(\vv)$.

The proof of existence and uniqueness for the saddle-point problem is technical but rather standard. It is therefore given in Appendix~\ref{sec:WP_saddle}. However, we point out that the solution $(\vw,r_0)$ of the above system is also the unique solution to \eqref{eqn:Riesz_j0}. Indeed, taking $(\vz,\tau) =(\mathbf{0},1)$ in the second equation in \eqref{def:saddle} yields $b(\vw,1)=0$, and thus $\int_{\Bd}\vw\cdot\vn=0$ thanks to the divergence theorem. Furthermore, the second equation of \eqref{def:saddle} restricted to $(\vz,\tau)\in H_0^1(\Omega)^d\times L_0^2(\Omega)$ shows that $(\vw,r_0)\in X_0^s$ satisfies $\mathcal S(\vw,r_0)={\bf 0}$ and thus $(\vw,r_0)\in \cXsz$. 
Lastly, restricting to  test functions $(\vv,q)\in\cXsz$ in the first equation in \eqref{def:saddle} so that $\mathcal{A}((\vv,q),(\vpi,\rho))=0$ yields 
$$
a_s(\vw,\vv) =\lambda(\vv,q) \quad \forall\,(\vv,q)\in\cXsz
$$
and $(\vw,r_0)$ satisfies \eqref{eqn:Riesz_j0}.

From the previous argumentation, we realize that if $\rho=\rho_0+\bar \rho$ with $\rho_0\in L_0^2(\Omega)$ and $\bar \rho:=\mean{(\rho)}\in\mathbb{R}$, $(\vpi,\rho_0)$ are the Lagrange multipliers enforcing the Stokes constraint, while $\bar \rho$ is the Lagrange multiplier enforcing the compatibility condition $\int_{\Bd}\vw\cdot\vn = 0$. The latter is critical not only for the equivalence between \eqref{def:saddle} and \eqref{eqn:Riesz_j0} but also for the existence and uniqueness proof presented in Appendix~\ref{sec:WP_saddle}.

\subsection{Finite element approximation}

Let $\{\mathcal{T}_h\}_{h>0}$ be a family of quasi-uniform and shape-regular subdivisions of $\overline{\Omega}$ made either of simplices or quadrilaterals/hexahedrons, possibly curved to match $\Gamma$. For each mesh $\mathcal{T}_h$, let $\mathbb{V}_h$ be a Lagrange (nodal) finite element space for $H^1(\Omega)^d$, and let $\mathbb{W}_h:=\mathbb{V}_h\cap H_0^1(\Omega)^d$ be the subspace of $\mathbb{V}_h$ with vanishing trace.

It will be convenient to distinguish between nodal basis functions associated with interior degrees of freedom and boundary ones. We let $N:=\textrm{dim}(\mathbb W_h)$ and $N_b:=\textrm{dim}(\mathbb V_h)-N$ and denote by $\vvarphi_i$, $i=1,...,N+N_b$, the nodal basis functions associated to the degrees of freedom parameterizing $\mathbb V_h$ numbered so as to start with the nodal basis associated with the interior degrees of freedom.
With this notation we have
\begin{align*}
    \mathbb{W}_h &={\rm span}\{\vvarphi_1,\ldots,\vvarphi_N\}\subset H_0^1(\Omega)^d, \qquad \mathbb{V}_h^b :={\rm span}\{\vvarphi_{N+1},\ldots,\vvarphi_{N+N_b}\}\subset H^1(\Omega)^d, \\
  \mathbb{V}_h &=\mathbb{W}_h\oplus \mathbb{V}_h^b\subset H^1(\Omega)^d.
\end{align*}

Regarding the pressure space, let $\mathbb{M}_h$ be a conforming finite element space for $L^2(\Omega)$ and let $\mathbb{Q}_h:=\mathbb{M}_h\cap L_0^2(\Omega)$ be the subspace with zero average. We suppose that $\mathbb{W}_h\times \mathbb{Q}_h$ satisfies the (Stokes) inf-sup condition, namely that there exists a constant $\beta>0$ (independent of $h$) such that
$$
\inf_{q_h\in \mathbb{Q}_h}\,\sup_{\vv_h\in\mathbb{W}_h}\frac{\int_{\Omega}q_h{\rm div}(\vv_h)}{\|\nabla\vv_h\|_{L^2(\Omega)}\|q_h\|_{L^2(\Omega)}}\ge \beta.
$$
We let $\widetilde N:=\textrm{dim}(\mathbb Q_h)$ and denote by $\phi_i$, $i=1,...,{\widetilde N}$, a basis of $\mathbb Q_h$.

Using the decomposition $L^2(\Omega)=L_0^2(\Omega)\oplus \mathbb{R}$, the discrete counterpart of the saddle-point system \eqref{def:saddle} reads: find $(\vw_h,r_{0,h})\in \mathbb{V}_h\times \mathbb{Q}_h$ and $(\vpi_h,\rho_h,\gamma)\in \mathbb{W}_h\times \mathbb{Q}_h\times \mathbb R$ such that
\begin{equation} \label{def:saddle_FE}
	\left\{\begin{array}{rcll}
		a_s(\vw_h,\vv_h)+\mathcal{A}((\vv_h,q_h),(\vpi_h,\rho_h))+b(\vv_h,\gamma) & = & \lambda(\vv_h,q_h), & \forall\, (\vv_h,q_h)\in \mathbb{V}_h\times \mathbb{Q}_h, \\
        \mathcal{A}((\vw_h,r_{0,h}),(\vz_h,\tau_h)) + b(\vw_h,\delta) & = & 0, & \forall\, (\vz_h,\tau_h,\delta)\in \mathbb{W}_h\times \mathbb{Q}_h \times \mathbb R.
    \end{array}\right.
\end{equation}
Note that $\mathbb{V}_h|_\Gamma=\mathbb{V}_h^b|_\Gamma$ is an $H^1(\Gamma)^d$ conforming finite element space on $\Gamma$, and thus functions in $\mathbb{V}_h$ have trace in $H^s(\Gamma)^d$ for any $s<3/2$.

In Section~\ref{subsec:iter_solve} we describe an efficient way to solve  \eqref{def:saddle_FE} based on the partition of interior
and boundary velocity degrees of freedom. The practical aspects are then discussed in Sections~\ref{subsec:lin_sys} and \ref{subsec:fractional_s}.
But before proceeding, we note that based on the analysis performed in \cite{BBCDDP2023,BG2024,BDS2025} for the Poisson problem, we expect $\varepsilon$ in Theorems~\ref{thm:OR} and~\ref{thm:OR_Q} to behave like $\varepsilon \in O(h^{\tilde s})$ for some $\tilde s>0$ depending on $s$, $\Omega$, the regularity of the solution to the Stokes equations in $\Omega$ with zero force, the regularity pickup index of $(I-\Delta_{\Gamma})^{-1}$ on $\Gamma$, and the polynomial degree used in $\mathbb V_h \times \mathbb Q_h$.

\subsubsection{Decoupled algorithm}
\label{subsec:iter_solve}

The use of Lagrangian finite elements for $\mathbb{V}_h$ gives the saddle point system \eqref{def:saddle_FE} a particular structure that allows for an efficient solver. Indeed, writing
$$\vw_h=\vw_h^0+\vw_h^b, \qquad \text{where} \quad \vw_h^0\in \mathbb{W}_h \quad \text{and} \quad \vw_h^b\in \mathbb{V}_h^b,$$
and splitting the velocity-pressure couple as $(\vw_h,r_{0,h})=(\vw_h^0,r_{0,h})+(\vw_h^b,0)$, the following iterative procedure can be used to solve the system \eqref{def:saddle_FE}:
\begin{enumerate}[Step \arabic*.]
    \item Restricting test functions to $(\vv_h^0,q_h)\in\mathbb{W}_h\times \mathbb{Q}_h$ in the first equation of \eqref{def:saddle_FE} and using the fact that $\vv_h^0|_{\Gamma}=\mathbf{0}$ and $b(\vv_h^0,\gamma)=0$ for all $\gamma \in \mathbb R$, we arrive at the following system for the Lagrange multiplier $(\vpi_h,\rho_h)\in\mathbb{W}_h\times \mathbb{Q}_h$:
\begin{equation} \label{sys:LagMult_FE}
  \mathcal{A}((\vv_h^0,q_h),(\vpi_h,\rho_h)) = \lambda(\vv_h^0,q_h), \qquad \forall\, (\vv_h^0,q_h)\in\mathbb{W}_h\times \mathbb{Q}_h.
\end{equation}
Notice that \eqref{sys:LagMult_FE} is Stokes system with forcing $\lambda$ and vanishing velocity on $\Gamma$ (and thus the pressure is with vanishing mean value).
    \item \label{enum:step2} Restricting to test functions $(\vv_h^b,0) \in \mathbb{V}_h^b \times \mathbb Q_h$, $\vz_h=\bf{0}$, $\tau_h=0$ and using that $\vw_h^0|_\Gamma=\mathbf{0}$ yield a coercive fractional diffusion system on $\Gamma$ with constraint
    \begin{equation} \label{sys:w_bdy_FE}
	\left\{\begin{array}{rcll}
		a_s(\vw_h^b,\vv_h^b)+b(\vv_h^b,\gamma) & = & \lambda(\vv_h^b,0)-k(\vv_h^b,\vpi_h)-b(\vv_h^b,\rho_h), & \forall\, \vv_h^b\in\mathbb{V}_h^b, \\
        b(\vw_h^b,\delta) & = & 0, & \forall\, \delta\in \mathbb R,
        % b(\vw_h^b,d) & = & 0 & \forall\, d\in\mathbb{R}
	\end{array}\right.
\end{equation}
which uniquely determine $\vw_h^b\in\mathbb{V}_h^b$, and the constant $\gamma \in \mathbb R$ necessary to enforce the compatibility condition $\int_\Gamma \vw_h^b \cdot \vn = 0$. 
\item The remaining unknown $(\vw_h^0,r_{0,h})\in\mathbb{W}_h\times \mathbb{Q}_h$ is uniquely determined by another Stokes system corresponding to the second equation of \eqref{def:saddle_FE} with $\delta=0$
\begin{equation} \label{sys:Riesz_FE}
\mathcal{A}((\vw_h^0,r_{0,h}),(\vz_h,\tau_h)) = -\mathcal{A}((\vw_h^b,0),(\vz_h,\tau_h)) \qquad \forall\, (\vz_h,\tau_h)\in\mathbb{W}_h\times \mathbb{Q}_h.
\end{equation}
% \todo{DG: I replaced $\vz_h^0$ by $\vz_h$ for consistency with \eqref{def:saddle_FE} (also no superscript in $\vpi_h$). To be discussed ($\Pi_0$ below).}
Notice that \eqref{sys:Riesz_FE} is Stokes system for $(\vw_h,r_{0h})$ without forcing term but with boundary condition $\vw_h|_\Gamma=\vw_h^b$ (and thus the pressure is with vanishing mean value). It has a unique solution because, thanks to the previous step, the boundary data satisfies the compatibility condition.
\end{enumerate}

\subsubsection{Linear systems}
\label{subsec:lin_sys}

In this section, we formulate the linear systems associated with steps 1 to 3. We suppose here that the computation of $H^s(\Gamma)$ is exact and postpone the discussion of a practical algorithm for non-integer $s$ to Section~\ref{subsec:fractional_s}. 

We write each unknown function in their respective basis, namely
$$\vw_h = \sum_{j=1}^{N+N_b}\mathrm{w}_j\vvarphi_j, \quad r_{0,h}=\sum_{j=1}^{\widetilde N}r_{0,j}\phi_j, \quad \vpi_h= \sum_{j=1}^N\pi_j\vvarphi_j, \quad \text{and} \quad  \rho_h=\sum_{j=1}^{\widetilde N}\rho_j\phi_j$$
and collect the unknown coefficients in the vectors $W_0:= (w_i)_{i=1}^N \in \mathbb R^N$, $W_b:= (w_{N+i})_{i=1}^{N_b} \in \mathbb R^{N_n}$, $R_0:= (r_{0,i})_{i=1}^{\widetilde N}$, 
$\Pi:= (\pi_{i})_{i=0}^N$, $P:= (\rho_i)_{i=1}^{\widetilde N}$. Notice that the coefficients for $\vw_h$ are decomposed according to $\mathbb V_h = \mathbb W_h \oplus \mathbb V_h^b$.

We now make explicit the linear systems corresponding to the three steps of the decoupled algorithm described in Section~\ref{subsec:iter_solve}. 

\begin{enumerate}[label={Step \arabic*.}]
\item The discrete system \eqref{sys:LagMult_FE} for the Lagrange multipliers reads
\begin{equation} \label{def:lin_sys_mult}
    \left(\begin{array}{cc}
       K_0 & B_0^T \\
       B_0 & 0
       \end{array}\right)
    \left(\begin{array}{c}
       \Pi \\ P 
    \end{array}\right) = 
    \left(\begin{array}{c}
       F_0 \\ G
    \end{array}\right).
\end{equation}
Here, $K_0 := (k(\vvarphi_j,\vvarphi_i))_{i,j=1}^{N,N} \in \mathbb{R}^{N\times N}$ and $B_0 := (b(\vvarphi_j,\phi_i))_{i,j=1}^{{\widetilde N},N} \in \mathbb{R}^{{\widetilde N}\times N}$ are the finite element matrices associated to the bilinear forms $k(\cdot,\cdot)$ restricted to $\mathbb W_h \times \mathbb W_h$ and $b(\cdot,\cdot)$ restricted to $\mathbb W_h \times  \mathbb Q_h$, while $F_0:= (\lambda(\vvarphi_i,0))_{i=1}^N\in\mathbb{R}^N$ and $G:=(\lambda({\bf 0},\phi_i))_{i=1}^{\widetilde N}\in\mathbb{R}^{\widetilde N}$.
Since $K_0$ is invertible, we have
\begin{equation}\label{e:algo_step1}
B_0K_0^{-1}B_0^TP = B_0K_0^{-1}F_0-G \qquad \textrm{and} \qquad K_0\Pi = F_0-B_0^TP.
\end{equation}
\item The vector $W_b$ and the Lagrange multiplier $\gamma$ for the compatibility condition satisfy
\begin{equation} \label{def:lin_sys_cstr}
    \left(\begin{array}{cc}
       A_s & D \\
       D^T & 0
       \end{array}\right)
    \left(\begin{array}{c}
       W_b \\ \gamma
    \end{array}\right) = 
    \left(\begin{array}{c}
       F_b-K_b^T\Pi-B_b^TP \\ 0
    \end{array}\right),
\end{equation}
where $A_s := (a_s(\vvarphi_{N+j},\vvarphi_{N+i}))_{i,j=1}^{N_b,N_b} \in \mathbb{R}^{N_b\times N_b}$ is the finite element matrices corresponding to $a_s(\cdot,\cdot)$ restricted to $\mathbb V_h^b \times \mathbb V_h^b$, $D:= (b(\vvarphi_{N+i},1))_{i=1}^{N_b}\in \mathbb{R}^{N_b}$ is the finite element vector
corresponding to $b(\cdot,\cdot)$ restricted to $\mathbb V_h^b \times \mathbb R$, and $F_b:= (\lambda(\vvarphi_{N+i},0))_{i=1}^{N_b}\in\mathbb{R}^{N_b}$.
The above linear system is decoupled:
\begin{equation}\label{e:algo_step2}
\gamma = \frac{1}{D^TA_s^{-1}D} D^TA_s^{-1}(F_b-K_b^T\Pi-B_b^TP) \qquad \textrm{and} \qquad A_sW_b = F_b-K_b^T\Pi-B_b^TP- \gamma D.
\end{equation}
The computation of $\gamma$ and $W_b$ requires the approximation of $A_s^{-1}D$ (which is independent of the measurements and can be reused for all Riesz representers) and of $A_s^{-1}(F_b-K_b^T\Pi-B_b^TP)$.

\item The remaining unknown vectors $W_N$ and $R_0$ are solution to the linear system
\begin{equation} \label{def:lin_sys_riesz}
    \left(\begin{array}{cc}
       K_0 & B_0^T \\
       B_0 & 0
       \end{array}\right)
    \left(\begin{array}{c}
       W_0 \\ R_0
    \end{array}\right) = 
    \left(\begin{array}{c}
       -K_bW_b \\ -B_bW_b
    \end{array}\right)
\end{equation}
and so 
\begin{equation}\label{e:algo_step3}
B_0K_0^{-1}B_0^TR_0 = (B_b-B_0K_0^{-1}K_b)W_b \qquad \textrm{and}\qquad K_0W_0 = -K_bW_b-B_0^TR_0. 
\end{equation}
\end{enumerate}

We note that \eqref{e:algo_step1}-\eqref{e:algo_step2}-\eqref{e:algo_step3} require the approximation of linear system involving $A_s$ and $K_0$. For $K_0$ this can be either performed by its $LU$ decomposition once for all or by using an iterative solver. Since these matrices are the same for all Riesz representers, we used the former approach for the numerical illustrations provided in Section~\ref{sec:numeric}. Furthermore the system for $P$ in \eqref{e:algo_step1}-left and for $R_0$ in \eqref{e:algo_step3}-left are solved using a conjugate gradient algorithm. The situation is more complicated for $A_s$. When $s=1$, we can proceed similarly than for $K_0$. However, when $s$ is an integer larger than $1$ the restrictions to $\Gamma$ of functions in $\mathbb W_h$ are not in $H^s(\Gamma)$. Furthermore, when $s$ is not integral, the matrix $A_s$ is full since it corresponds to a fractional diffusion, and thus nonlocal, system. In the next section, we present an algorithm to circumvent these issues.

\subsubsection{Approximation of $A_s^{-1}$}
\label{subsec:fractional_s}

In this section, we discuss the resolution of \eqref{e:algo_step2} when $s$ is non integral or/and $s>1$.
For the latter, it suffices to observe that $A_s^{-1} = A_{s-1}^{-1}A_1^{-1}$ so that this inverse reduces to $\lfloor s \rfloor$ application of $A_1^{-1}$ and one application of $A_t^{-1}$ with $t:=s-\lfloor s \rfloor$, which is between $0$ and $1$. Therefore, it remains to discuss the case $0<s<1$.

Observe that the matrix $A_s$ is full. Rather than assembling it, we provide an approximation of $A_s^{-1}$ following the technology proposed in \cite{BG2024}. This approach takes advantage of the spectral definition of the inner product of $H^s(\Gamma)$, see \eqref{def:Hs}, and relies on one of the well-established numerical methods for fractional diffusion problems. Refer to \cite{bonito2015numerical,bonito2017sinc} for elliptic operators acting on functions defined on Euclidean domains and to \cite{bonito2021approximation} for an extension to functions defined on hyper-surfaces as needed in the present work. 

We introduce the mass matrix $M:=( \langle \vvarphi_{N+j},\vvarphi_{N+i}\rangle_{L^2(\Gamma)})_{i,j=1}^{N_b,N_b}\in\mathbb{R}^{N_b\times N_b}$ and stiffness matrix $L:=( \langle \nabla_{\Gamma} \vvarphi_{N+j},\nabla_{\Gamma} \vvarphi_{N+i}\rangle_{L^2(\Gamma)})_{i,j=1}^{N_b,N_b}\in\mathbb{R}^{N_b\times N_b}$ and recall the Balakrishnan formula
\begin{equation} \label{eqn:Balak_Twh}
  A_s^{-1} := \frac{\sin(\pi s)}{\pi}\int_{-\infty}^{\infty}e^{(1-s)y}\left( (e^y+1)M+L\right)^{-1} \text{d}y.
\end{equation}
A sinc quadrature is then used to approximate the improper integrals. Given a spacing parameter $k>0$, we set $\texttt{M}:=\left\lceil\frac{\pi^2}{2(1-s)k^2}\right\rceil$ and $\texttt{N}:=\left\lceil\frac{\pi^2}{2sk^2}\right\rceil$, 
and advocate the approximation of $A_s^{-1}$ by
\begin{equation} \label{eqn:Balak_sinc_Twh}
  A_s^{-1} \approx \mathcal{Q}_k^{-s}:=\frac{k\sin(\pi s)}{\pi}\sum_{l=-\texttt{M}}^{\texttt{N}}e^{(1-s)y_l}\left( (e^{y_l}+1)M+L\right)^{-1}, \qquad y_l:=kl.
\end{equation}
We refer to \cite{bonito2017sinc,bonito2021approximation} for the exponential (in $k$) convergence  properties of this approximation and to \cite{BG2024} for its influence on an optimal recovery problem for the Poisson problem.

\section{Numerical experiments}\label{sec:numeric}

% \todo{DG: replace $\vw$ (already used for Riesz repres); also previous version was not clear (see commented part)}
In this section, we illustrate the performance of Algorithm~\ref{alg:OR} when the Riesz representers are approximated using a finite element method as described in Section~\ref{sec:Riesz}. The measurements are Gaussian functionals serving as proxy for point evaluations of the component of the velocity and the pressure. More precisely, given a point (center) $\vz\in\Omega$ we define
\begin{equation}\label{def:lambda_z}
\lambda_{\vz,i}(\vpsi) =  \frac{1}{\sqrt{2\pi r^2}} \int_\Omega \exp{\left(-\frac{|\vx-\vz|^2}{2r^2}\right)} \psi_i(\vx)d\vx, \qquad r=0.1,
\end{equation}
where $i\in\{1,2,\ldots,d+1\}$ and $\vpsi:=((\psi_1,...,\psi_d),\psi_{d+1}) \in H^1(\Omega)^d\times L^2(\Omega)$.
The number and the distribution of the points $\vz$ will be indicated for each test case.

We recall that the approximation of each Riesz representer is performed using the decoupled algorithm presented in Section~\ref{subsec:lin_sys}. We store an $LU$ decomposition of $K_0$ and $A_1$ and use a conjugate gradient algorithm for the resolution of \eqref{e:algo_step1}-left and \eqref{e:algo_step3}-left. The conjugate gradient algorithms are stopped when the $l_2$-norm of the system residual is smaller that $10^{-9}$ times the $l_2$-norm of the right-hand side of the system. The same process is used for the finite element approximation of $(\vu_{\vf},p_{\vf}) \in H^1_0(\Omega)^d\times L^2_0(\Omega)$ satisfying \eqref{def:pb1}.
To limit numerical errors due to the poor conditioning of $\widehat G$ in \eqref{eqn:LS_Gram_approx}, which becomes more pronounced as the number of measurements increases, we employ a diagonal preconditioner  as well as a thresholding strategy with parameter $\epsilon=10^{-10}$; refer to Section~\ref{sec:ill_cond} for details. 

In Section~\ref{sec:square}, we consider a square computational domain with two manufactured solutions and investigate the impact of the number of measurements on the recovery error. In Section~\ref{sec:square_hole}, we extend this study to a square domain with a circular hole, where only the boundary conditions on the hole boundary are unknown and the pressure to be recovered does not have vanishing mean value. Building on these results, Section~\ref{sec:airfoil} considers the flow around an airfoil, for which  both the outflow boundary condition and the boundary condition on the airfoil are unknown. In absence of an exact solution, we use a numerical solution with prescribed boundary conditions to draw the measurements and assess the recovery error. We also report recovery errors for the drag and lift coefficients. Section~\ref{sec:ill_cond} is devoted to the numerical study of different strategies to cope with the conditioning of $\widehat G$ \ref{sec:ill_cond}. We end the numerical section by presenting the recovery of functions defined on a three-dimensional domain in Section~\ref{sec:3D}.

All numerical studies are conducted with regularity parameter $s=1$ and with Taylor--Hood $\mathbb{Q}_2^d\times\mathbb{Q}_1$ element for the finite element space; the influence of the regularity parameter $s$ is examined at the end of Section~\ref{sec:square}. Note that other finite element spaces could be used, such as the MINI finite element or higher-order Taylor--Hood finite element spaces $\mathbb{Q}_k^d\times\mathbb{Q}_{k-1}$, $k\ge 2$. 

\subsection{Square domain} \label{sec:square}

In this section, we  provide a numerical study on the influence of the number of measurements on the recovery error.
We set $\Omega=(0,1)^2$ and consider two different solutions to be recovered, namely
\begin{equation} \label{num:case1}
    u_1(\vx) = e^{x_1}\cos(x_2), \quad u_2(\vx)=-e^{x_1}\sin(x_2)+2x_1^2, \quad p(\vx)=2(2x_2-1)
\end{equation}
and
\begin{equation} \label{num:case2}
    u_1(\vx) = -\cos(\pi x_1)\sin(\pi x_2), \quad u_2(\vx)=\sin(\pi x_1)\cos(\pi x_2), \quad p(\vx)=-\cos(2\pi x_1)-\cos(2\pi x_2),
\end{equation}
where $\vx=(x_1,x_2)\in\Omega$.  Note that in the first case \eqref{num:case1}, $\vf = {\bf 0}$ as well as  $\|\vu\|_{H^1(\Omega)}=3.274$ and  $\|p\|_{L^2(\Omega)}=1.155$.
We have $\vf \not\equiv \mathbf{0}$ in the second case \eqref{num:case2} and $\|\vu\|_{H^1(\Omega)}=3.220$, $\|p\|_{L^2(\Omega)}=1$.

Given a square integer $m=l^2$, the points controlling the centers of the Gaussian measurements defined in \eqref{def:lambda_z} are uniformly distributed in $\Omega$:
\begin{equation} \label{eqn:z_Gauss}
    \vz_{i,j} = \left(\frac{i}{l+1},\frac{j}{l+1}\right), \quad i,j=1,2,\ldots,l.
\end{equation}

Finally, given an integer $n$ (refinement level), the domain $\Omega$ is discretized into $2^{2n}$ squares of side-length $2^{-n}$, yielding a uniform mesh with mesh size $h=\sqrt{2}\cdot 2^{-n}$. 

We report results for two refinements corresponding to $n=6$ (33282 degrees of freedom for the velocity and 4225 for the pressure) and $n=8$ (526338 degrees of freedom for the velocity and 66049 for the pressure). 

We start with the case \eqref{num:case1}. 
Figure~\ref{fig:case1_mu_mp} contains the exact and recovery solutions obtained for different numbers of measurements and with a subdivision corresponding to $n=8$ refinements. Note that the solution to \eqref{def:pb1} in $H^1_0(\Omega)^2\times L^2_0(\Omega)$ vanishes in the case \eqref{num:case1} since $\vf=\mathbf{0}$.
The recovery errors
% \todo{would be better on a single line}
$$
\text{err}_u:=\| \vu - \widehat \vu_h \|_{H^1(\Omega)}, \quad  \text{err}_p:=\| p - \widehat p_h \|_{L^2(\Omega)} \quad \textrm{and}\quad  \text{err}:=\sqrt{\text{err}_u^2+\text{err}_p^2}
$$
are reported in Table~\ref{tab:case1_mu_mp}. Here $(\widehat \vu_h,\widehat p_h)$ is the output of the recovery algorithm. 

\begin{figure}[htbp]
\begin{center}
\begin{tabular}{ccccc}
\includegraphics[trim={11cm 5cm 10cm 4cm},clip,width=0.16\textwidth]{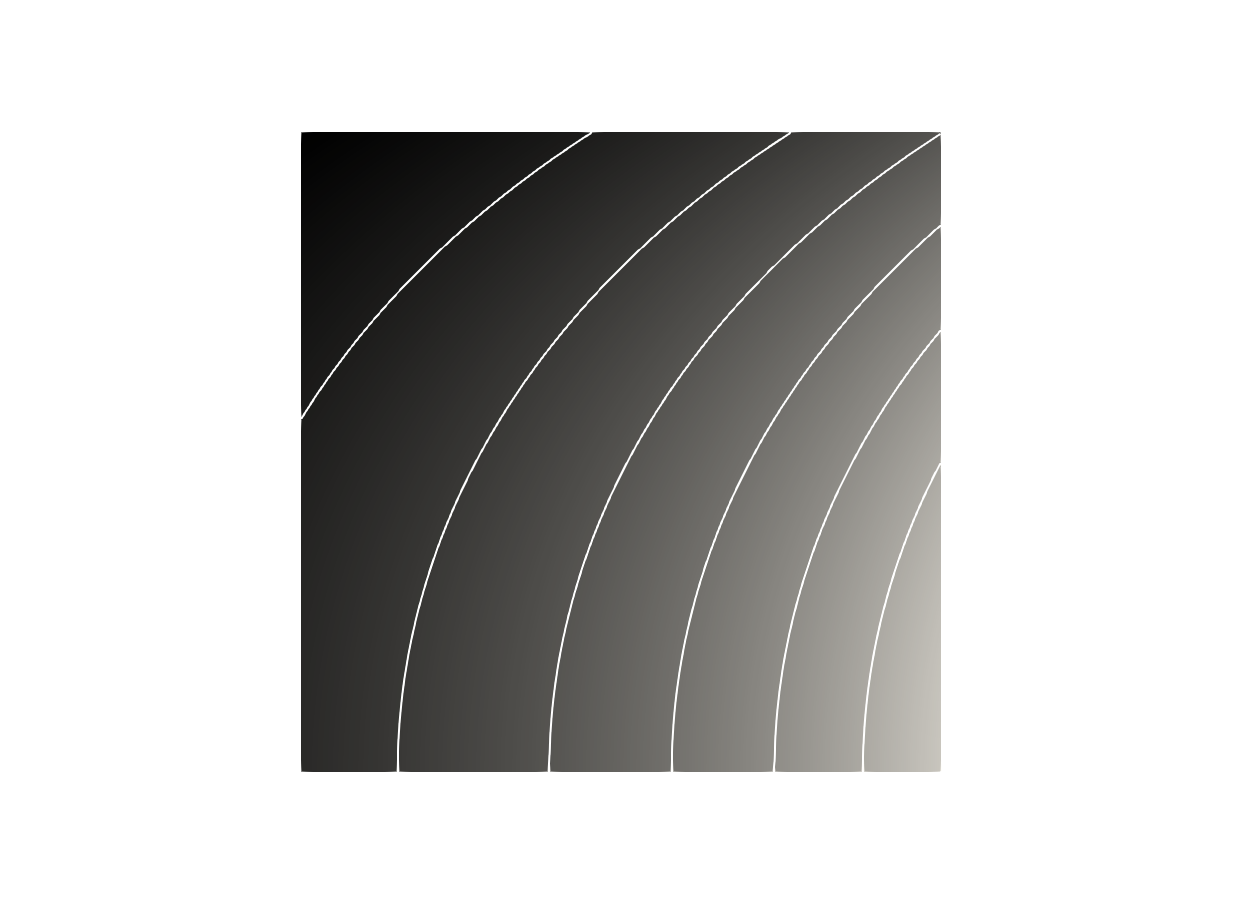} & 
\includegraphics[trim={11cm 5cm 10cm 4cm},clip,width=0.16\textwidth]{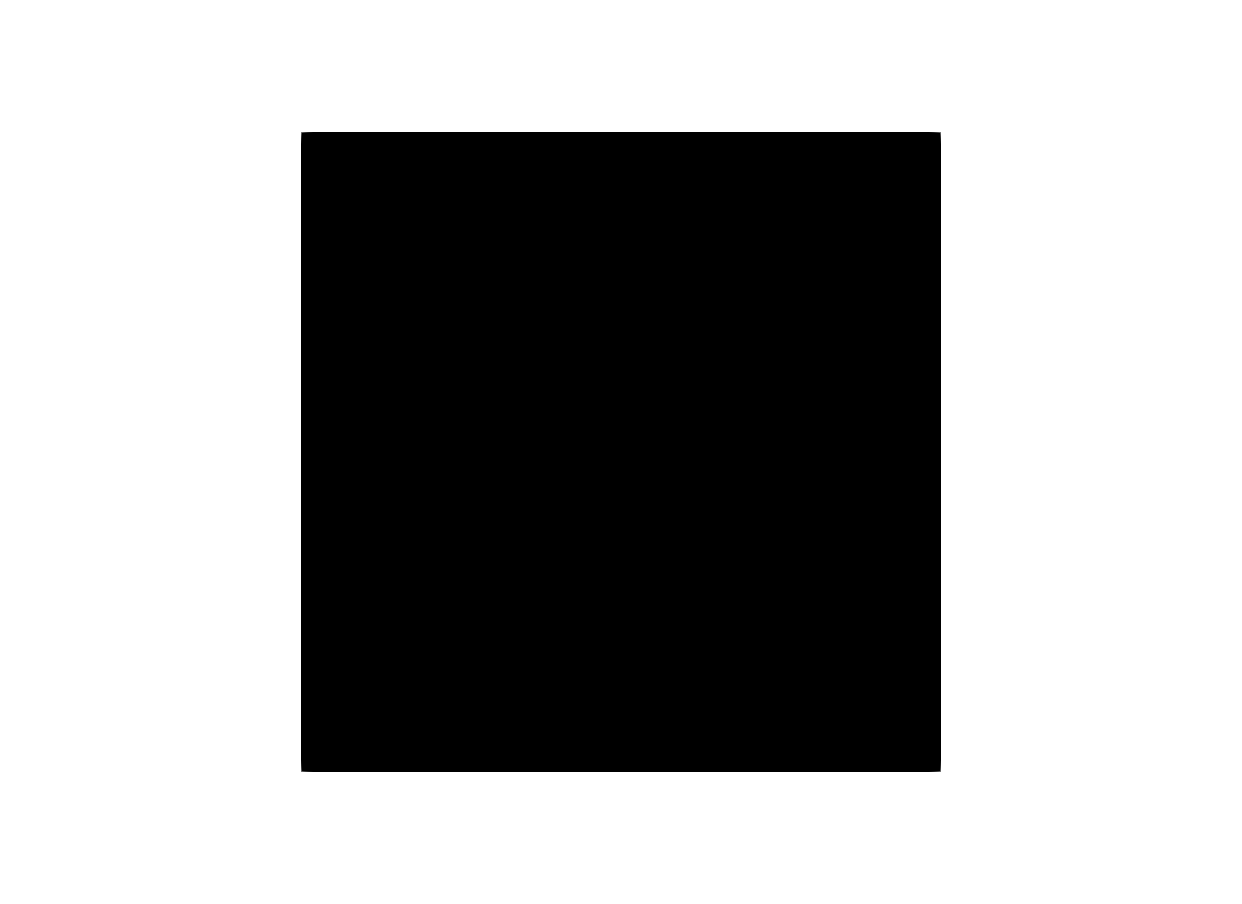} &
\includegraphics[trim={11cm 5cm 10cm 4cm},clip,width=0.16\textwidth]{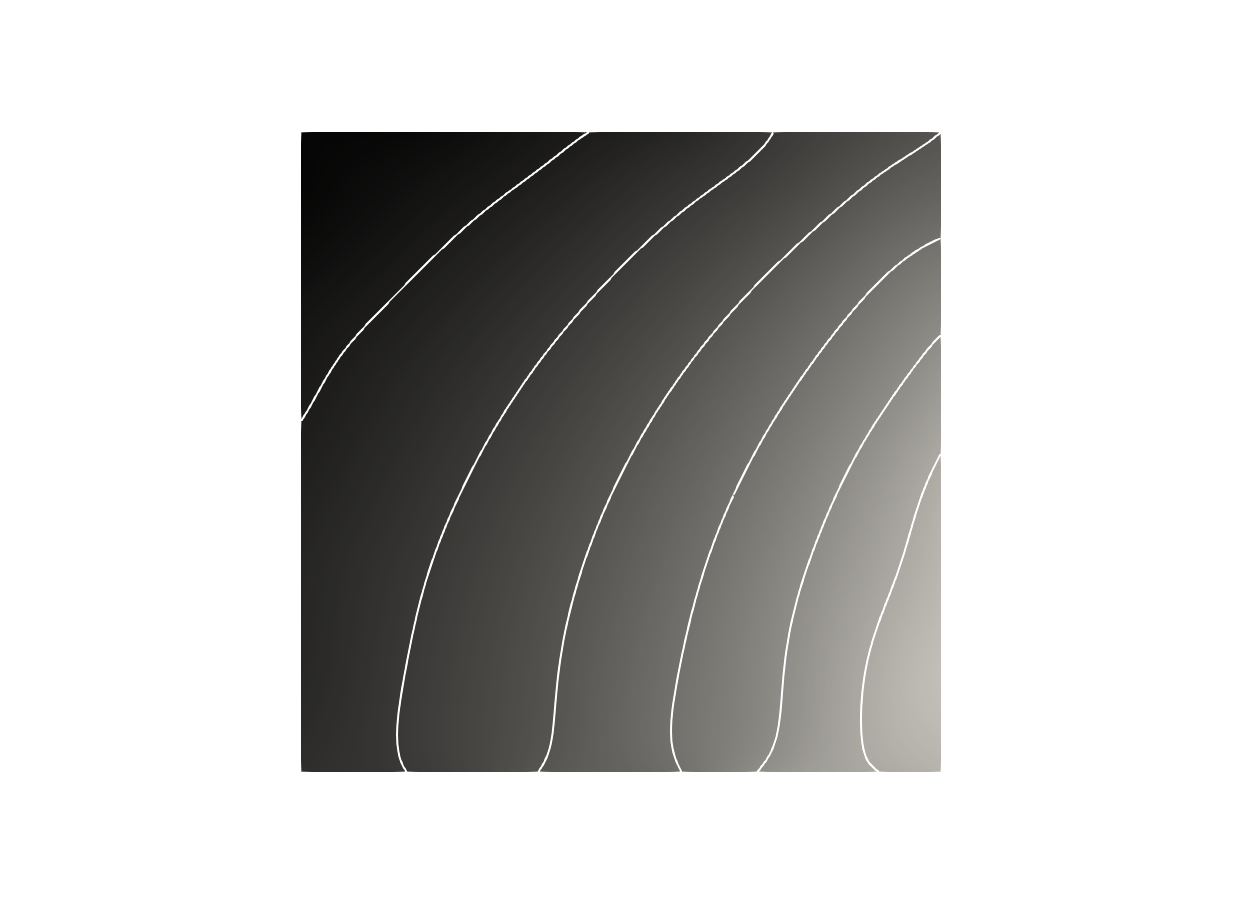}&
\includegraphics[trim={11cm 5cm 10cm 4cm},clip,width=0.16\textwidth]{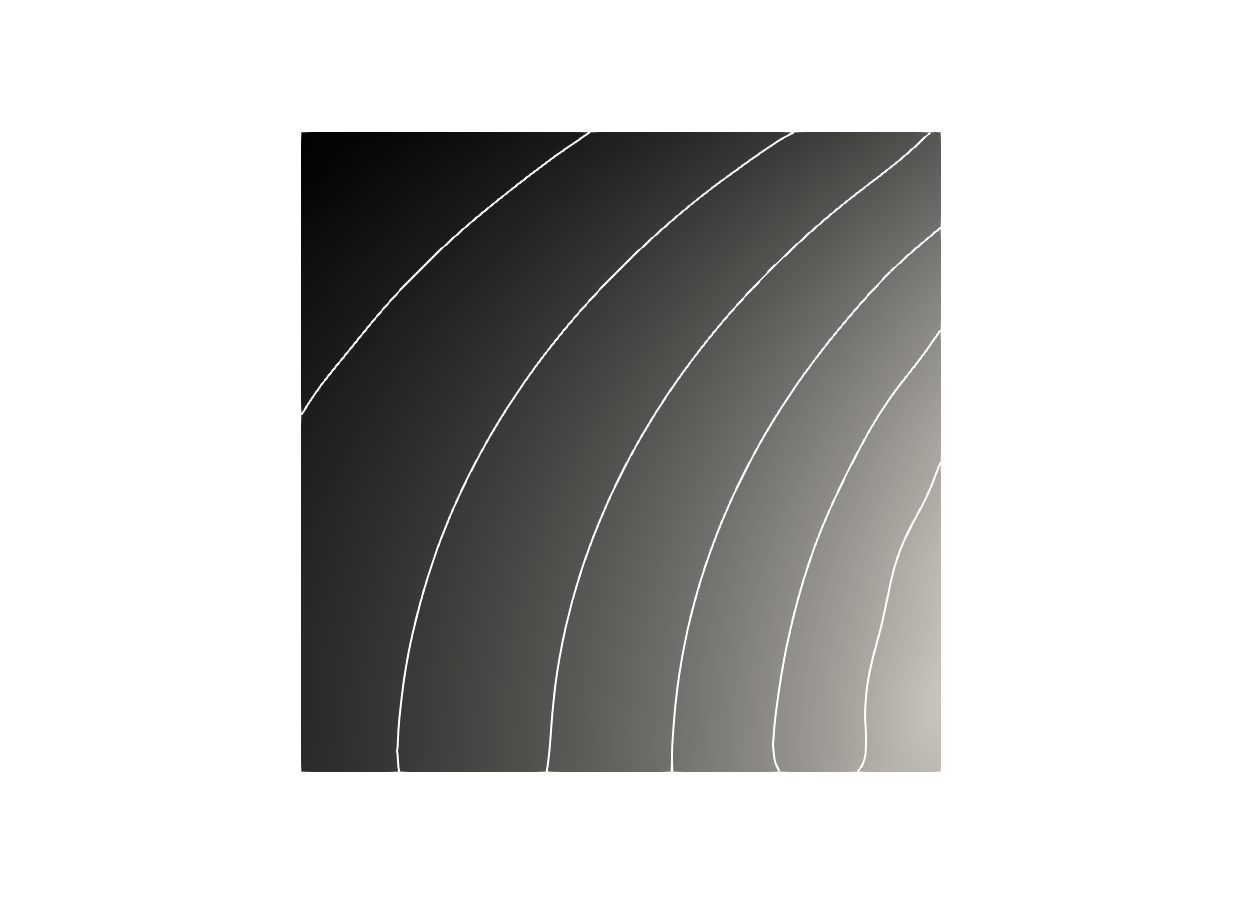}&
\includegraphics[trim={11cm 5cm 10cm 4cm},clip,width=0.16\textwidth]{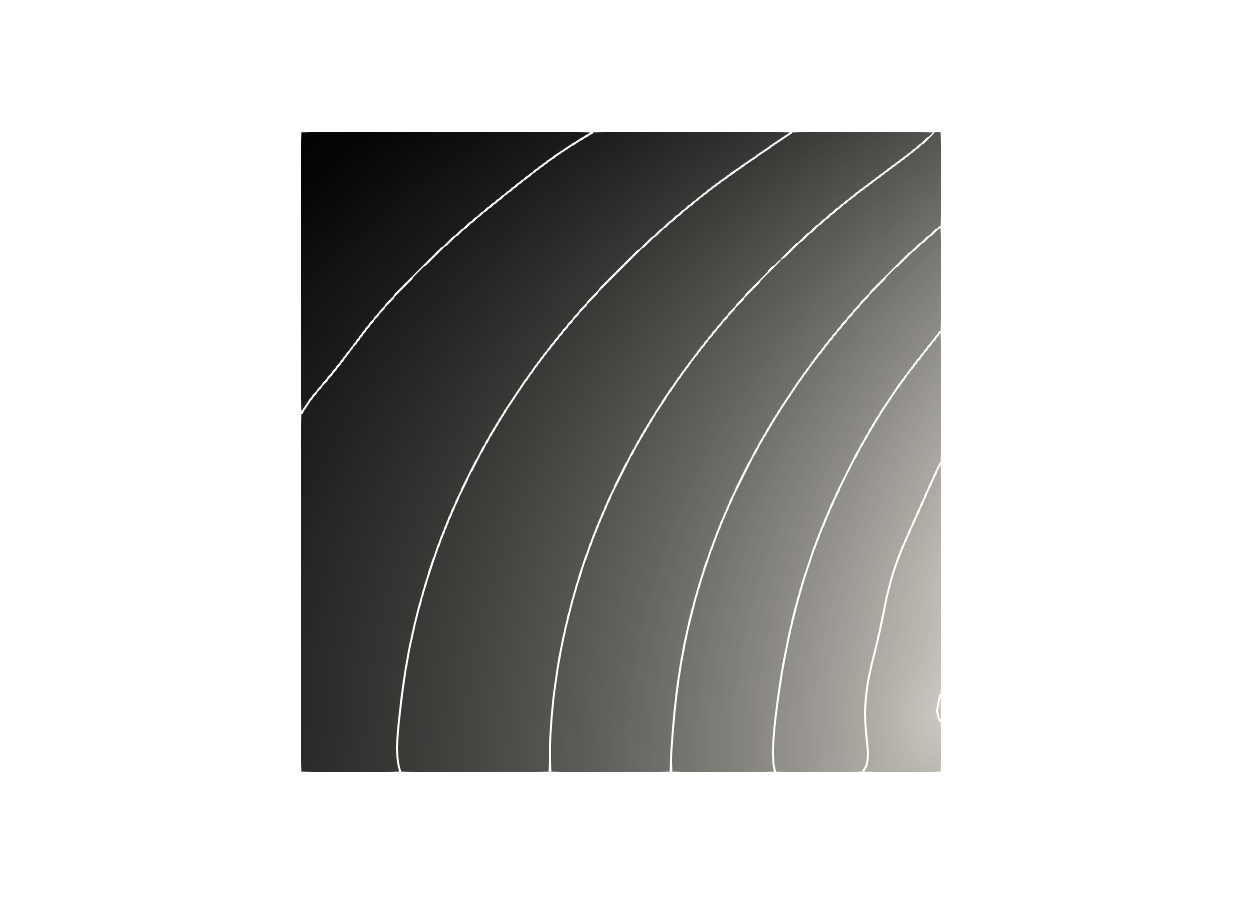}\\
$u_1\in[0.54,2.72]$ & $\widehat u_1\in[-0.15,0.15]$ & $\widehat u_1\in[0.60,2.64]$ & $\widehat u_1\in[0.57,2.71]$  & $\widehat u_1\in[0.58,2.73]$ \\[1.5ex]
\includegraphics[trim={11cm 5cm 10cm 4cm},clip,width=0.16\textwidth]{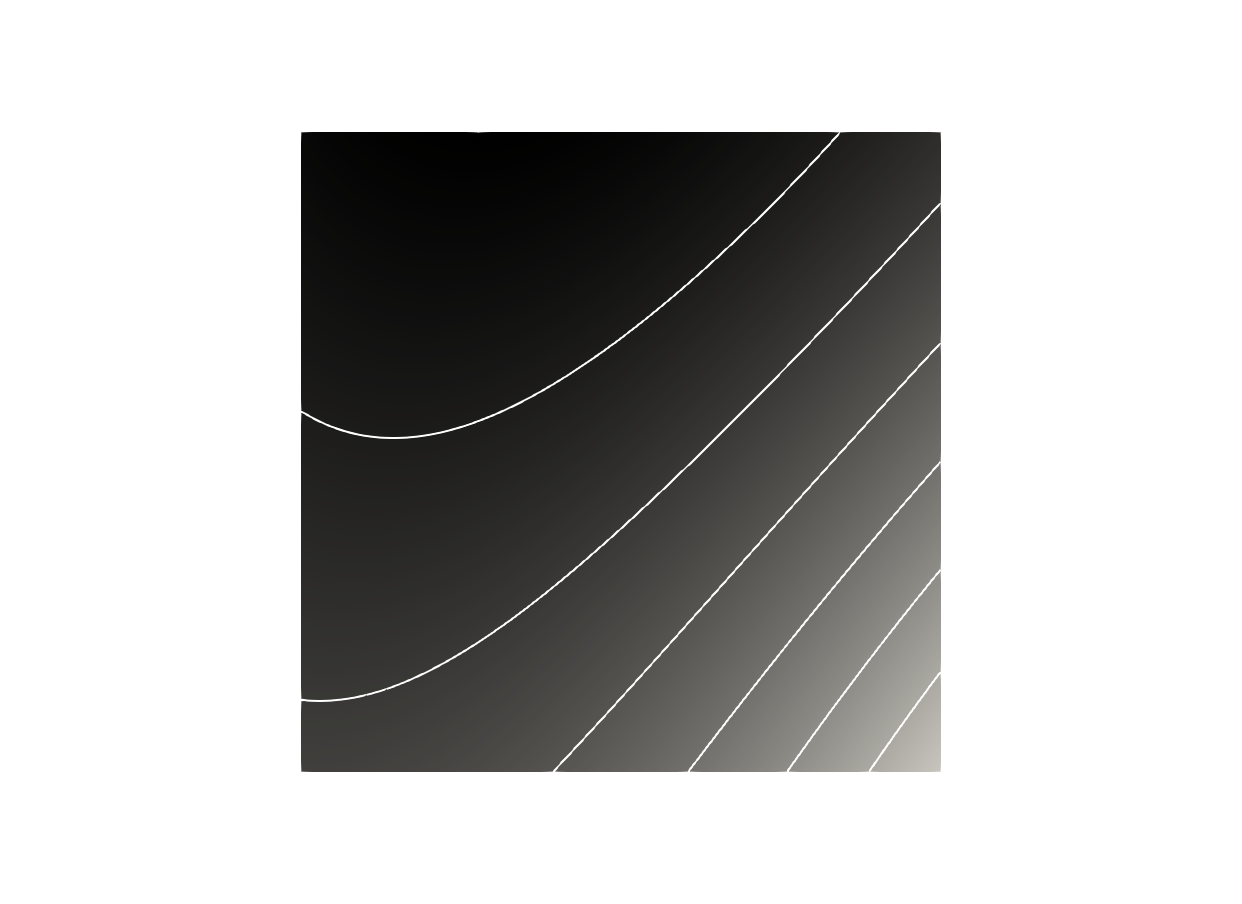} & 
\includegraphics[trim={11cm 5cm 10cm 4cm},clip,width=0.16\textwidth]{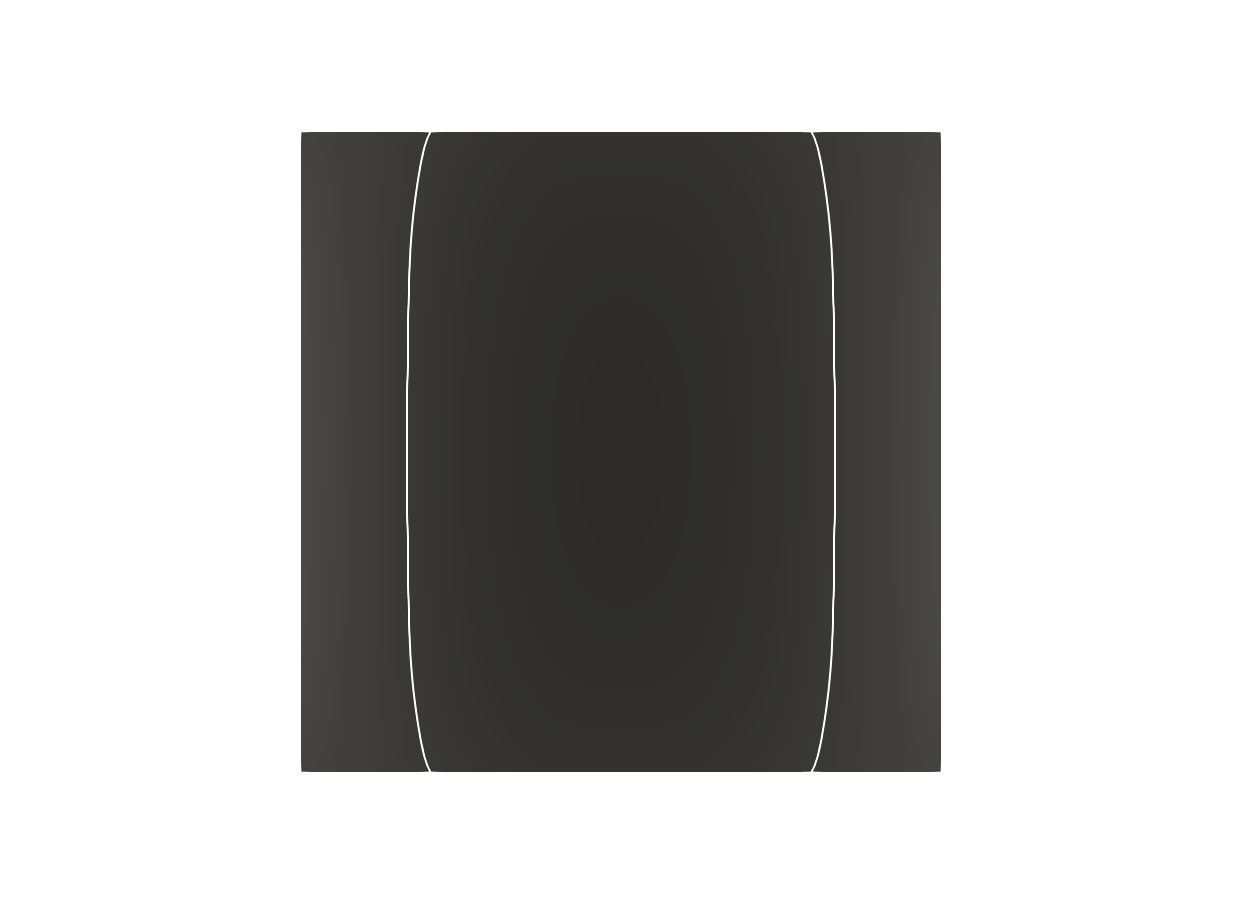} &
\includegraphics[trim={11cm 5cm 10cm 4cm},clip,width=0.16\textwidth]{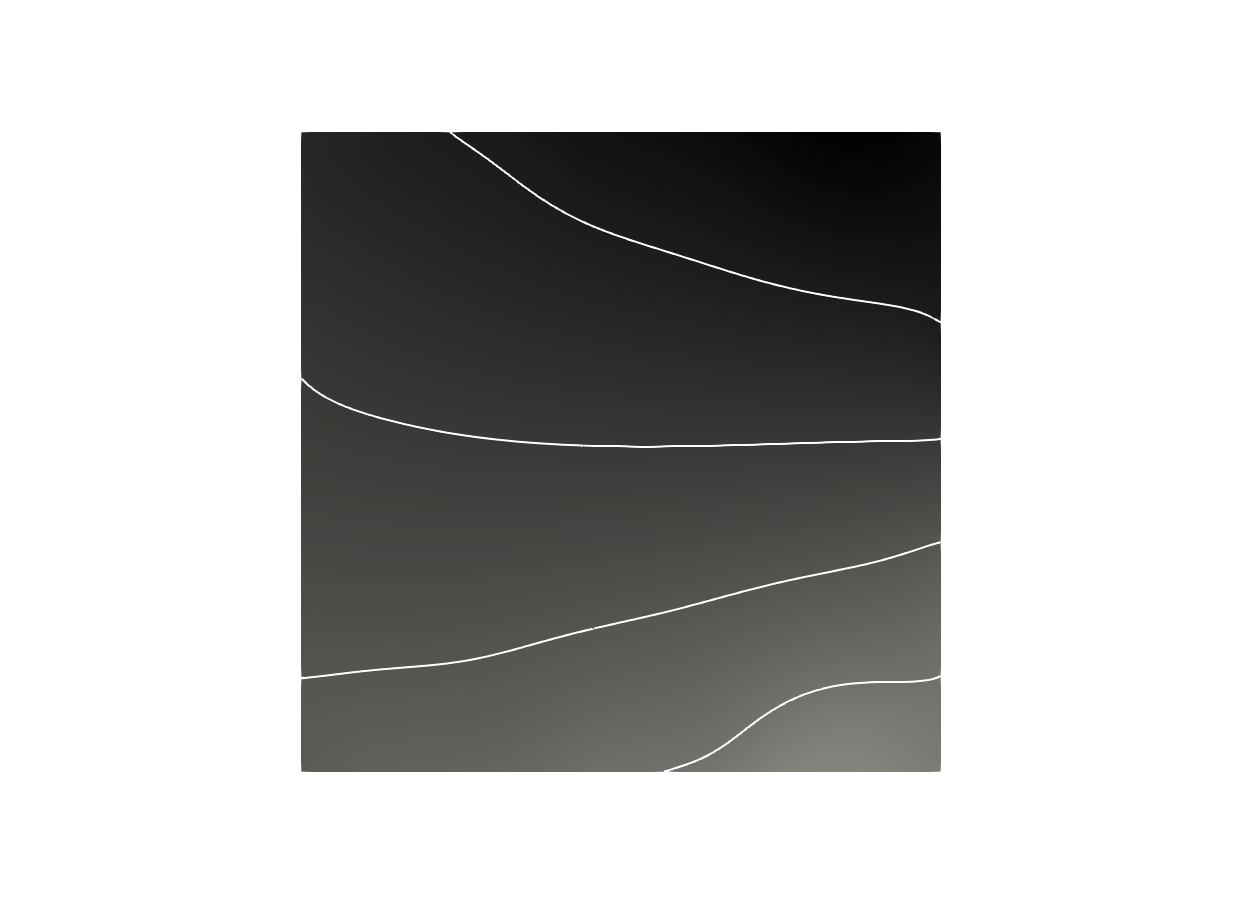}&
\includegraphics[trim={11cm 5cm 10cm 4cm},clip,width=0.16\textwidth]{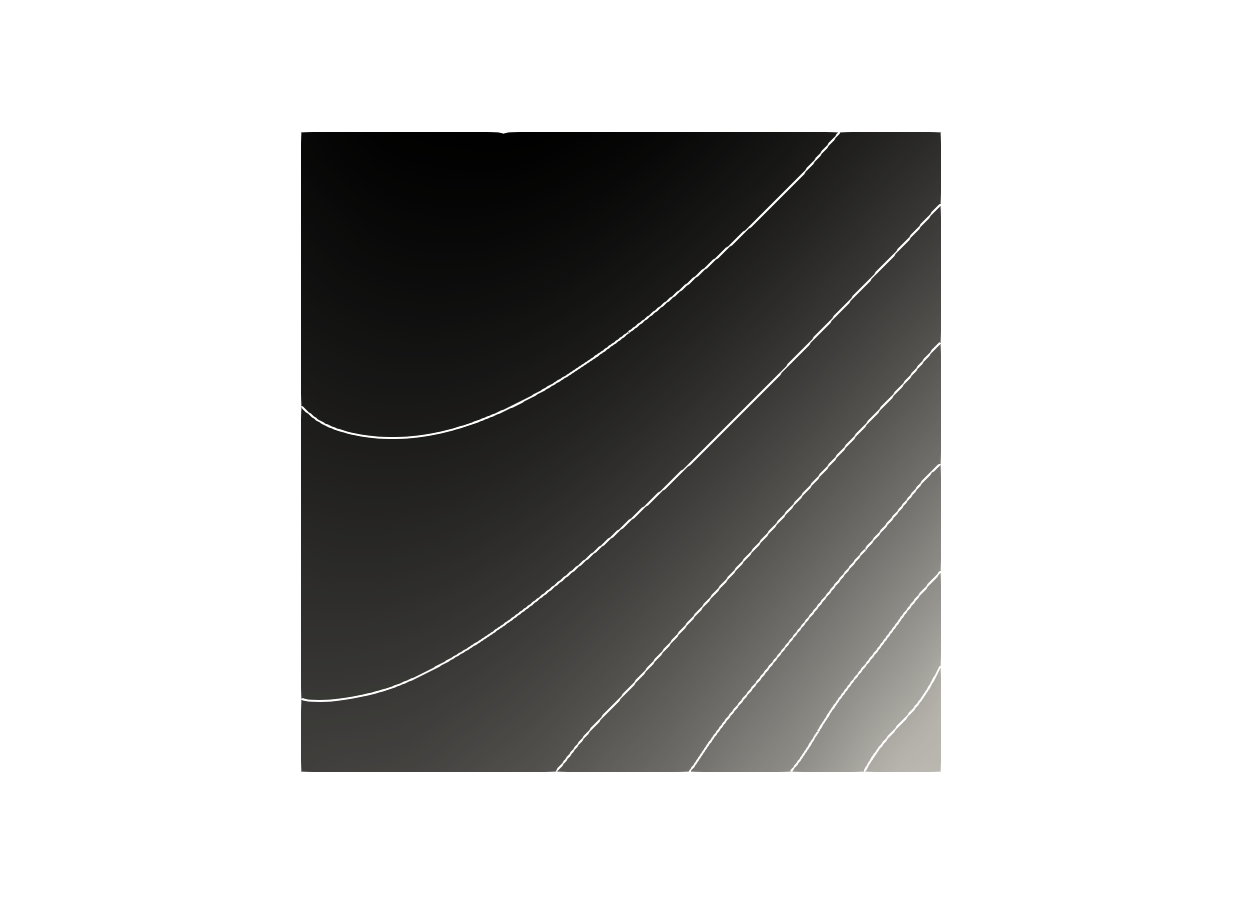}&
\includegraphics[trim={11cm 5cm 10cm 4cm},clip,width=0.16\textwidth]{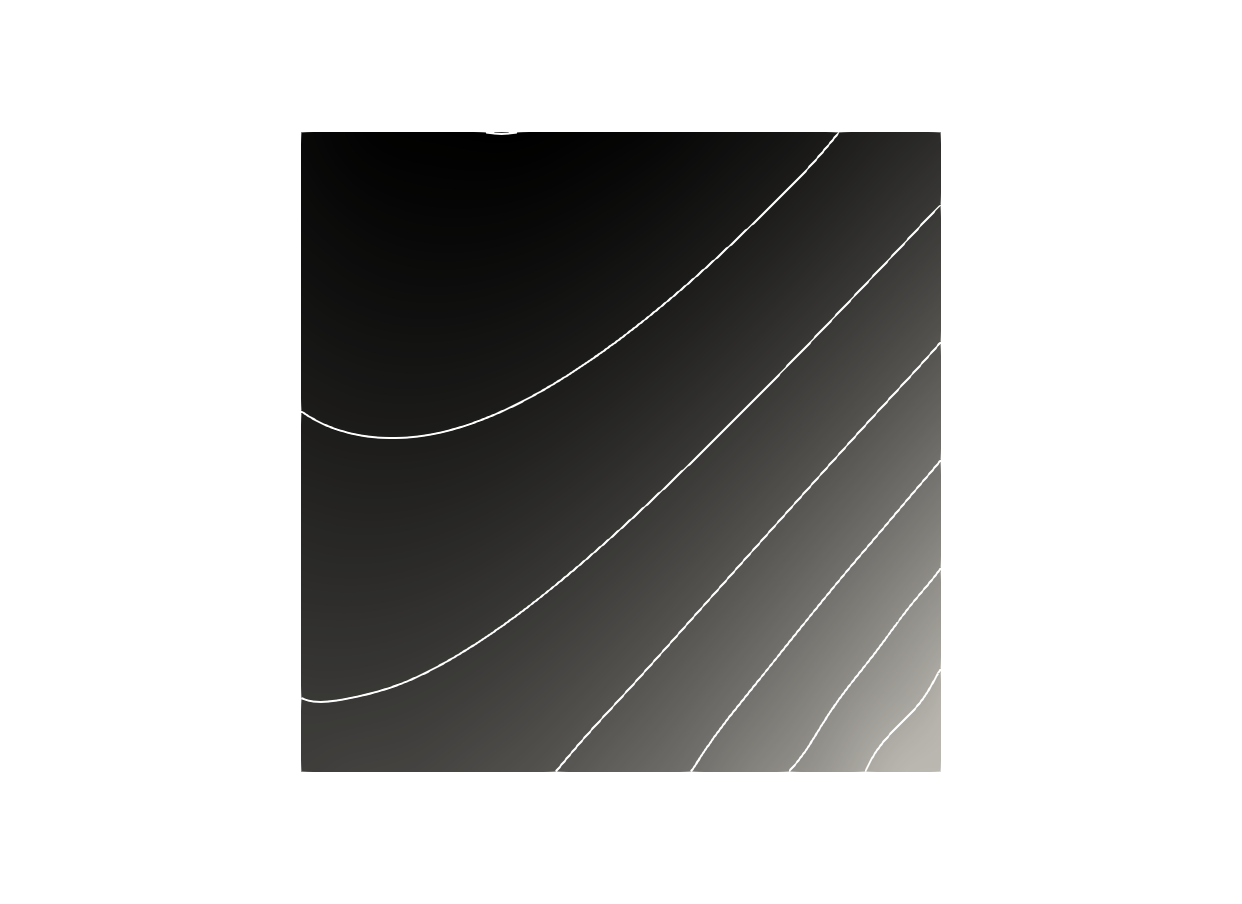}\\
$u_2\in[-0.96,2.00]$ & $\widehat u_2\in[-0.31,0.14]$ & $\widehat u_2\in[-0.93,1.04]$ & $\widehat u_2\in[-0.96,1.79]$  &$\widehat u_2\in[-0.96,1.81]$\\[1.5ex]
\includegraphics[trim={11cm 5cm 10cm 4cm},clip,width=0.16\textwidth]{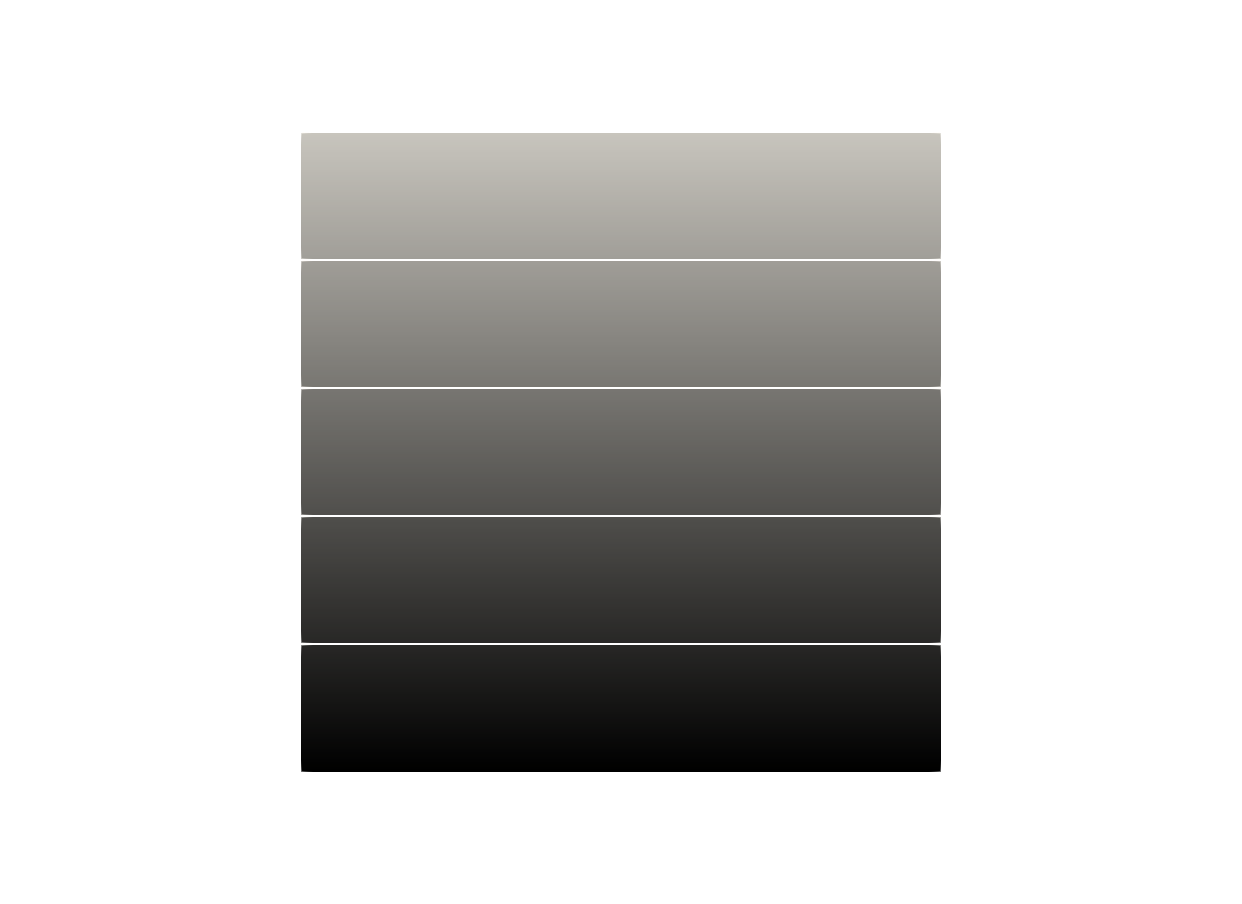} & 
\includegraphics[trim={11cm 5cm 10cm 4cm},clip,width=0.16\textwidth]{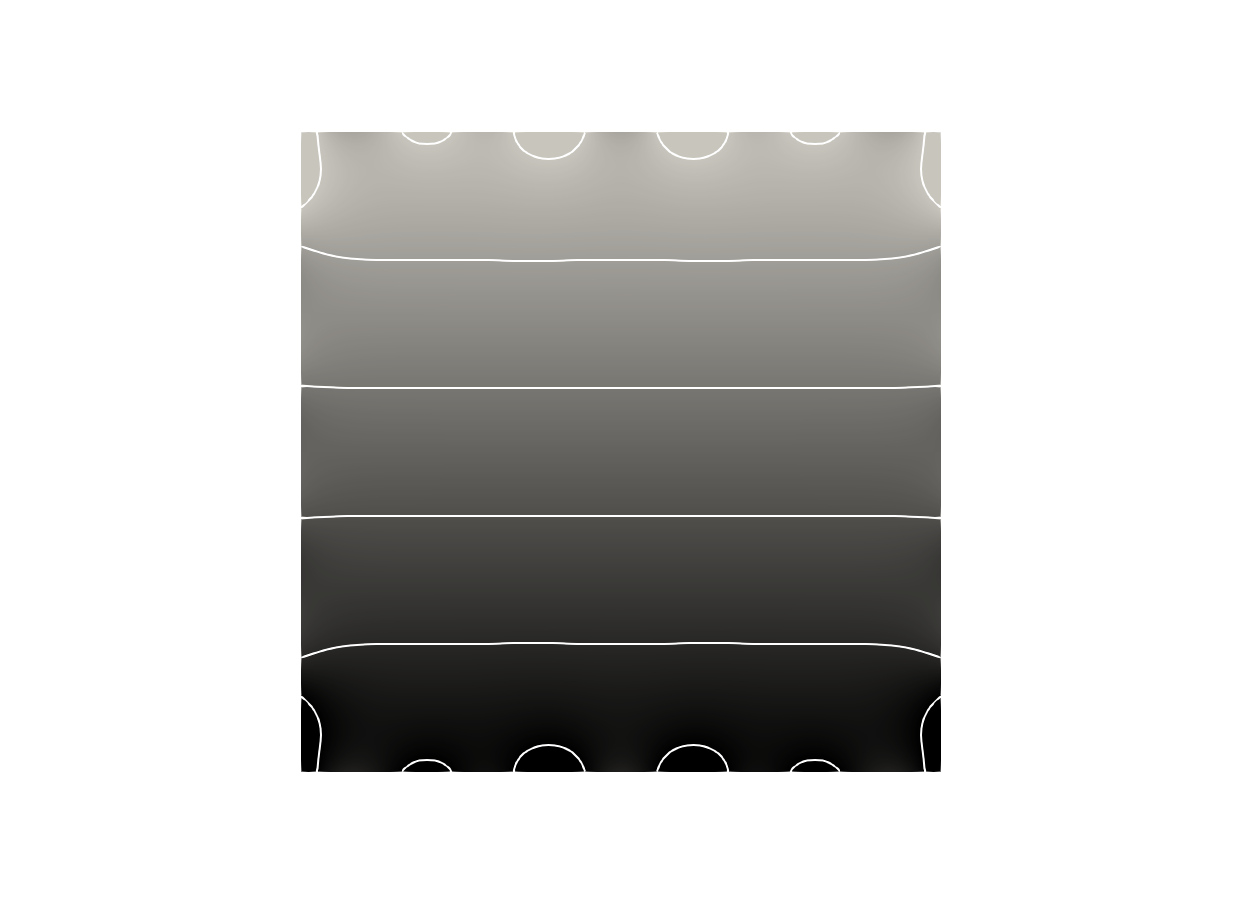} &
\includegraphics[trim={11cm 5cm 10cm 4cm},clip,width=0.16\textwidth]{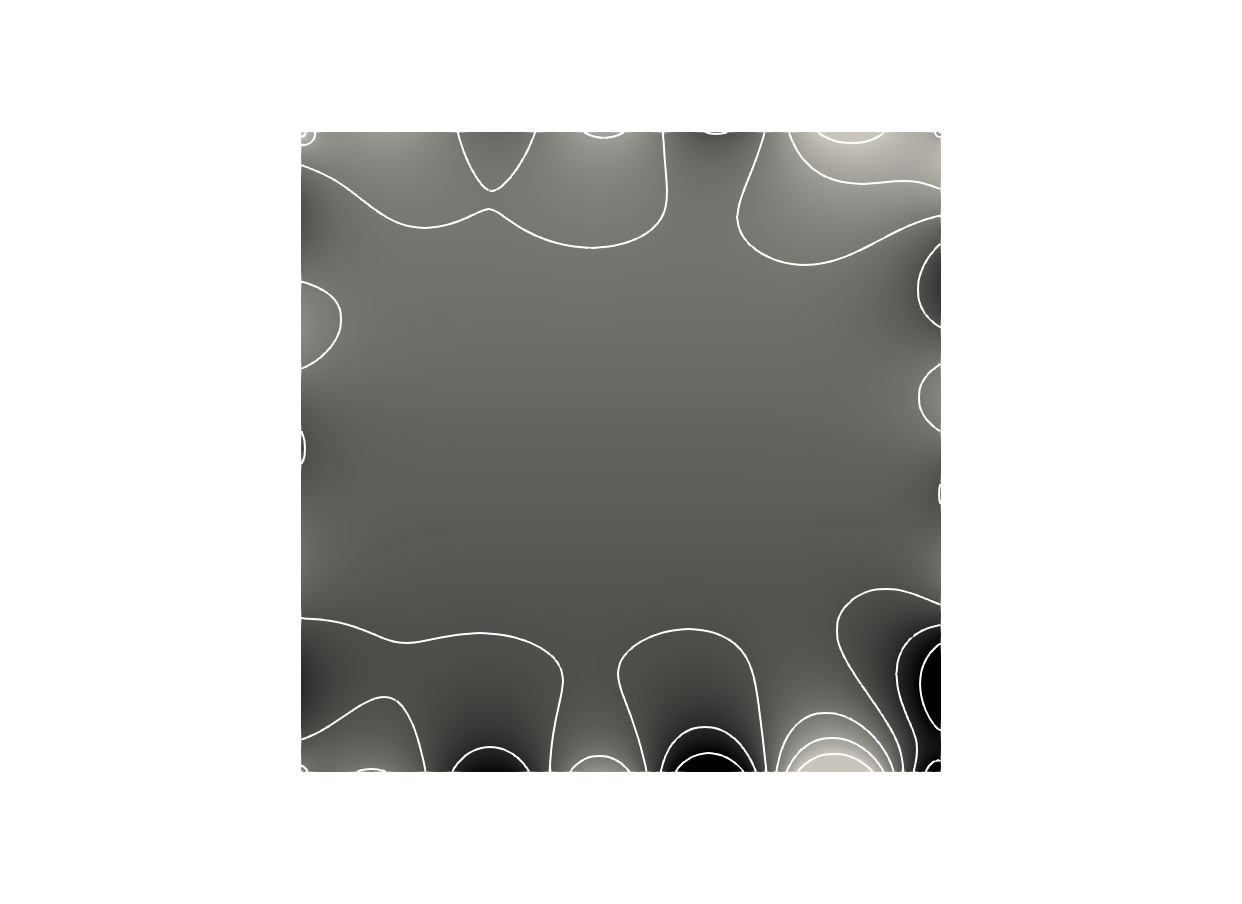}&
\includegraphics[trim={11cm 5cm 10cm 4cm},clip,width=0.16\textwidth]{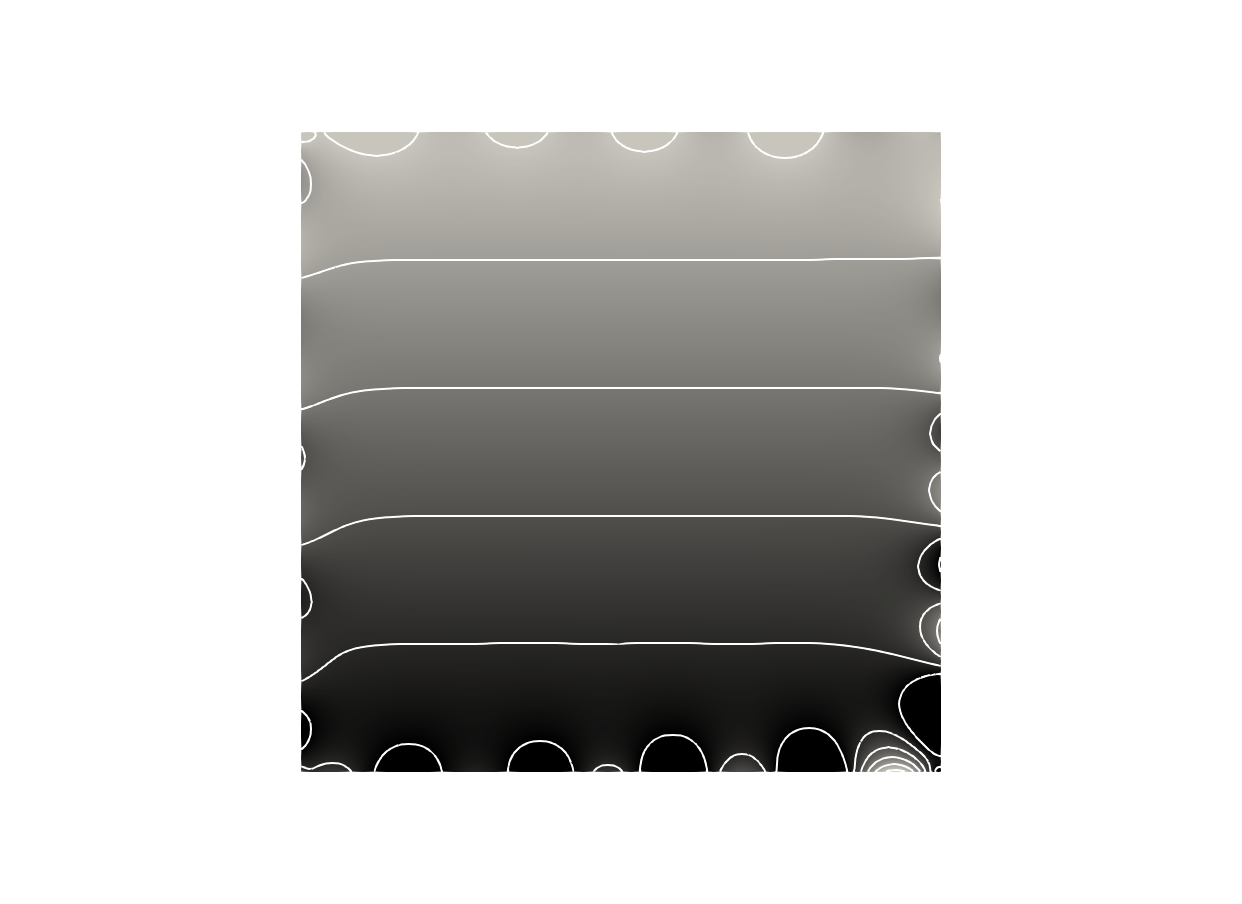}&
\includegraphics[trim={11cm 5cm 10cm 4cm},clip,width=0.16\textwidth]{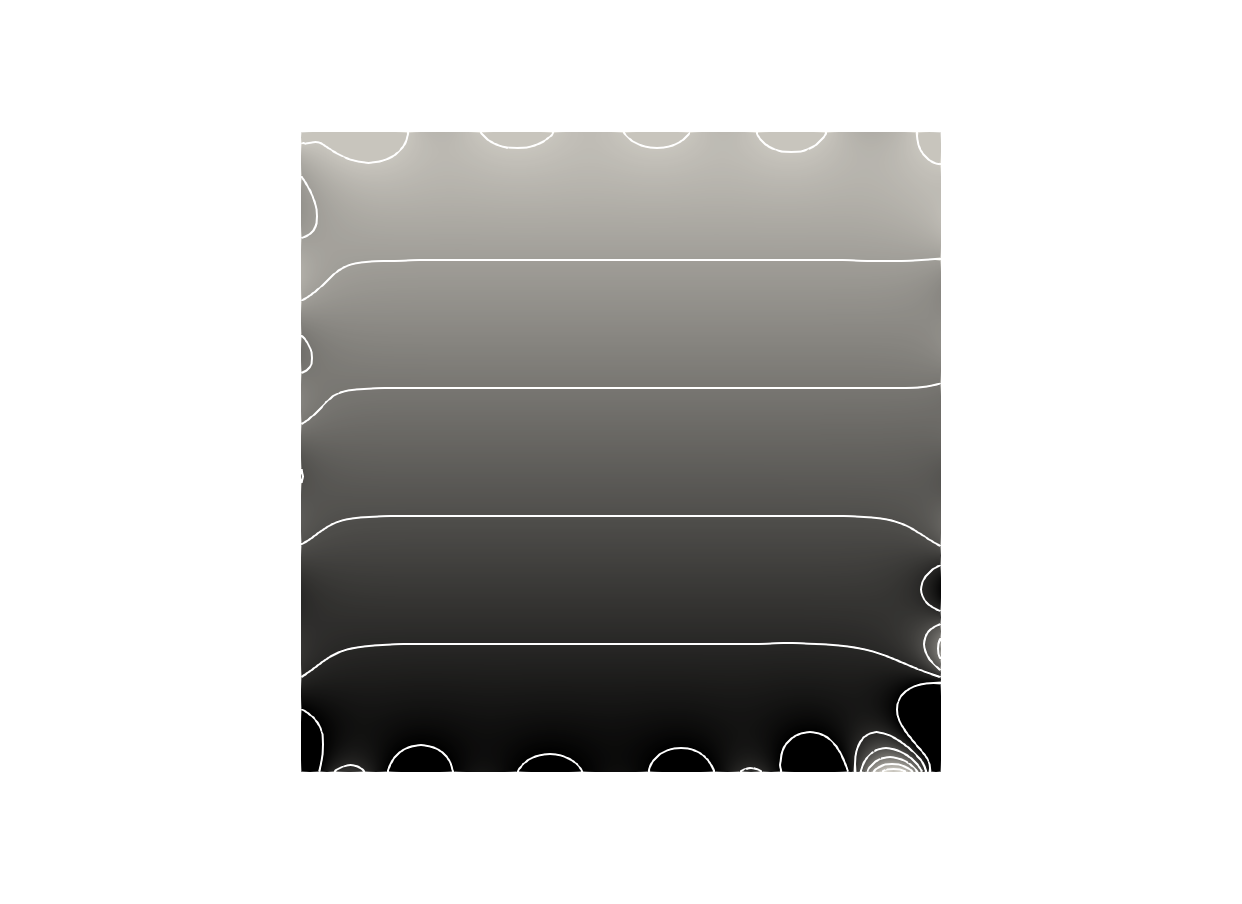}\\
$p\in[-2.00,2.00]$ & $\widehat p\in[-3.39,3.39]$ & $\widehat p\in[-4.50,4.11]$ & $\widehat p\in[-5.30,4.43]$ & $\widehat p\in[-7.28,4.01]$
\end{tabular}
\end{center}
\caption{Exact solution (col 1) and recovery solution (col 2 to 5) for the case \eqref{num:case1} when $(m_{u_1},m_{u_2},m_p)$ is equal to $(0,0,36)$, $(36,0,0)$, $(36,36,0)$ and $(36,36,36)$ using $n=8$ refinements for the mesh. From top to bottom, the rows correspond to the recovery of $u_1$, $u_2$ and $p$. For the color map, black corresponds to $\min_{\vx\in\Omega}v(\vx)$ and white to $\max_{\vx\in\Omega}v(\vx)$ for $v\in\{u_1,u_2,p\}$, and the constants for the level curves are the same in each row.} \label{fig:case1_mu_mp}
\end{figure}

\begin{table}[htbp]
    \centering
    \begin{tabular}{|c|c|c|c|c|c|c|c|c|c|}
    \hline
$(m_{u_1},m_{u_2},m_p)$ & (0,0,36) & (36,0,0) & (36,36,0) & (36,36,36) \\ 
\hline
$\text{err}_u$ & 3.12811 & 1.59357 & 0.28662 & 0.27723 \\
\hline
$\text{err}_p$ & 0.08480 & 0.87312 & 0.23818 & 0.22850 \\
\hline
$\text{err}$ & 3.12925 & 1.81708 & 0.37267 & 0.35926 \\
\hline
    \end{tabular}
    \caption{Recovery error by the discrete approximations depicted in Figure~\ref{fig:case1_mu_mp}.}
    \label{tab:case1_mu_mp}
\end{table}

The results for the recovery of the velocity are as expected. Indeed, Figure~\ref{fig:case1_mu_mp} indicates that the velocity is poorly recovered when there is no measurement of the velocity (col 2), only the first component $u_1$ is accurately recovered when we have measurements of the first component only (col 3), while having measurements of both components leads to an accurate recovery of the velocity field (cols 4 and 5). Regarding the recovery of the pressure, we observe that the recovered solution is accurate inside the domain but present oscillations near the boundary. Moreover, the quality of the recovered pressure deteriorates when adding velocity measurements, namely the oscillations are amplified. Note that this behaviour is observed in most numerical tests performed and is not ruled out by Theorem~\ref{thm:OR}. However, we expect from Theorem~\ref{thm:OR} that up to the space discretization error, the total recovery error $\textrm{err}$ does not increase as the number of measurements increases; which is indeed observed in Table~\ref{tab:case1_mu_mp} (last row).

We report in Table~\ref{tab:case1_errors} the recovery errors $\text{err}_u$ and $\text{err}_p$  for a mesh with $n=6$ refinements. We set $m_{u_1}=m_{u_2}=m_u$ (i.e., we use the same number of measurements of each component of the velocity field) and analyze the effect of increasing the number of velocity and pressure measurements. We observe that when $m_u=0$ (no measurements of the velocity), the pressure recovery error decreases as the number of pressure measurement increases while the velocity recovery error is more or less constant. On the other extreme, when $m_p=0$ (no measurements of the pressure), we see that increasing the number of velocity measurements has a positive impact on both the velocity and the pressure. Moreover, we observe that when $m_p>0$, the pressure recovery error first increases and then decreases when more velocity measurements are added. Again the total recovery errors decrease as the number of measurements increases.

\begin{table}[htbp]
    \centering
    \begin{tabular}{|c|c|c|c|c|c||c|c|c|c|c|c|}
    \cline{2-11}
\multicolumn{1}{c|}{} & \multicolumn{5}{c||}{$\text{err}_u$} & \multicolumn{5}{|c|}{$\text{err}_p$}\\
 \hline
 $m_u\setminus m_p$ & 0 & 4 & 16 & 36 & 64 & 0 & 4 & 16 & 36 & 64 \\
     \hline 
0 & --       & 3.2739  & 3.1324 & 3.1281 & 3.1287  & --    & 0.6234 & 0.2772 & 0.0849 & 0.0460\\
4 & 1.7727   & 1.6270  & 0.8222 & 0.7291 & 0.7365 & 1.4252 & 1.2337 & 0.5022 & 0.1430 & 0.0730 \\
16 & 0.6324  & 0.6203  & 0.4957 & 0.3165 & 0.2536 & 0.5073 & 0.4931 & 0.3985 & 0.2406 & 0.1517\\
36 & 0.2866  & 0.2874  & 0.2795 & 0.2772 & 0.2155 & 0.2383 & 0.2387 & 0.2333 & 0.2286 & 0.1654\\
64 & 0.2032  & 0.1977  & 0.1801 & 0.2201 & 0.1893 & 0.1657 & 0.1675 & 0.1418 & 0.1855 & 0.1443\\
\hline
    \end{tabular}
  \caption{Velocity recovery error $\text{err}_u$ and pressure recovery error $\text{err}_p$ for the case \eqref{num:case1}, for which $\|\vu\|_{H^1(\Omega)}=3.274$ and $\|p\|_{L^2(\Omega)}=1.155$, using $n=6$ refinements for the mesh.}
    \label{tab:case1_errors}
\end{table}

We now move to the case \eqref{num:case2}, for which $\vf\ne\mathbf{0}$ and thus the solution $(\vu_{\vf},p_{\vf}) \in H^1_0(\Omega)^2\times L^2_0(\Omega)$ to \eqref{def:pb1} is not trivial. Table~\ref{tab:case2_errors} reports the recovery errors $\text{err}_u$ and $\text{err}_p$ for various numbers of velocity and pressure measurements, again using a mesh with $n=6$ refinements. The results are similar to those obtained for the case \eqref{num:case1}, except that now the pressure recovery error decreases monotonically as $m_u$ increases when $m_p\le 16$. 

\begin{table}[htbp]
    \centering
    \begin{tabular}{|c|c|c|c|c|c||c|c|c|c|c|c|}
    \cline{2-11}
\multicolumn{1}{c|}{} & \multicolumn{5}{c||}{$\text{err}_u$} & \multicolumn{5}{|c|}{$\text{err}_p$}\\
 \hline
 $m_u\setminus m_p$ & 0 & 4 & 16 & 36 & 64 & 0 & 4 & 16 & 36 & 64 \\
     \hline 
0 & --     & 3.1996 & 1.5363 & 1.1077 & 1.0743 & --     & 2.0747 & 0.7921 & 0.2415 & 0.1253 \\
4 & 1.1007 & 1.1007 & 0.8068 & 0.5965 & 0.5967 & 0.7089 & 0.7089 & 0.5822 & 0.1992 & 0.1055\\
16 & 0.6807 & 0.6807 & 0.5893 & 0.3604 & 0.2915 & 0.5395 & 0.5395 & 0.4511 & 0.2721 & 0.1901  \\
36 & 0.3114 & 0.3114 & 0.3190 & 0.3344 & 0.2489 & 0.2618 & 0.2618 & 0.2675 & 0.2792 & 0.1888\\
%49 & 0.2912 & 0.2912 & 0.3036 & 0.2797 & 0.2379 & 0.2494 & 0.2494 & 0.2582 & 0.2360 & 0.1761 \\
64 & 0.2413 & 0.2413 & 0.2258 & 0.2611 & 0.2394 & 0.1875 & 0.1875 & 0.1705 & 0.2228 & 0.1876 \\
\hline
    \end{tabular}
    \caption{Velocity recovery error $\text{err}_u$ and pressure recovery error $\text{err}_p$ for the case \eqref{num:case2}, for which $\|\vu\|_{H^1(\Omega)}=3.220$ and $\|p\|_{L^2(\Omega)}=1$, using $n=6$ refinements for the mesh.}
    \label{tab:case2_errors}
\end{table}

\subsubsection{Influence of the regularity parameter $s$} \label{sec:vary_s}

We now study the effect of the parameter $s$, which dictates the regularity assumption for the trace of the velocity on $\Gamma$.
We use the same setting as in the previous section but focus on the case \eqref{num:case2}. 
However, for non integral $s$, we must resort to the algorithm described in Section~\ref{subsec:fractional_s} and we use \eqref{eqn:Balak_sinc_Twh} with $k=0.4$ for the approximation of $A_s^{-1}$.
Table~\ref{tab:case2_square_vary_s} reports the recovery errors for different numbers of measurements and different values of $s$. 
Note that the functions to be recovered are smooth inside $\Omega$. However, since $\Omega$ is a square domain, the value of $s$ is limited to $s<3/2$ for the definition of $H^s(\Gamma)$.
We observe in all cases that the larger the value of $s$, the smaller the recovery error. Moreover, increasing the value of $s$ has the effect of reducing the oscillations in the pressure that can be observed near the boundary of the domain, see Figure~\ref{fig:case2_pressure_vary_s} for an illustration.

\begin{table}[htbp]
\centering
\begin{tabular}{|c|c|c|c|c|c|c|}
    \cline{2-7}
\multicolumn{1}{c}{ } & \multicolumn{2}{|c|}{$m_u=0$, $m_p=64$} & \multicolumn{2}{|c|}{$m_u=64$, $m_p=0$} & \multicolumn{2}{|c|}{$m_u=16$, $m_p=16$} \\	
 \hline
$s$	& $\text{err}_u$ & $\text{err}_p$ & $\text{err}_u$ & $\text{err}_p$ & $\text{err}_u$ & $\text{err}_p$ \\
\hline
0.6 & 1.44987 & 0.26324 & 0.31373 & 0.27224 & 0.70380 & 0.55222 \\
0.8 & 1.28371 & 0.17485 & 0.29098 & 0.23551 & 0.66682 & 0.51844 \\
1.0 & 1.07426 & 0.12531 & 0.24127 & 0.18747 & 0.58931 & 0.45114 \\
1.2 & 0.81662 & 0.09599 & 0.22255 & 0.18016 & 0.50884 & 0.38185 \\
1.4 & 0.62243 & 0.07565 & 0.19807 & 0.15376 & 0.44041 & 0.32288 \\
\hline
\end{tabular}
\caption{Velocity recovery error $\text{err}_u$ and pressure recovery error $\text{err}_p$ for the case \eqref{num:case2} using $n=6$ refinements for the mesh and different values of $m_u$, $m_p$ and $s$.}
    \label{tab:case2_square_vary_s}
\end{table}

\begin{figure}[htbp]
    \centering
    \includegraphics[trim={11cm 5cm 10cm 4cm},clip,width=0.16\textwidth]{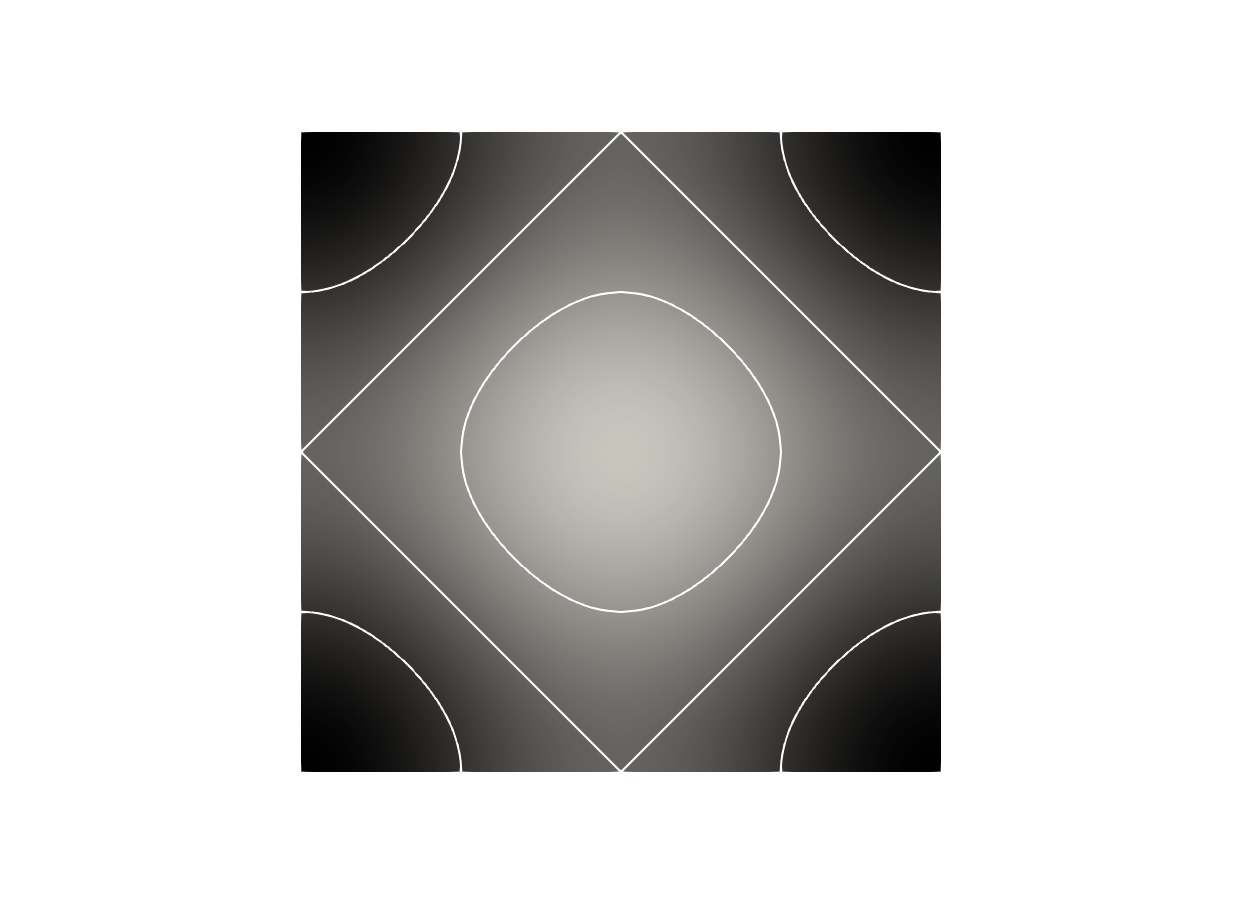}    
    \qquad \includegraphics[trim={11cm 5cm 10cm 4cm},clip,width=0.16\textwidth]{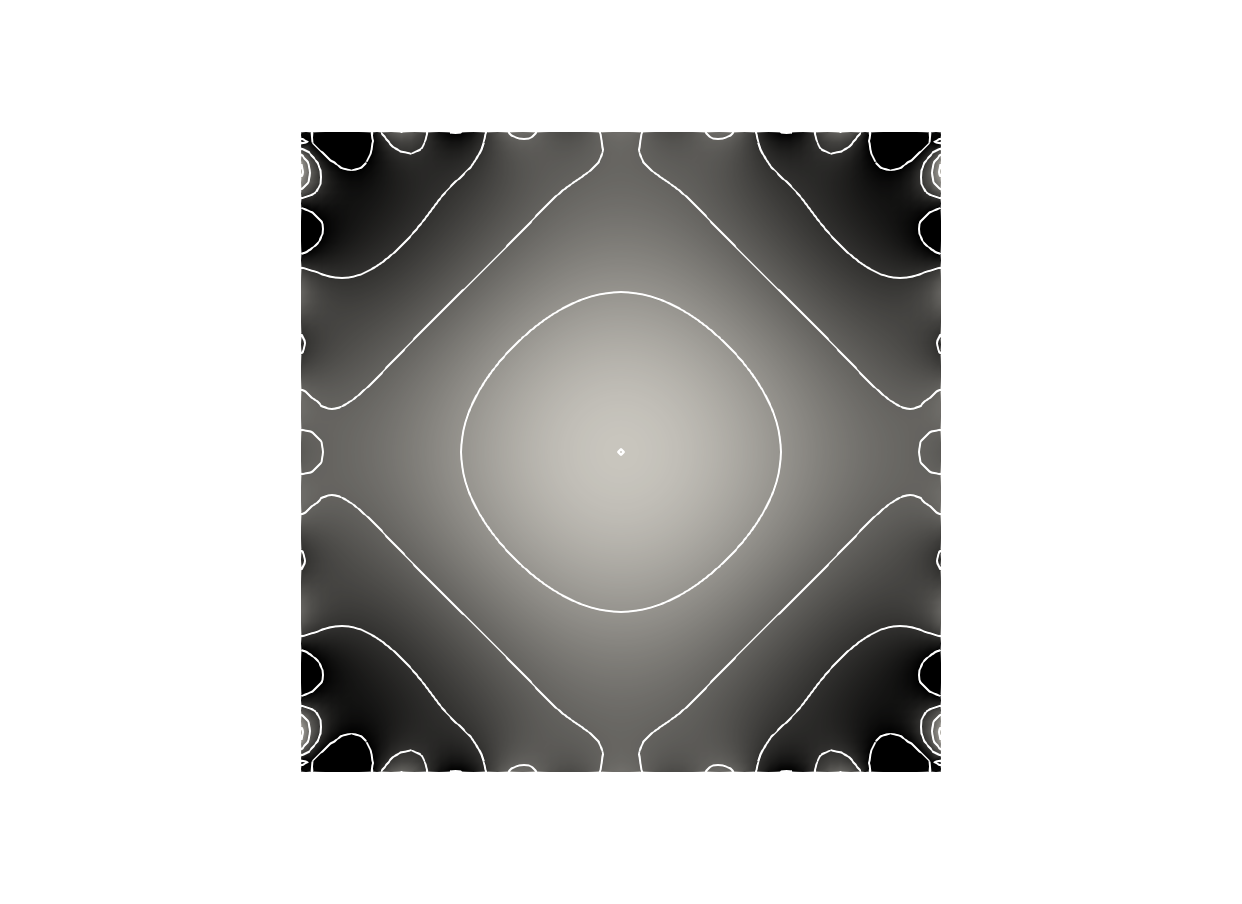}
    \qquad \includegraphics[trim={11cm 5cm 10cm 4cm},clip,width=0.16\textwidth]{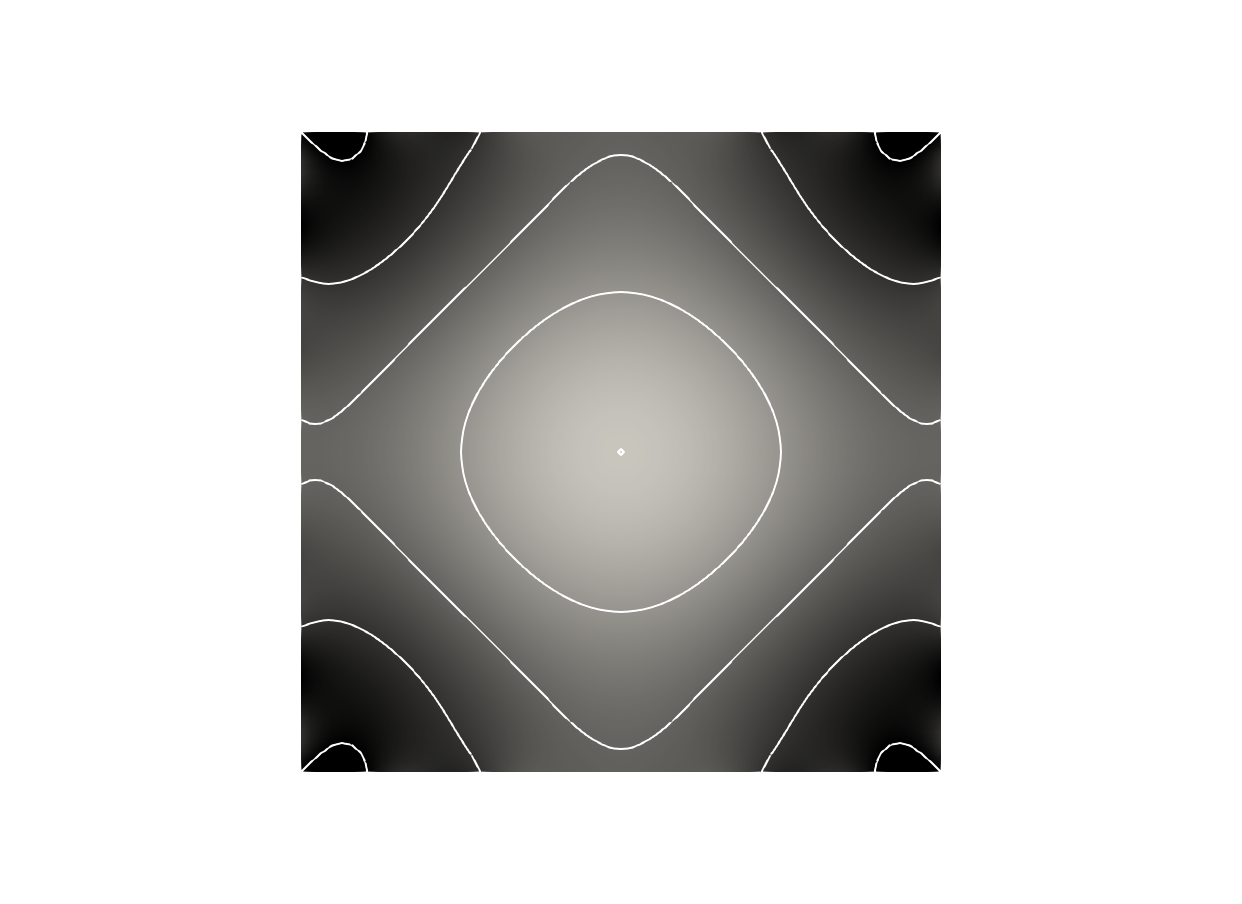}
  \caption{Recovery of the pressure for the case \eqref{num:case2} using $n=6$ refinements for the mesh: exact solution (left) and recovery solution with $m_u=0$ and $m_p=64$ using $s=0.6$ (middle) and $s=1.4$ (right). For the color map, black corresponds to $\min_{\vx\in\Omega}p(\vx)=-2$ and white to $\max_{\vx\in\Omega}p(\vx)=2$. Note that when $s=0.6$, $\widehat p\in[-5.14,2.00]$ while $\widehat p \in[-2.86,2.00]$ when $s=1.4$.}
    \label{fig:case2_pressure_vary_s}
\end{figure}

\subsection{Square domain with a hole} \label{sec:square_hole}

In this subsection, we consider an example in which partial knowledge of the boundary data is available. The domain is the unit square with a hole in the middle, namely $\Omega=(0,1)^2\setminus B$,  where $B$ is the closed ball of radius $0.1$ centered at $(0.5,0.5)$. We assume that the velocity is known on the outside boundary but unknown on the boundary of $B$. The functions to be recovered are the ones in \eqref{num:case2}, but, due to the hole, the pressure does not have a vanishing mean value. Moreover, we have $\|\vu\|_{H^1(\Omega)}=3.127$ and $\|p\|_{L^2(\Omega)}=0.941$.
For the measurements, we consider the Gaussian functionals \eqref{def:lambda_z} with centers as in \eqref{eqn:z_Gauss} but discarding the points that falls inside $B$.
We present results for two different uniform subdivisions: $n=4$ refinements corresponding to 10240 quadrilateral, 82944 degrees of freedom for $\vu$, and 10496 for $p$; $n=5$ refinements corresponding to 40960 quadrilaterals, 329728 degrees of freedom for the velocity, and 41472 for the pressure. Note that we do not use curved elements but rather a piecewise linear approximation of the circle. 

The velocity and pressure recovery errors for different numbers of measurements when using $n=4$ and $n=5$ are provided in Table~\ref{tab:case2_hole_n4} and \ref{tab:case2_hole_n5}, respectively. We find that with $n=4$ the errors $\text{err}_u$ and $\text{err}_p$ reach a plateau at about $2.45\cdot 10^{-3}$ and $1.91\cdot 10^{-4}$, respectively, after which increasing the number of measurements does not decrease the error. Since smaller recover errors are observed with $n=5$, we conclude that these plateaus are due to the finite element error. We also observe that without measurements of the pressure ($m_p=0$), the mean value of the pressure cannot be recovered and  the algorithm constructs a pressure approximation with vanishing mean value, leading to a large recovery error $\text{err}_p$. We illustrate the behaviour of the recovery error for the pressure by comparing  two measurement configurations in Figure~\ref{fig:case2_hole_pressure}.  We observe large oscillations near the hole when using $m_u=m_p=4$ measurements. These oscillations are significantly reduced when using $m_u=m_p=16$ measurements.

\begin{table}[htbp]
    \centering
    \begin{tabular}{|c|c|c|c|c|c|}
    \cline{2-6}
    \multicolumn{1}{c|}{} & \multicolumn{5}{c|}{$\text{err}_u$}\\
    \hline
        $m_u\setminus m_p$ & 0 & 4 & 8 & 16 & 24 \\
    \hline
0 & -- & 1.1091$\cdot 10^{-0}$ & 4.1886$\cdot 10^{-3}$ & 4.2215$\cdot 10^{-3}$ & 2.4567$\cdot 10^{-3}$ \\
4 & 2.5497$\cdot 10^{-1}$ & 2.5497$\cdot 10^{-1}$ & 2.7100$\cdot 10^{-3}$ & 2.4634$\cdot 10^{-3}$ & 2.4544$\cdot 10^{-3}$ \\
8 & 3.1285$\cdot 10^{-3}$ & 3.1323$\cdot 10^{-3}$ & 2.4558$\cdot 10^{-3}$ & 2.4543$\cdot 10^{-3}$ & 2.4544$\cdot 10^{-3}$ \\
16 & 2.4549$\cdot 10^{-3}$ & 2.4549$\cdot 10^{-3}$ & 2.4543$\cdot 10^{-3}$ & 2.4548$\cdot 10^{-3}$ & 2.4542$\cdot 10^{-3}$ \\
24 & 2.4542$\cdot 10^{-3}$ & 2.4542$\cdot 10^{-3}$ & 2.4543$\cdot 10^{-3}$ & 2.4542$\cdot 10^{-3}$ & 2.4543$\cdot 10^{-3}$ \\
\hline
\end{tabular}

\vspace*{0.4cm}

\begin{tabular}{|c|c|c|c|c|c|}
    \cline{2-6}
    \multicolumn{1}{c|}{} & \multicolumn{5}{c|}{$\text{err}_p$}\\
    \hline
    $m_u\setminus m_p$ & 0 & 4 & 8 & 16 & 24 \\
    \hline
0 & -- & 1.4431$\cdot 10^{-0}$ & 1.4527$\cdot 10^{-3}$ & 2.4817$\cdot 10^{-4}$ & 1.9117$\cdot 10^{-4}$ \\
4 & 1.8012$\cdot 10^{-0}$ & 2.8270$\cdot 10^{-1}$ & 1.3397$\cdot 10^{-3}$ & 3.1139$\cdot 10^{-4}$ & 1.9102$\cdot 10^{-4}$ \\
8 & 1.7789$\cdot 10^{-0}$ & 2.2452$\cdot 10^{-3}$ & 1.9174$\cdot 10^{-4}$ & 1.9099$\cdot 10^{-4}$ & 1.9102$\cdot 10^{-4}$ \\
16 & 1.7789$\cdot 10^{-0}$ & 2.0284$\cdot 10^{-4}$ & 1.9240$\cdot 10^{-4}$ & 2.0055$\cdot 10^{-4}$ & 1.9104$\cdot 10^{-4}$ \\
24 & 1.7789$\cdot 10^{-0}$ & 1.9104$\cdot 10^{-4}$ & 1.9252$\cdot 10^{-4}$ & 1.9101$\cdot 10^{-4}$ & 1.9108$\cdot 10^{-4}$ \\
\hline
\end{tabular}
\caption{Velocity recovery error $\text{err}_u$ and pressure recovery error $\text{err}_p$ for the domain with a hole using $n=4$ refinements for the mesh. In this setting, we have $\|\vu\|_{H^1(\Omega)}=3.127$, $\|p\|_{L^2(\Omega)}=0.941$ and $\mean(p)\not=0$.}
    \label{tab:case2_hole_n4}
\end{table}

\begin{table}[htbp]
    \centering
    \begin{tabular}{|c|c|c|c|c|c|}
    \cline{2-6}
    \multicolumn{1}{c|}{} & \multicolumn{5}{c|}{$\text{err}_u$}\\
    \hline
$m_u\setminus m_p$ & 0 & 4 & 8 & 16 & 24 \\
     \hline  
0 & -- & 1.1092$\cdot 10^{-0}$ & 3.5309$\cdot 10^{-3}$ & 3.5522$\cdot 10^{-3}$ & 8.6152$\cdot 10^{-4}$ \\
4 & 2.5503$\cdot 10^{-1}$ & 2.5503$\cdot 10^{-1}$ & 1.4741$\cdot 10^{-3}$ & 8.7577$\cdot 10^{-4}$ & 8.6019$\cdot 10^{-4}$ \\
8 & 2.1678$\cdot 10^{-3}$ & 2.1733$\cdot 10^{-3}$ & 8.6355$\cdot 10^{-4}$ & 8.6019$\cdot 10^{-4}$ & 8.6020$\cdot 10^{-4}$ \\
16 & 8.6017$\cdot 10^{-4}$ & 8.6029$\cdot 10^{-4}$ & 8.6040$\cdot 10^{-4}$ & 8.6020$\cdot 10^{-4}$ & 8.6017$\cdot 10^{-4}$ \\
24 & 8.6018$\cdot 10^{-4}$ & 8.6017$\cdot 10^{-4}$ & 8.6042$\cdot 10^{-4}$ & 8.6018$\cdot 10^{-4}$ & 8.6018$\cdot 10^{-4}$ \\
\hline
\end{tabular}

\vspace*{0.4cm}

\begin{tabular}{|c|c|c|c|c|c|}
    \cline{2-6}
    \multicolumn{1}{c|}{} & \multicolumn{5}{c|}{$\text{err}_p$}\\
    \hline
$m_u\setminus m_p$ & 0 & 4 & 8 & 16 & 24 \\
     \hline   
0 & -- & 1.4433$\cdot 10^{-0}$ & 1.4968$\cdot 10^{-3}$ & 1.1313$\cdot 10^{-4}$ & 4.1624$\cdot 10^{-5}$ \\
4 & 1.8012$\cdot 10^{-0}$ & 2.8281$\cdot 10^{-1}$ & 1.3824$\cdot 10^{-3}$ & 1.9465$\cdot 10^{-4}$ & 4.1802$\cdot 10^{-5}$ \\
8 & 1.7788$\cdot 10^{-0}$ & 2.2949$\cdot 10^{-3}$ & 5.0334$\cdot 10^{-5}$ & 4.1712$\cdot 10^{-5}$ & 4.1761$\cdot 10^{-5}$ \\
16 & 1.7788$\cdot 10^{-0}$ & 4.4890$\cdot 10^{-5}$ & 4.7888$\cdot 10^{-5}$ & 4.2479$\cdot 10^{-5}$ & 4.1787$\cdot 10^{-5}$ \\
24 & 1.7788$\cdot 10^{-0}$ & 4.1778$\cdot 10^{-5}$ & 4.8227$\cdot 10^{-5}$ & 4.1846$\cdot 10^{-5}$ & 4.1977$\cdot 10^{-5}$ \\
\hline
\end{tabular}
\caption{Velocity recovery error $\text{err}_u$ (top) and pressure recovery error $\text{err}_p$ (bottom) for the domain with a hole, for which $\|\vu\|_{H^1(\Omega)}=3.1269$ and $\|p\|_{L^2(\Omega)}=0.9414$ with $p$ not mean free, using $n=5$ refinements for the mesh.}
    \label{tab:case2_hole_n5}
\end{table}

\begin{figure}[htbp]
    \centering
    \includegraphics[width=4.8cm]{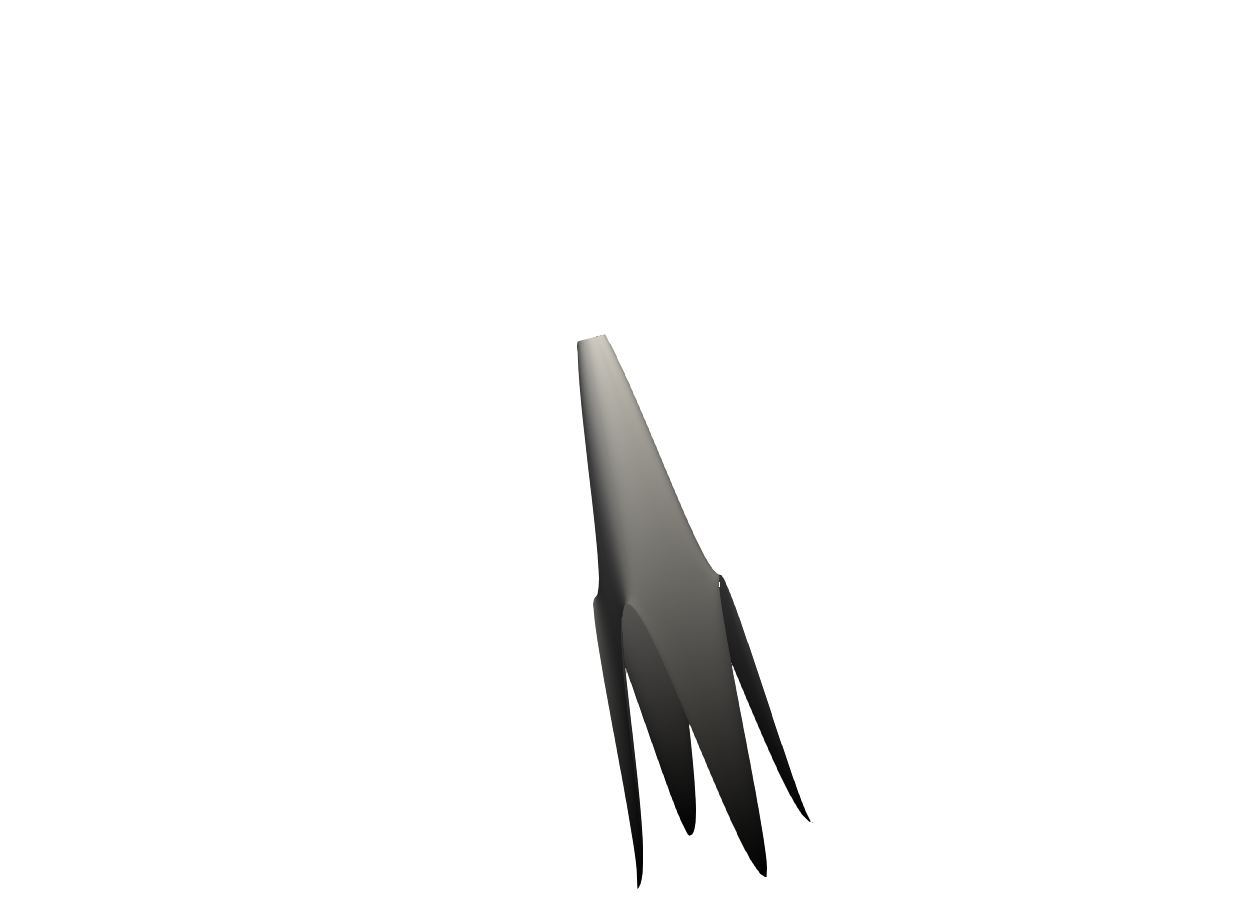}    
    \hspace{-2cm}\includegraphics[width=4.8cm]{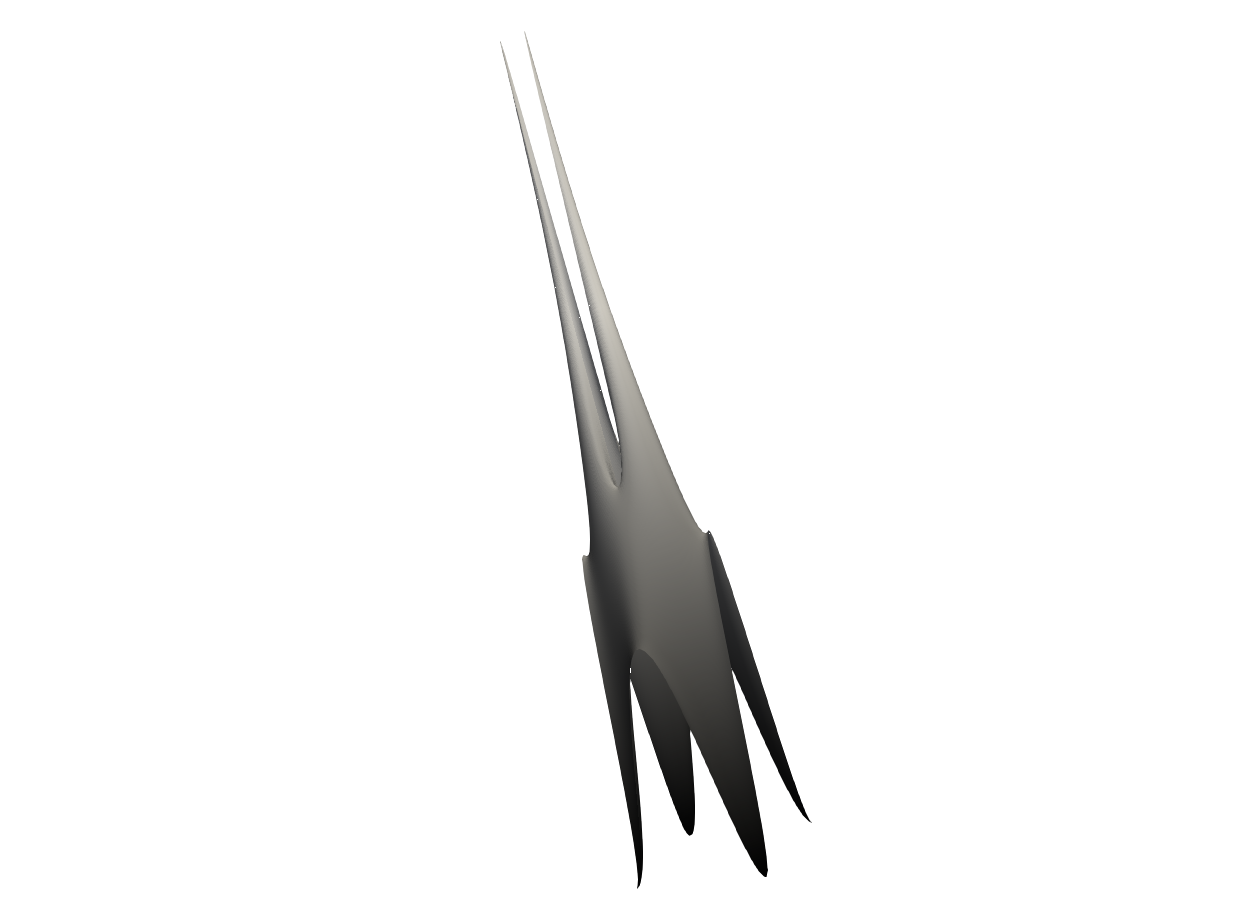}
    \hspace{-2cm}\includegraphics[width=4.8cm]{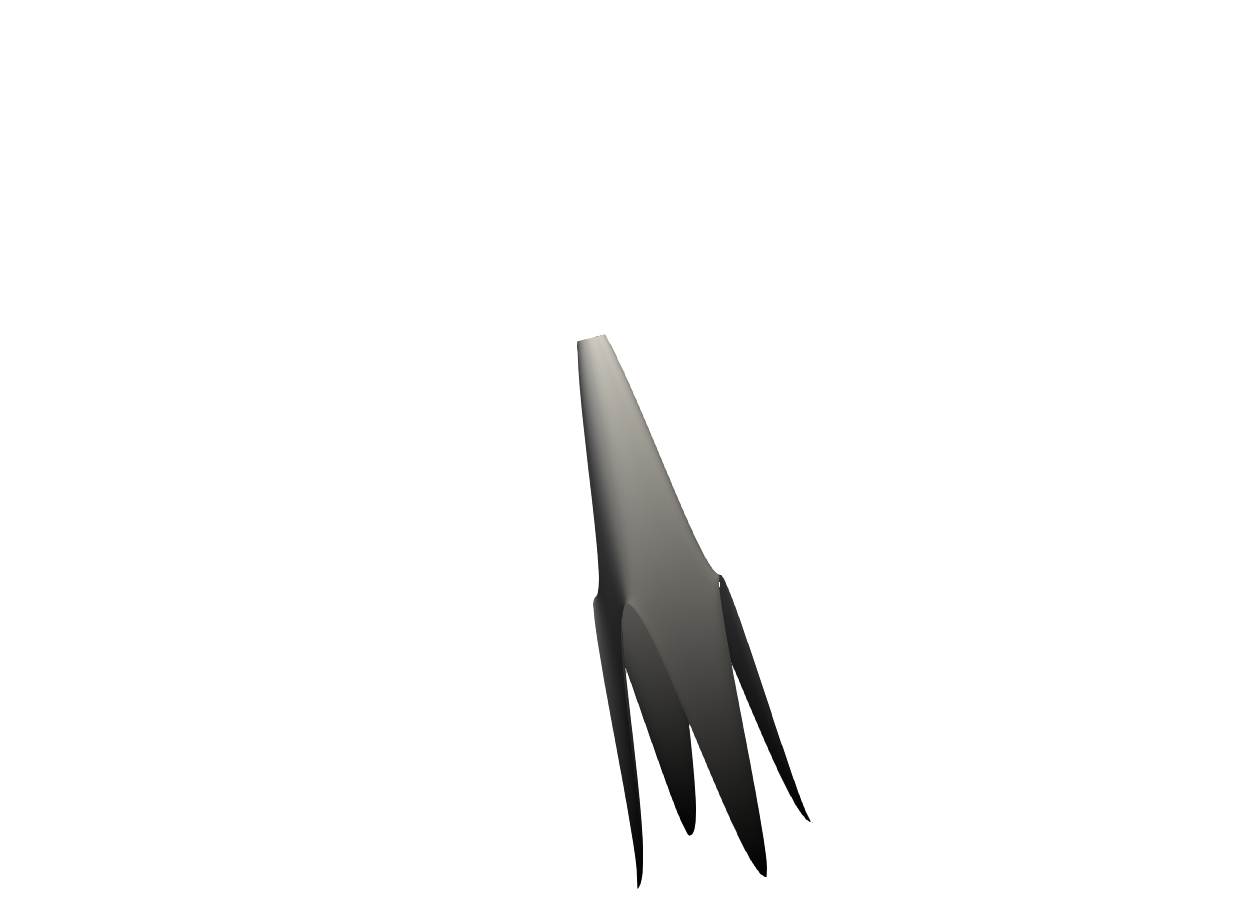}

  \caption{Recovery of the pressure for the domain with a hole using $n=5$ refinements for the mesh: exact solution (left), recovery using $m_u=m_p=4$ (middle) and $m_u=m_p=16$ (right). For the color map, black corresponds to $\min_{\vx\in\Omega}p(\vx)=-2$ and white to $\max_{\vx\in\Omega}p(\vx)=1.810$. Note that when $m_u=m_p=4$, $\widehat p\in[-2.00,3.96]$ while $\widehat p \in[-2.00,1.81]$ when $m_u=m_p=16$.}
    \label{fig:case2_hole_pressure}
\end{figure}

\subsection{Airfoil} \label{sec:airfoil}

We consider in this subsection the flow around an airfoil, and we are interested in the recovery of specific quantities of interest, namely the drag and lift coefficients. The geometry of the domain is described in Figure~\ref{fig:airfoil_mesh}-left. We assume that the boundary conditions are unknown around the airfoil and at the outflow, that $\vu=(1,0)^T$ for the remaining part of the boundary (the inflow, the top and the bottom), and that the forcing term is $\vf=\mathbf{0}$. Denoting by $\Gamma_A$ the boundary of the airfoil and following \cite[Equation (2.2)]{TT2000}, we compute the drag and lift coefficients using
\begin{equation} \label{def:cDcL}
    (c_D,c_L)^T = -\int_{\Gamma_A}\left( 2\veps(\vu)\vn-p\vn\right).
\end{equation}
Note that the drag and lift coefficients could be alternatively computed using a variational formulation, see \cite[Equation (2.9)]{TT2000}. The two formulations are equivalent at the continuous level, but the variational based formulation is more accurate after discretization, see for instance \cite{GLLS1997}. However, in our setting, the two approaches give comparable results and we only report those obtained using \eqref{def:cDcL}.

To compute the error in velocity, pressure, drag and lift coefficients, we use a numerically computed reference solution. The latter is obtained considering the same known boundary conditions as above, complemented by $\vu=\mathbf{0}$ on $\Gamma_A$ and $2\veps(\vu)\vn-p\vn=\mathbf{0}$ on the outflow. Note that the pressure is not mean-free in this case. Moreover, we use the same mesh for the reference and recovery solution, see Figure~\ref{fig:airfoil_mesh}-right, which consists of 15360 quadrilaterals (124032 degrees of freedom for the velocity and 15648 for the pressure).

For the measurements, we employ the Gaussian functionals \eqref{def:lambda_z} with the location of the centers, indicated with dots on Figure~\ref{fig:airfoil_mesh}-left, given by $\vz_1=(-0.3,0)$, $\vz_2=(2,0)$, $\vz_3=(0.5,0.15)$, $\vz_4=(0.5,-0.1)$, $\vz_{5,6}=(0.7,\pm 0.1)$, $\vz_{7,8}=(0.1,\pm 0.1)$, $\vz_9=(0.3,0.15)$, $\vz_{10}=(0.3,-0.1)$, and $\vz_{11,12}=(2,\pm 0.6)$.
The effect on the number of measurements is tested using $m=2$, $m=6$ and $m=12$ measurements of (each component of) the velocity and of the pressure using $\{\vz_i\}_{i=1}^m$ for the centers.

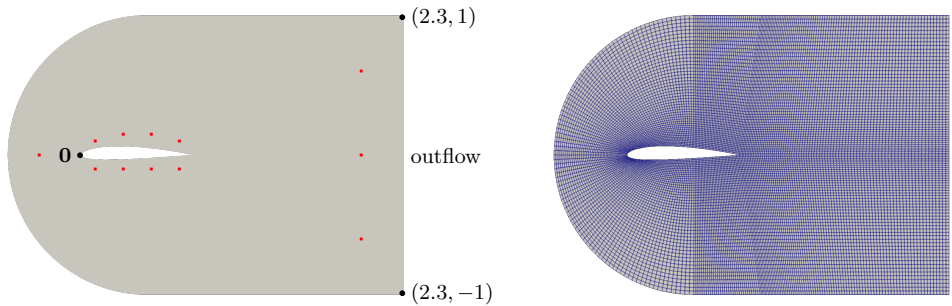
\begin{figure}[htbp]
       \centering
\scalebox{0.9}{
\begin{tikzpicture}
\node[inner sep=0pt] (geometry) at (1.5,0)
    {\includegraphics[width=6.5cm]{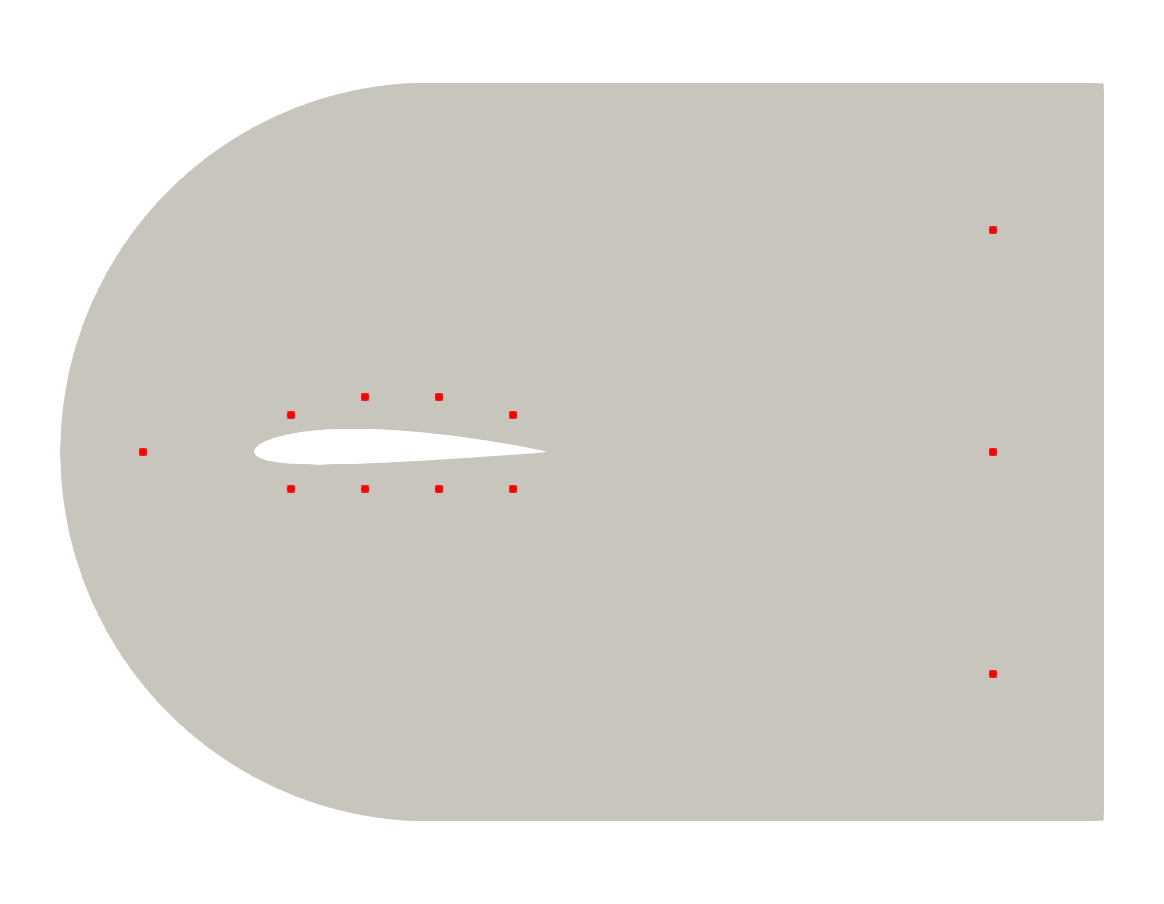}};

\filldraw (4.4,2.04) circle (1pt);
\draw (4.4,2.04) node[right]{{\small $(2.3,1)$}};

\filldraw (4.4,-2.04) circle (1pt);
\draw (4.4,-2.04) node[right]{{\small $(2.3,-1)$}};

\filldraw (-0.36,0) circle (1pt);
\draw (-0.36,0) node[left]{{\small $\mathbf{0}$}};

\draw (4.4,0) node[right]{{\small outflow}};

\end{tikzpicture} \hspace{0.2cm}
\includegraphics[width=6.5cm]{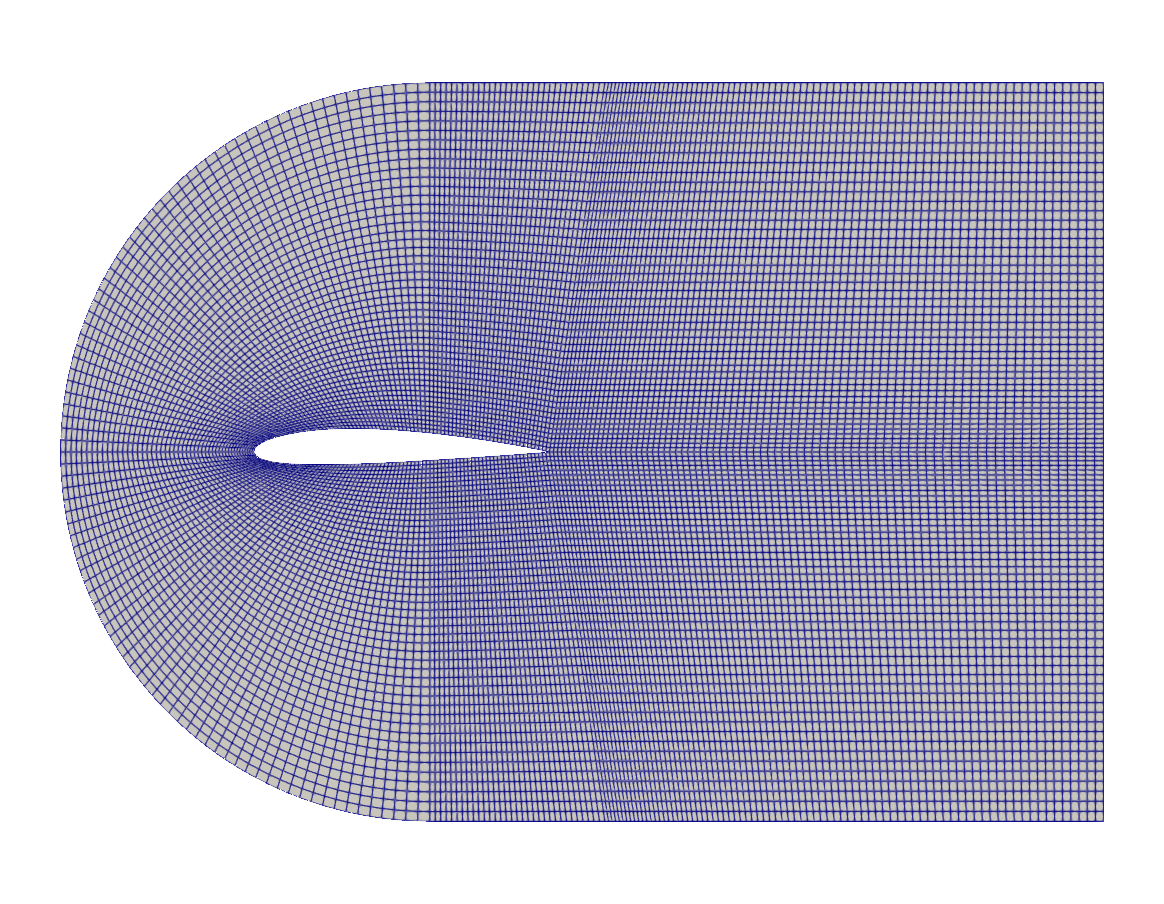}
        }
        \caption{Geometry for the airfoil test case with the location of the measurements (left) and the mesh (right).}
        \label{fig:airfoil_mesh}
        
\end{figure}

The velocity and pressure recovery errors $\text{err}_u$ and $\text{err}_p$ are reported in Table~\ref{tab:airfoil_errors}, while Table~\ref{tab:airfoil_QoI} contains the recovery errors in the drag and lift coefficients $c_D$ and $c_L$. Overall, we observe that adding measurements improves the recovery of the pressure and the velocity as well as the specific quantities of interest $c_D$ and $c_L$. 

\begin{table}[htbp]
\centering
\begin{tabular}{|c|c|c|c|c||c|c|c|c|}
    \cline{2-9}
\multicolumn{1}{c}{}    & \multicolumn{4}{|c||}{$\text{err}_u$} &\multicolumn{4}{|c|}{$\text{err}_p$}\\
    \hline
$m_u\setminus m_p$ & 0 & 2 & 6 & 12 & 0 & 2 & 6 & 12 \\
     \hline       
0 & -- & 1.1650 & 0.8489 & 0.1267 & -- & 3.0705 & 1.8428 & 0.1663 \\
2 & 0.5355 & 0.6736 & 0.6090 & 0.1217 &5.7651 & 1.2652 & 1.1309 & 0.1620\\
6 & 0.0696 & 0.3732 & 0.1628 & 0.0976& 5.6704 & 0.6024 & 0.2425 & 0.1543 \\
12 & 0.0042 & 0.0816 & 0.0726 & 0.0215& 5.6691 & 0.1177 & 0.1055 & 0.0301\\
\hline
\end{tabular}

\caption{Velocity recovery error $\text{err}_u$ and pressure recovery error $\text{err}_p$ for the airfoil test case.}
    \label{tab:airfoil_errors}
\end{table}

\begin{table}[htbp]
\centering
\begin{tabular}{|c|c|c|c|c||c|c|c|c|}
    \cline{2-9}
 \multicolumn{1}{c}{} &  \multicolumn{4}{|c||}{error in the drag coefficient $c_D$} & \multicolumn{4}{|c|}{error in the lift coefficient $c_L$}\\
    \hline
$m_u\setminus m_p$ & 0 & 2 & 6 & 12 & 0 & 2 & 6 & 12 \\
     \hline  
0 & -- & 0.5698 & 0.4462 & 0.0378 & -- & 4.7070$\cdot 10^{-1}$ & 1.0058$\cdot 10^{-1}$ & 3.3610$\cdot 10^{-3}$\\
2 & 1.3036 & 0.3659 & 0.3019 & 0.0389 & 5.8686$\cdot 10^{-2}$ & 5.5617$\cdot 10^{-2}$ & 3.1370$\cdot 10^{-2}$ & 1.3332$\cdot 10^{-3}$\\
6 & 0.1066 & 0.2584 & 0.0453 & 0.0045 & 1.5672$\cdot 10^{-2}$ & 4.4071$\cdot 10^{-2}$ & 8.2003$\cdot 10^{-3}$ & 8.4000$\cdot 10^{-5}$ \\
12 & 0.0004 & 0.0031 & 0.0009 & 0.0004&  7.9000$\cdot 10^{-6}$ & 6.2620$\cdot 10^{-4}$ & 2.0500$\cdot 10^{-4}$ & 1.2530$\cdot 10^{-4}$ \\    
\hline
\end{tabular}
\caption{Recovery error for the drag and lift coefficients $c_D$ and $c_L$. Note that for the reference solution, $c_D=12.4042$ and $c_L=-0.0738222$.}
    \label{tab:airfoil_QoI}
\end{table}

\subsection{Ill-conditioning of the Gram matrix} \label{sec:ill_cond}

As already mentioned above, the condition number of the Gram matrix $\widehat G$ deteriorates as the number of measurements increases.
The impact of this poor conditioning on the recovery error when solving \eqref{eqn:LS_Gram} is studied numerically in this section.
We study the effect of replacing $\widehat G$ by $\widehat G^\epsilon$ obtained by setting to zero the singular values of $\widehat G$ smaller than $\epsilon\cdot\sigma_{\max}$, where $\sigma_{\max}$ is the largest singular value of $\widehat G$.

\subsubsection{Effect of thresholding on the Schur complement tolerance} \label{sec:tol_Schur}

We start by examining the effect of the stopping criterion used for the conjugate gradient algorithm used to solve   \eqref{e:algo_step1} and \eqref{e:algo_step3} as well as compute the finite element approximation of $(\vu_{\vf},p_{\vf})\in H_0^1(\Omega)^2\times L_0^2(\Omega)$ solving \eqref{def:pb1}. We let $s_{tol}^1$ be the stopping criteria for \eqref{def:pb1} and $s_{tol}^2$ the one for \eqref{e:algo_step1} and \eqref{e:algo_step3}.
% and the discrete Stokes finite element discretization of 
We report in Table~\ref{tab:influence_schur_tol} the recovery errors obtained when using $\epsilon=0$ (no thresholding) and $\epsilon=10^{-10}$ for different values of $s_{tol}$. We observe that it is counter-productive to use a small tolerance of the conjugate gradients without thresholding. 
In contrast, when $\epsilon=10^{-10}$, the recovery error for both the velocity and the pressure are robust with respect to the tolerance parameters.

\begin{table}[htbp]
    \centering
    \begin{tabular}{|c|l|c|c||c|c|}
    \cline{3-6}
    \multicolumn{2}{c|}{ } & \multicolumn{2}{|c||}{$\epsilon=0$} & \multicolumn{2}{|c|}{$\epsilon=10^{-10}$} \\
    \hline
    $s_{tol}^1$ & \multicolumn{1}{|c|}{$s_{tol}^2$} & $\text{err}_u$ & $\text{err}_p$ & $\text{err}_u$ & $\text{err}_p$ \\
    \hline
$10^{-6}$ & $10^{-12}$ & 1.09432 & 0.241753 & 1.07426 & 0.125310 \\
$10^{-6}$ & $10^{-9}$ & 1.07428 & 0.125512 & 1.07426 & 0.125309 \\ $10^{-9}$ & $10^{-12}$ & 1.07086 & 0.113884 & 1.07426 & 0.125309 \\
$10^{-9}$ & $10^{-9}$ & 1.07426 & 0.125350 & 1.07426 & 0.125308 \\
\hline
    \end{tabular}   
    \caption{Velocity recovery error $\text{err}_u$ and pressure recovery error $\text{err}_p$ for the case \eqref{num:case2} using $m_u=0$, $m_p=64$ and $n=6$ refinements for the mesh. The tolerances $s_{tol}^1$ and $s_{tol}^2$ are used as stopping criteria for the conjugate gradient algorithm used for the resolution of \eqref{def:pb1} and \eqref{e:algo_step1}-\eqref{e:algo_step3}.}
    \label{tab:influence_schur_tol}
\end{table}

\subsubsection{Effect of thresholding and preconditioning} \label{sec:eps_SVD}

We explore numerically two preconditioning strategies for the resolution of system \eqref{eqn:LS_Gram}. We follow \cite[Section 11.5]{GL2013} and rewrite \eqref{eqn:LS_Gram} as
\begin{equation} \label{eqn:LS_Gram_precond}
    \widehat P^{-\frac{1}{2}}\widehat G \widehat P^{-\frac{1}{2}}\widetilde{\valpha} = \widehat P^{-\frac{1}{2}}\vmeas_{\cS}, \qquad \widehat \valpha = \widehat P^{-\frac{1}{2}}\widetilde\valpha,
\end{equation}
where $\widehat P:=\text{diag}(\widehat G)$ and $\widehat P^{-\frac{1}{2}}$ is the principal square root of $\widehat P^{-1}$.
To solve the preconditioned linear system \eqref{eqn:LS_Gram_precond}, we use the Moore--Penrose inverse obtained via a singular value decomposition of $\widehat G_{\widehat P}:=\widehat P^{-\frac{1}{2}}\widehat G \widehat P^{-\frac{1}{2}}$ with thresholding parameter $\epsilon=10^{-10}$.
This means that all singular values of magnitude smaller than $\epsilon \cdot \sigma_{\max}$ are set to zero, where $\sigma_{\max}$ is the largest singular value of $\widehat P^{-\frac{1}{2}}\widehat G\widehat P^{-\frac{1}{2}}$; see \cite{AH2019} for a detailed analysis, in particular Section 5. The resulting  matrix is denoted $\widehat G_{\widehat P}^\epsilon$. For comparison, observe that the original system without preconditioning corresponds to $\widehat G_I$, where $I$ is the identity matrix, and thus $\widehat G^\epsilon:=\widehat G_I^\epsilon$ will refer to the original matrix with thresholding.  

We start with a numerical study of the influence of the thresholding parameter $\epsilon$ as well as the effect of preconditioner $\widehat P$. Tables~\ref{tab:condG_n2} and \ref{tab:condG_n6} report the recovery errors obtained using $\widehat G$, $\widehat G_{\widehat P}$, or $\widehat G_{\widehat P}^\epsilon$ for two different subdivisions of $\Omega$. The condition number of the various matrices is also provided. We observe that the condition number of the matrices increases as the number of measurements increases. Moreover, there is not a clear difference between the condition number of $\widehat G$ and its preconditioned version $G_P$. However, the condition number of $\widehat G_{\widehat P}^{\epsilon}$ is significantly smaller. In terms of recovery errors, we observe that for a small number of measurements, namely $m_u,m_p\in\{0,4\}$, the recovery error is similar regardless of the matrix system used or the mesh refinement level. In the case $n=2$, the effects of poor conditioning become more pronounced as additional measurements are incorporated, resulting in large recovery errors that are partially alleviated through thresholding. Moreover, for a smaller mesh size ($n=6$), all three linear systems produce similar recovery errors for moderate number of measurements, while the performance is marginally  better without thresholding ($\epsilon=0$) for larger values of $m_u$ and $m_p$ because the error due to the thresholding becomes relevant. 

\begin{table}[htbp]
\centering
\begin{tabular}{|c|c|c|c|c|c|c|c|c|c|c|}
    \cline{3-11}
\multicolumn{2}{c|}{ } & \multicolumn{3}{|c|}{$\widehat G$} & \multicolumn{3}{|c|}{$\widehat G_{\widehat P}$} & \multicolumn{3}{|c|}{$\widehat G_{\widehat P}^{\epsilon}$} \\
    \hline
$m_u$ & $m_p$ & $\text{err}_u$ & $\text{err}_p$ & cond & $\text{err}_u$ & $\text{err}_p$ & cond & $\text{err}_u$ & $\text{err}_p$ & cond \\
\hline
4 & 0 & 1.108 & 0.724 & 1.02$\cdot 10^{03}$ & 1.108 & 0.724 & 1.02$\cdot 10^{03}$ & 1.108 & 0.724 & 1.02$\cdot 10^{03}$ \\
16 & 0 & 0.928 & 0.648 & 2.57$\cdot 10^{08}$ & 0.928 & 0.648 & 2.66$\cdot 10^{08}$ & 0.928 & 0.648 & 2.66$\cdot 10^{08}$ \\
64 & 0 & 0.524 & 0.257 & 1.65$\cdot 10^{16}$ & 0.570 & 0.281 & 1.33$\cdot 10^{16}$ & 0.805 & 0.516 & 7.48$\cdot 10^{09}$ \\
\hline
0 & 4 & 3.200 & 2.075 & 2.54$\cdot 10^{00}$ & 3.200 & 2.075 & 2.54$\cdot 10^{00}$ & 3.200 & 2.075 & 2.54$\cdot 10^{00}$ \\
4 & 4 & 1.108 & 0.725 & 3.98$\cdot 10^{03}$ & 1.108 & 0.725 & 1.19$\cdot 10^{03}$ & 1.108 & 0.725 & 1.19$\cdot 10^{03}$ \\
16 & 4 & 0.928 & 0.650 & 3.14$\cdot 10^{08}$ & 0.928 & 0.650 & 2.70$\cdot 10^{08}$ & 0.928 & 0.650 & 2.70$\cdot 10^{08}$ \\
64 & 4 & 1.354 & 0.327 & 2.28$\cdot 10^{16}$ & 1.310 & 0.302 & 3.29$\cdot 10^{16}$ & 0.805 & 0.528 & 7.50$\cdot 10^{09}$ \\
\hline
0 & 16 & 1.925 & 2.393 & 8.32$\cdot 10^{04}$ & 1.925 & 2.393 & 3.95$\cdot 10^{04}$ & 1.925 & 2.393 & 3.95$\cdot 10^{04}$ \\
4 & 16 & 1.446 & 2.332 & 9.64$\cdot 10^{04}$ & 1.446 & 2.332 & 3.95$\cdot 10^{04}$ & 1.446 & 2.332 & 3.95$\cdot 10^{04}$ \\
16 & 16 & 2.216 & 1.510 & 1.24$\cdot 10^{10}$ & 2.216 & 1.510 & 2.60$\cdot 10^{09}$ & 2.216 & 1.510 & 2.60$\cdot 10^{09}$ \\
64 & 16 & 6.478 & 0.165 & 2.03$\cdot 10^{16}$ & 5.794 & 0.206 & 3.41$\cdot 10^{16}$ & 6.679 & 0.231 & 3.26$\cdot 10^{09}$ \\
\hline
0 & 64 & 9668 & 0.138 & 1.16$\cdot 10^{16}$ & 3138 & 0.152 & 2.17$\cdot 10^{16}$ & 5.045 & 0.152 & 1.47$\cdot 10^{07}$ \\
4 & 64 & 2661 & 0.136 & 1.03$\cdot 10^{16}$ & 8698 & 0.146 & 1.22$\cdot 10^{16}$ & 5.358 & 0.152 & 2.05$\cdot 10^{07}$ \\
16 & 64 & 196132 & 0.138 & 8.43$\cdot 10^{16}$ & 204076 & 0.142 & 1.40$\cdot 10^{16}$ & 10.749 & 0.152 & 9.85$\cdot 10^{09}$ \\
64 & 64 & 29.182 & 0.134 & 3.75$\cdot 10^{16}$ & 7.613 & 0.134 & 5.08$\cdot 10^{16}$ & 2.529 & 0.145 & 3.27$\cdot 10^{09}$ \\
\hline
\end{tabular}
\caption{Condition number and recovery errors for the case \eqref{num:case2} on the square domain $\Omega=(0,1)^2$ using $n=2$ refinements for the mesh, when considering the Gram matrix $\widehat G$ and its preconditioned versions $\widehat G_{\widehat P}$, or $\widehat G_{\widehat P}^{\epsilon}$ (with truncation $\epsilon=10^{-10}$).}
    \label{tab:condG_n2}
\end{table}

\begin{table}[htbp]
\centering
\begin{tabular}{|c|c|c|c|c|c|c|c|c|c|c|}
    \cline{3-11}
\multicolumn{2}{c|}{ } & \multicolumn{3}{|c|}{$\widehat G$} & \multicolumn{3}{|c|}{$\widehat G_{\widehat P}$} & \multicolumn{3}{|c|}{$\widehat G_{\widehat P}^{\epsilon}$} \\
    \hline
$m_u$ & $m_p$ & $\text{err}_u$ & $\text{err}_p$ & cond & $\text{err}_u$ & $\text{err}_p$ & cond & $\text{err}_u$ & $\text{err}_p$ & cond \\
\hline
4 & 0 & 1.101 & 0.709 & 9.08$\cdot 10^{02}$ & 1.101 & 0.709 & 9.08$\cdot 10^{02}$ & 1.101 & 0.709 & 9.08$\cdot 10^{02}$ \\
16 & 0 & 0.681 & 0.540 & 1.59$\cdot 10^{07}$ & 0.681 & 0.540 & 1.64$\cdot 10^{07}$ & 0.681 & 0.540 & 1.64$\cdot 10^{07}$ \\
64 & 0 & 0.163 & 0.136 & 2.64$\cdot 10^{16}$ & 0.163 & 0.137 & 2.84$\cdot 10^{16}$ & 0.241 & 0.187 & 8.44$\cdot 10^{09}$ \\
\hline
0 & 4 & 3.200 & 2.075 & 2.78$\cdot 10^{00}$ & 3.200 & 2.075 & 2.78$\cdot 10^{00}$ & 3.200 & 2.075 & 2.78$\cdot 10^{00}$ \\
4 & 4 & 1.101 & 0.709 & 3.66$\cdot 10^{03}$ & 1.101 & 0.709 & 1.06$\cdot 10^{03}$ & 1.101 & 0.709 & 1.06$\cdot 10^{03}$ \\
16 & 4 & 0.681 & 0.540 & 4.33$\cdot 10^{08}$ & 0.681 & 0.540 & 3.80$\cdot 10^{08}$ & 0.681 & 0.540 & 3.80$\cdot 10^{08}$ \\
64 & 4 & 0.170 & 0.150 & 1.79$\cdot 10^{16}$ & 0.172 & 0.153 & 2.28$\cdot 10^{16}$ & 0.241 & 0.187 & 9.85$\cdot 10^{09}$ \\
\hline
0 & 16 & 1.536 & 0.792 & 4.17$\cdot 10^{05}$ & 1.536 & 0.792 & 1.85$\cdot 10^{05}$ & 1.536 & 0.792 & 1.85$\cdot 10^{05}$ \\
4 & 16 & 0.807 & 0.582 & 4.22$\cdot 10^{05}$ & 0.807 & 0.582 & 1.86$\cdot 10^{05}$ & 0.807 & 0.582 & 1.86$\cdot 10^{05}$ \\
16 & 16 & 0.589 & 0.451 & 9.56$\cdot 10^{09}$ & 0.589 & 0.451 & 2.21$\cdot 10^{09}$ & 0.589 & 0.451 & 2.21$\cdot 10^{09}$ \\
64 & 16 & 0.143 & 0.113 & 2.96$\cdot 10^{16}$ & 0.142 & 0.111 & 5.36$\cdot 10^{16}$ & 0.226 & 0.171 & 8.93$\cdot 10^{09}$ \\
\hline
0 & 64 & 1.074 & 0.125 & 1.43$\cdot 10^{15}$ & 1.074 & 0.125 & 4.03$\cdot 10^{14}$ & 1.074 & 0.125 & 1.16$\cdot 10^{09}$ \\
4 & 64 & 0.597 & 0.106 & 1.44$\cdot 10^{15}$ & 0.597 & 0.106 & 4.02$\cdot 10^{14}$ & 0.597 & 0.106 & 1.16$\cdot 10^{09}$ \\
16 & 64 & 0.252 & 0.152 & 1.44$\cdot 10^{15}$ & 0.252 & 0.152 & 4.11$\cdot 10^{14}$ & 0.292 & 0.190 & 5.22$\cdot 10^{09}$ \\
64 & 64 & 0.184 & 0.174 & 7.98$\cdot 10^{16}$ & 0.166 & 0.153 & 4.07$\cdot 10^{16}$ & 0.239 & 0.188 & 6.80$\cdot 10^{09}$ \\
\hline
\end{tabular}
\caption{Condition number and recovery errors for the case \eqref{num:case2} on the square domain $\Omega=(0,1)^2$ using $n=6$ refinements for the mesh, when considering the Gram matrix $\widehat G$ and its preconditioned versions $\widehat G_{\widehat P}$, or $\widehat G_{\widehat P}^{\epsilon}$ (with truncation $\epsilon=10^{-10}$).}
    \label{tab:condG_n6}
\end{table}

We note that the positive effect of the preconditioner $\widehat P$ could appear minimal. However, this depends on the situation. For example, we have observed (not reported for the sake of brevity) more robust recovery errors when the computational domain is a square domain with a hole (described in Section~\ref{sec:square_hole}).
In summary, the best option appears to be $\widehat G_{\widehat P}^\epsilon$, but further studies are warranted.

\subsection{Three-dimensional problem} \label{sec:3D}

We now illustrate the proposed algorithm for the recovery of functions defined in $\Omega=(0,1)^3 $.
The velocity and pressure to be recovered are given by
\begin{equation} \label{num:case3}
\vu(\vx) = \left(\begin{array}{c}
   \frac{1}{2}\sin(2\pi x_1)\cos(2\pi x_2)\cos(2\pi x_3) \\[0.5ex]
   \frac{1}{2}\cos(2\pi x_1)\sin(2\pi x_2)\cos(2\pi x_3) \\[0.5ex]
   -\cos(2\pi x_1)\cos(2\pi x_2)\sin(2\pi x_3)
\end{array}\right), \quad p(\vx) = \frac{\cos(2\pi x_1)+\cos(2\pi x_2)}{4}\left(\cos(2\pi x_3)+2\right),
\end{equation}
where $\vx=(x_1,x_2,x_3)\in\Omega$. Note that $\|\vu\|_{H^1(\Omega)}=4.732$ and $\|p\|_{L^2(\Omega)}=0.530$. Moreover, the corresponding force $\vf$ is not equal to zero, and thus the solution $(\vu_{\vf},p_{\vf})\in H^1_0(\Omega)^3\times L^2_0(\Omega)$ of \eqref{def:pb1} in the splitting \eqref{def:splitting} is not identically zero. The pressure is mean-free.

We suppose that the boundary data is unknown on the whole $\Gamma$. We set $s=1$ and use $\epsilon=10^{-10}$ for the thresholding parameter; see Section~\ref{sec:ill_cond}. For the measurements, we consider again the Gaussian functional as in \eqref{def:lambda_z} for $v\in\{u_1,u_2,u_3,p\}$ with centers uniformly distributed in $\Omega$, i.e., for an integer $m={\tilde l}^3$ we use
$$\vz_{(i,j,l)}=\left(\frac{i}{\tilde l+1},\frac{j}{\tilde l+1},\frac{l}{\tilde l+1}\right), \quad i,j,l=1,2,\ldots,\tilde l.$$
We take $m_u$ measurements on each component velocity and $m_p$ measurements of the pressure. The computational domain $\Omega$ is subdivided into $2^{12}$ squares of side-length $2^{-4}$, yielding a uniform mesh with mesh size $h=\sqrt{3}\cdot 2^{-4}$, 107811 degrees of freedom for the velocity and 4913 for the pressure. We report in Table~\ref{tab:case3_errors} the velocity and pressure recovery error for different values of $m_u$ and $m_p$.

\begin{table}[htbp]
    \centering
    \begin{tabular}{|c|c|c|c|c||c|c|c|c|c|c|c|c|}
    \cline{2-9}
   \multicolumn{1}{c|}{}&  \multicolumn{4}{c||}{$\textrm{err}_u$} &  \multicolumn{4}{c|}{$\textrm{err}_p$} \\
   \hline
$m_u\setminus m_p$ & 0 & 27  & 125  & 343& 0 & 27  & 125  & 343 \\
     \hline 
0 & -- & 3.2387& 3.2332  & 3.2272 & --  & 0.4696  & 0.3545  & 0.2205\\
8 & 3.4836 & 3.2139  & 3.1990 & 3.1943 & 1.4946 & 0.5350 & 0.3429  & 0.2189 \\
27 & 3.1626 & 2.8387 & 2.0882  & 1.7689 & 2.2798 & 2.4323  & 1.0895  & 0.2208\\
64 & 1.3862  & 1.3622 & 1.1954 & 1.1574   & 2.2798 & 2.4323  & 1.0895  & 0.2208\\
125 & 0.8409 & 0.8386  & 0.8165  & 0.7124 & 0.3869 & 0.4087 & 0.4121 & 0.2420\\
\hline
    \end{tabular}
\caption{Velocity recovery error $\text{err}_u$ and pressure recovery error $\text{err}_p$  for the case three-dimensional case \eqref{num:case3}. We have $\|\vu\|_{H^1(\Omega)}=4.7322$ and $\|p\|_{L^2(\Omega)}=0.5303$.}
\label{tab:case3_errors}
\end{table}

\section{Conclusion}

The main goal of this paper was to extend to the Stokes system the near-optimal recovery algorithm introduced in \cite{BBCDDP2023} for the Poisson problem with missing boundary data. The model class for the recovery solutions is assumed to be a unit ball in a Hilbert space which compactly embeds in $H^1(\Omega)^d \times L^2(\Omega)$ so that the optimal recovery solution belongs to the span of the Riesz representers of the measurement functionals. The key ingredient consists in approximating these velocity-pressure Riesz representers. Two main challenges arise when compared to the Poisson problem. The treatment of the compatibility condition for the trace of admissible velocities as well as the special treatment of the mean value of the pressures. Using finite elements, the computation of approximate Riesz representers reduces to solving iteratively three problems: first a Stokes system for the Lagrange multipliers enforcing the Stokes constraint, then a saddle-point problem for the velocity component on the boundary on the domain with a Lagrange multiplier enforcing the divergence-free compatibility constrain, and finally a second Stokes system for the Riesz representer inside the domain.
We provide an extensive collection of numerical experiments showcasing the performance of the proposed recovery algorithm, including the interplay of velocity and pressure measurements, the influence of the regularity parameter $s$ and the recovery of quantities of interest. Overall, we observe that the more measurements the better the recovery error of both error as well as quantity of interest (here the drag and lift coefficients). Moreover, increasing the value of $s$ not only reduces the velocity and pressure recovery error, but also tend to reduce the oscillations of the pressure field near the boundary where the boundary data are missing. We also discuss two strategies used to deal with the poor conditioning of the linear system $\widehat{G}\widehat{\valpha}=\vw_{\cS}$ as the number of measurements increases. We observe that using a preconditioner does not significantly improve the conditioning of the matrix, whereas singular value decomposition with thresholding appears to be critical to prevent unreasonably large recovery errors.
We note that the methodology can be applied to other PDEs, including nonlinear PDEs when $\mathbf{f}=\mathbf{0}$ such as the incompressible Navier--Stokes equations, provided that the extension operator mapping boundary data to the solution inside the domain is well-defined. That is, provided that Problems \eqref{sys:Riesz_FE} and \eqref{sys:LagMult_FE} with $\mathcal{A}$ replaced by the appropriate form is well-posed.

\appendix

\section{Well-posedness of System \eqref{def:saddle}}
\label{sec:WP_saddle}

Let $Y:=H_0^1(\Omega)^d\times L_0^2(\Omega)\times\mathbb{R}$
equipped with the norm $\|(\vz,\tau,\ell)\|_Y:=\sqrt{\| \nabla \vz\|_{L^2(\Omega)}^2+\|\tau\|_{L^2(\Omega)}^2+|\ell|^2}$.
Using the decomposition $L^2(\Omega)=L_0^2(\Omega)\oplus\mathbb{R}$, the saddle point system \eqref{def:saddle} can be rewritten as: find $(\vw,r_0)\in X_0^s$ and $(\vpi,\rho_0,\bar \rho)\in Y$ such that
\begin{equation} \label{def:saddle_v2}
	\left\{\begin{array}{rcll}
		a(\vw,\vv)+\mathcal{B}((\vv,q),(\vpi,\rho_0,\bar \rho)) & = & \lambda(\vv,q), & \forall\, (\vv,q)\in X_0^s, \\[2ex]
        \mathcal{B}((\vw,r_0),(\vz,\tau,\ell)) & = & 0, & \forall\, (\vz,\tau,\ell)\in Y,
	\end{array}\right.
\end{equation}
where for $(\vv,q)\in X_0^s$ and $(\vz,\tau,\ell)\in Y$
$$\mathcal{B}((\vv,q),(\vz,\tau,\ell)) := \mathcal{A}((\vv,q),(\vz,\tau+\ell)) = k(\vv,\vz)+b(\vv,\tau)+b(\vz,q)+b(\vv,\ell).$$
To prove the existence and uniqueness, we show that $\mathcal{B}$ satisfies an inf-sup condition and that $a$ is coercive on the kernel of $\mathcal{B}$.  

Let $(\vv,q)\in X_0^s$ be such that
$$\mathcal{B}((\vv,q),(\vz,\tau,\ell))=0 \quad \forall\,(\vz,\tau,\ell)\in Y.$$
Restricting \eqref{def:saddle} to $(\vz,\tau)\in H_0^1(\Omega)^d\times L_0^2(\Omega)$ shows that $(\vv,q)\in X_0^s$ satisfies $\mathcal S(\vv,q)={\bf 0}$. Moreover, taking $(\vz,\tau) =(\mathbf{0},1)$ in the second equation in \eqref{def:saddle} yields $b(\vv,1)=0$, and thus $\int_{\Bd}\vv\cdot\vn=0$ thanks to the divergence theorem. 
We thus find that the kernel of $\mathcal B$ is contained in
$$
\mathcal K:=\{ (\vv,q)\in H^1(\Omega)^d \times L^2_0(\Omega) \ : \ \int_\Gamma \vv \cdot \vn = 0 \qquad \textrm{and} \qquad \mathcal S(\vv,q)={\bf 0}\}.
$$
In particular, for $(\vv,q)\in \mathcal K$, \eqref{e:a-priori} yields
$$
\|\nabla\vv\|_{L^2(\Omega)}^2+\|q\|_{L^2(\Omega)}^2 \le C^2 \|\vv\|^2_{H^s(\Gamma)}.
$$
In view of the definition \eqref{def:X0s_norm} of the $X_0^s$-norm defined, we directly deduce 
$a(\vv,\vv)=\|\vv\|^2_{H^s(\Gamma)} \ge \frac{1}{C^2+1}\|\vv,q\|_{X_0^s}^2$,
which proves the coercivity of $a$ on $\mathcal K$ and thus on the kernel of $a$.

For the the inf-sup condition, we show that
$$\inf_{(\vz,\tau,\ell)\in Y}\,\sup_{(\vv,q)\in X_0^s}\frac{\mathcal{B}((\vv,q),(\vz,\tau,\ell))}{\|(\vv,q)\|_{X_0^s}\|(\vz,\tau,\ell)\|_{Y}}>0.$$
To do so, for any $(\vz,\tau,\ell)\in Y$, we construct $(\vv,q)\in X_0^s$ such that
\begin{equation} \label{eqn:C1C2}
    \mathcal{B}((\vv,q),(\vz,\tau,\ell))\ge C_1\|(\vz,\tau,\ell)\|_Y^2 \qquad \text{and} \qquad \|(\vv,q)\|_{X_0^s}\le C_2\|(\vz,\tau,\ell)\|_Y
\end{equation}
for some positive constants $C_1$ and $C_2$. 

The bilinear form $b$ satisfies the inf-sup condition (see for instance \cite{GR86}) on $H_0^1(\Omega)^d\times L_0^2(\Omega)$ and in particular, there is a constant $\beta>0$ such that for all $\tau\in L^2_0(\Omega)$, there is $\vv_\tau \in H^1_0(\Omega)$ such that
$$
b(\vv_\tau,\tau) \ge \beta \|\nabla\vv_\tau\|_{L^2(\Omega)} \|\tau\|_{L^2(\Omega)}.
$$
Upon rescaling $\vv_\tau$, we can assume that $\| \nabla \vv_\tau \|_{L^2(\Omega)} = \eta \| \tau \|_{L^2(\Omega)}$ for some $\eta>0$ to be determined. Hence, we have
$$
b(\vv_\tau,\tau) \geq \beta  \eta \| \tau \|_{L^2(\Omega)}^2.
$$
The last ingredient is $\vv_\ell(\vx):=-\frac{\ell}{d|\Omega|}\vx$ for $\vx\in\overline{\Omega}$ so that $b(\vv_\ell,\ell)=\ell^2$ with $\vv_\ell$ satisfying
$$\|\vv_\ell \|_{H^s(\Gamma)}=C(\Omega,s)|\ell| \quad \text{and} \quad \|\nabla\vv_\ell \|_{L^2(\Omega)}=(d|\Omega|)^{\frac{1}{2}}|\ell|$$
for some positive constant $C(\Omega,s)$ depending only on $\Omega$ and $s$.

We now set $\vv:=\vz+\vv_\tau+\vv_\ell$ and $q:=-\tau$ to get
$$b(\vv,\tau)+b(\vz,q)+b(\vv,\ell) \ge \beta\eta\|\tau\|_{L^2(\Omega)}^2+\ell^2.$$
Using Cauchy-Schwarz and Young's inequalities, we infer that for any $\delta,\alpha>0$ we have
\begin{equation*}
    \begin{split}
k(\vv,\vz)&\ge \left(2-\frac{1}{\delta}-\frac{1}{\alpha}\right)\|\veps(\vz)\|_{L^2(\Omega)}^2-\alpha \|\veps(\vv_\tau)\|_{L^2(\Omega)}^2-\delta d^2=\left(2-\frac{1}{\delta}-\frac{1}{\alpha}\right)\|\veps(\vz)\|_{L^2(\Omega)}^2-\alpha \|\tau\|_{L^2(\Omega)}^2-\delta d^2.
    \end{split}
\end{equation*}
Combining the last two relations yields
$$\mathcal{B}((\vv,q),(\vz,\tau,\ell))\ge \left(2-\frac{1}{\delta}-\frac{1}{\alpha}\right)\|\veps(\vz)\|_{L^2(\Omega)}^2+(\beta\eta-\alpha\eta^2)\|\tau\|_{L^2(\Omega)}^2+(1-\delta)\ell^2.$$
Choosing for instance, $\delta=3/4$, $\alpha=2$ and $\eta=\beta/4$, and invoking the Korn's inequality
$$
\| \nabla \vv \|_{L^2(\Omega)} \leq C_K \| \veps(\vv)\|\qquad \forall\, \vv \in H^1_0(\Omega)^d,
$$
where $C_K$ only depends on $\Omega$, proves the first relation in \eqref{eqn:C1C2}.
It remains to show the second relation in \eqref{eqn:C1C2}. We estimate
\begin{align*}
    \|(\vv,q)\|_{X_0^s}^2 &=\|\nabla(\vz+\vv_\tau+\vv_\ell)\|_{L^2(\Omega)}^2+\|\tau\|_{L^2(\Omega)}^2+\|\vv_\ell\|_{H^s(\Omega)}^2 \\
    & \le 3\left(\|\nabla\vz\|_{L^2(\Omega)}^2+\|\nabla\vv_\tau\|_{L^2(\Omega)}^2+\|\nabla \vv_\ell\|_{L^2(\Omega)}^2\right)+\|\tau\|_{L^2(\Omega)}^2+C(\Omega,s)^2\ell^2
\end{align*}
and the second estimate in \eqref{eqn:C1C2} follows from the scaling $\| \nabla \vv_\tau \|_{L^2(\Omega)}=\eta \| \tau \|_{L^2(\Omega)}$ and $\|\nabla \vv_\ell\|_{L^2_0(\Omega)} = (d|\Omega|)^{\frac 1 2}\ell$.

\bibliographystyle{abbrv} % {amsplain} % {amsalpha} % {spmpsci}
\bibliography{biblio}

\end{document}